\documentclass[bj]{imsart}
\usepackage{amsthm,amsmath,amssymb,amsfonts,delarray}
\usepackage[dvipdfmx]{graphicx}

\newtheorem{theorem}{Theorem}[section]

\newtheorem{lemma}{Lemma}[section]
\newtheorem{proposition}{Proposition}[section]

\newtheorem{condition}{Condition}[section]
\theoremstyle{definition}
\newtheorem{definition}{Definition}[section]
\newtheorem{remark}{Remark}[section]

\newcommand{\R}{\mathbb{R}}

\newcommand{\Ep}{\mathbb{E}}
\renewcommand{\Pr}{\mathbb{P}}
\renewcommand{\tilde}{\widetilde}
\renewcommand{\hat}{\widehat}
\newcommand{\indep}{\mathop{\perp\!\!\!\!\perp}}

\newcommand{\uplambda}{\overline{\lambda}}
\newcommand{\downlambda}{\underline{\lambda}}

\newcommand{\upw}{\overline{w}}
\newcommand{\downw}{\underline{w}}

\renewcommand{\leq}{\leqslant}
\renewcommand{\geq}{\geqslant}

\DeclareMathOperator{\op}{op}

\DeclareMathOperator{\rank}{rank}

\begin{document}

\begin{frontmatter}
\title{On frequentist coverage errors \\ of Bayesian credible sets \\ in moderately high dimensions}
\runtitle{Bayesian credible sets in moderately high dimensions}

\begin{aug}
\author{\fnms{Keisuke}  \snm{Yano}\thanksref{a,e1}\ead[label=e1,mark]{yano@mist.i.u-tokyo.ac.jp}},
\and
\author{\fnms{Kengo} \snm{Kato}\thanksref{b,e2}\ead[label=e2,mark]{kk976@cornell.edu}}
\runauthor{K. Yano and K. Kato}
\affiliation{University of Tokyo and Cornell University}
\address[a]{Department of Mathematical Informatics, Graduate School of Information Science and Technology, University of Tokyo, 7-3-1 Hongo, Bunkyo-ku, Tokyo 113-0033, Japan. \printead{e1}}
\address[b]{Department of Statistics and Data Science, Cornell University, 1194 Comstock Hall, Ithaca, NY 14853.\printead{e2}}
\end{aug}



\begin{abstract}
In this paper, we study frequentist coverage errors of Bayesian credible sets for an approximately linear regression  model with (moderately) high dimensional regressors, where the dimension of the regressors may increase with but is smaller than the sample size.
Specifically, we consider quasi-Bayesian inference on the slope vector under the quasi-likelihood with Gaussian error distribution.
Under this setup, we derive finite sample bounds on frequentist coverage errors of  Bayesian credible rectangles. Derivation of those bounds builds on a novel Berry--Esseen type bound on quasi-posterior distributions and recent results on high-dimensional CLT on hyperrectangles. We use this general result to quantify coverage errors of Castillo--Nickl and $L^{\infty}$-credible bands for Gaussian white noise models, linear inverse problems, and (possibly non-Gaussian) nonparametric regression models.
In particular, we show that Bayesian credible bands for those nonparametric models have coverage errors decaying polynomially fast in the sample size, implying advantages of Bayesian credible bands over confidence bands based on extreme value theory.
\end{abstract}


\end{frontmatter}

\maketitle

\section{Introduction}
\label{section: introduction}

Bayesian inference for high or nonparametric statistical models is an active research area in the recent  statistics literature. Posterior distributions provide not only point estimates but also credible sets.
In a classical regular statistical model with a fixed finite dimensional parameter space, it is well known that the Bernstein--von Mises (BvM) theorem holds under mild conditions and  the posterior distribution can be approximated (under the total variation distance) by a normal distribution centered at an efficient estimator (e.g. MLE) and with
covariance matrix identical to the inverse of the Fisher information matrix as the sample size increases. 
The BvM theorem implies that a Bayesian credible set is typically a valid confidence set in the frequentist sense, namely, the coverage probability of a $(1-\alpha)$-Bayesian credible set evaluated under the true parameter value  is approaching $(1-\alpha)$ as the sample size increases; cf. \cite{vanderVaart_book}, Chapter 10. 
There is also a large literature on the BvM theorem in  nonparametric statistical models.
Compared to the finite dimensional case, however, Bayesian uncertainty quantification is more complicated and more sensitive to prior choices in the infinite dimensional case.
\cite{Cox(1993), Freedman(1999)} find some negative results on the BvM theorem in the infinite dimensional case.
\cite{Bontemps(2011), Johnstone(2010), Leahu(2011)} develop conditions 
under which the BvM theorem holds for Gaussian white noise models and nonparametric regression models; see also \cite{ClarkeandGhosal(2010),Ghosal(1999),Spokoiny(2014)}.
Employing weaker topologies than $L^{2}$,
\cite{CastilloandNickl(2013)} elegantly formulate and establish the BvM theorem for Gaussian white noise models; see also \cite{Ray(2017)} for the adaptive BvM theorem for Gaussian white noise models.
Subsequently, \cite{CastilloandNickl(2014)} establish the BvM theorem in a weighted  $L^{\infty}$-type norm
for nonparametric regression and density estimation.
There are also several papers on frequentist coverage errors of Bayesian credible sets in the $L^{2}$-norm. 
\cite{KnapikvanderVaartvanZanten(2011)} study asymptotic frequentist coverage errors of $L^{2}$-type Bayesian credible sets based on Gaussian priors for linear inverse problems; see also \cite{SniekersandvanderVaart(2015a),SzabovanderVaartvanZanten(2015b)} for related results.
Using an empirical Bayes approach,
 \cite{SzabovanderVaartvanZanten(2015)} develop  $L^{2}$-type Bayesian credible sets adaptive to unknown smoothness of the function of interest. 
We refer the reader to Chapter 7 in \cite{GineandNickl_book} and
Chapter 12 in \cite{GhosalandvanderVaart(2017)} 
for further references on these topics.

This paper aims at studying frequentist coverage errors of Bayesian credible rectangles
in an approximately linear regression model with an increasing number of regressors.
We provide finite sample bounds on frequentist coverage errors of (quasi-)Bayesian credible rectangles based on sieve priors,
where 
the model allows both an unknown bias term and an unknown error variance,
and
the true distribution of the error term may not be Gaussian.
Sieve priors are distributions on the slope vector whose dimension increases with the sample size. 
We allow sieve priors to be non-Gaussian or not to be an independent product.
We employ a ``quasi-Bayesian" approach with Gaussian error distributions.
The resulting posterior distribution is called a ``quasi-posterior."

An important application of our results is finite sample quantification of Bayesian nonparametric credible bands based on sieve priors.
We derive finite sample bounds on coverage errors of Castillo--Nickl \cite{CastilloandNickl(2014)} and $L^{\infty}$-credible bands
in Gaussian white noise models, linear inverse problems, and (possibly non-Gaussian) nonparametric regression models; see Section \ref{section: Castillo--Nickl credible bands in Gaussian white noise model} ahead for the definition of Castillo--Nickl credible bands.
The literature on frequentist confidence bands is broad.
Frequentist approaches to constructing confidence bands date back to Smirnov and Bickel--Rosenblatt \cite{Smirnov(1950),BickelandRosenblatt(1973)}; see also \cite{CCK(2014b),ClaeskensandVanKeilegom(2003),GineandNickl(2010)} for more recent results. 
In contrast, there are relatively limited results on Bayesian uncertainty quantification based on $L^{\infty}$-type norms. \cite{GineandNickl(2011)} study posterior contraction rates in the $L^{r}$-norm for $1\leq r\leq \infty$, and \cite{Castillo(2014)} derive sharp posterior contraction rates in the $L^{\infty}$-norm.
\cite{HoffmannRousseauSchmidt-Hieber(2015)} derive adaptive posterior contraction rates in the $L^{\infty}$-norm for Gaussian white noise models and density estimation; see also \cite{YooRousseauRivoirard(2017)} for adaptive posterior contraction rates.
Building on their new BvM theorem, \cite{CastilloandNickl(2014)} develop credible bands (Castillo-Nickl bands) based on product priors that have correct frequentist coverage probabilities and at the same time shrink at (nearly) minimax optimal rates for Gaussian white noise models. 
\cite{YooandGhosal(2016)} study conditions under which 
frequentist coverage probabilities of credible bands based on Gaussian series priors approach one as the sample size increases for  nonparametric regression models with sub-Gaussian errors.
\cite{Ray(2017)} establish qualitative results on adaptive credible bands for Gaussian white noise models.
Still, \textit{quantitative results} on frequentist coverage errors of nonparametric credible bands are scarce.
Our quantitative result complements the qualitative results established
by  \cite{CastilloandNickl(2014)} and \cite{YooandGhosal(2016)}  and contributes to the literature on Bayesian nonparametrics by developing deeper understanding on Bayesian uncertainty quantification in nonparametric models. 
More recently, \cite{YangBhattacharyaPati(2017)} also derive a quantitative result on coverage errors of Bayesian credible bands based on Gaussian process priors.
We will clarify the difference between their results and ours in Section \ref{subsection: literature} ahead.

Notably, our results lead to an implication that supports the use of Bayesian approaches to constructing nonparametric confidence bands.
It is well known that confidence bands based on extreme value theory (such as e.g. those of \cite{BickelandRosenblatt(1973)})
perform poorly because of  the slow convergence of Gaussian maxima.
In the kernel density estimation case, \cite{Hall(1991)} shows that
confidence bands based on extreme value theory have coverage errors decaying only at the $1/\log n$ rate (regardless of how we choose bandwidths) where $n$ is the sample size, while those based on bootstrap have coverage errors (for the surrogate function) decaying polynomially fast in the sample size; see also \cite{CCK(2014b)}.
Our result shows that Bayesian credible bands (for the true function in Gaussian white noise models and linear inverse problems; for the surrogate function in nonparametric regression models) have also coverage errors decaying polynomially fast in the sample size,
implying an advantage of Bayesian credible bands over confidence bands based on extreme value theory; see Remarks 
\ref{remark: comparison of coverage errors} and \ref{remark: coverage errors for the surrogate function} for more details.
Another potentially interesting implication of our analysis of the Castillo-Nickl band is the following. 
In this paper, we use a sieve prior that truncates high frequency terms of the function. 
In a Gaussian white noise model,
our results show that the coverage error for the true function of the Castillo-Nickl band decays fast in the sample size (i.e., decays at a polynomial rate in the sample size), and at the same time the $L^{\infty}$-diameter converges at a minimax optimal rate
as long as the cut-off level $2^{J}$ is chosen in such a way that $2^{J} \sim (n/\log n )^{1/(2s+1)}$
where $s$ is the smoothness level. 
This implies that, as long as we confine ourselves to nonadaptive credible bands, a sieve prior would not be less favorable than a prior that models high-frequency terms of the function.

The main ingredients in the derivation of the coverage error bound in Section \ref{section: frequentist validation of Bayesian credible rectangles} are
(i) a novel Berry--Esseen type bound for the BvM theorem for sieve priors,
i.e., 
a finite sample bound on the total variation distance between
the quasi-posterior distribution based on a sieve prior and the corresponding Gaussian distribution,
and
(ii) recent results on high dimensional CLT on hyperrectangles \cite{CCK(2013),CCK(2016)}.
Our Berry--Esseen type bound improves upon existing BvM-type results for sieve priors; see the discussion in Section \ref{subsection: literature}.
The high dimensional CLT  is used to approximate  the sampling distribution of the centering estimator by the Gaussian distribution that matches with the Gaussian distribution approximating the (normalized) posterior distribution.

In addition, importantly, derivations of coverage error bounds for nonparametric models in Section \ref{sec: applications} are by no means trivial and require further technical arguments. 
Specifically, for Gaussian white noise models, we will consider
both credible bands based on centering estimators with fixed cut-off dimensions and without cut-off dimensions, 
which require different analyses on bounding the effect of the bias to the coverage error.
For linear inverse problems, we will cover both mildly and severely ill-posed cases. 
For nonparametric regression models, we will consider random designs and so can not directly apply the  result of Section \ref{section: frequentist validation of Bayesian credible rectangles} since we assume fixed designs in Section \ref{section: frequentist validation of Bayesian credible rectangles}; hence we have to take care of the randomness of the design, and to this end, we will employ some empirical process techniques.

\subsection{Literature review and contributions}
\label{subsection: literature}

For a nonparametric regression model,
\cite{YangBhattacharyaPati(2017)} derive finite sample bounds on
frequentist coverage errors of Bayesian credible bands based on Gaussian process priors.
They assume
(i) Gaussian process priors,
(ii) that the error term follows a sub-Gaussian distribution,
and
(iii)  that the error variance is known.
The present paper markedly differs from \cite{YangBhattacharyaPati(2017)} 
in that (i) we work with possibly non-Gaussian priors; (ii) we allow a more flexible error distribution; and (iii) we allow the error variance to be unknown.
More specifically,  (i) 
to allow for non-Gaussian priors,
we develop novel Berry--Esseen type bounds on quasi-posterior distributions in (mildly) high dimensions.
(ii) In addition, to weaken the dimensionality restriction and the moment assumption on the error distribution, 
we make use of high-dimensional CLT on hyperrectangles developed in \cite{CCK(2013),CCK(2016)}.
(iii) Finally, when the error variance is unknown, 
the quasi-posterior contraction for the error variance impacts on the coverage error for the slope vector and so a careful analysis is required to take care of the unknown variance.

The present paper also contributes to the literature on the BvM theorem in nonparametric statistics, which is now quite broad; see
\cite{CastilloandNickl(2013),CastilloandNickl(2014),Freedman(1999),Johnstone(2010),Leahu(2011),Ray(2017)}
for Gaussian white noise models, \cite{Bontemps(2011),Ghosal(1999)} for  linear regression models with high dimensional regressors, and 
\cite{YangBhattacharyaPati(2017),YooandGhosal(2016)} for nonparametric regression models with Gaussian process priors.
See \cite{CastilloSchmidt-HiebervanderVaart(2015)} for high-dimensional linear regression under sparsity constraints.
Note that \cite{CastilloSchmidt-HiebervanderVaart(2015)} also discusses non-Gaussian error distributions.
See also \cite{BoucheronandGassiat(2009),CastilloandRousseau(2015),GaoandZhou(2016),Ghosal(2000),Nickl(2018),NicklandSohl(2017),RivoirardandRousseau(2012)} for related results.
We refer the reader to 
\cite{Atchade(2017),ChernozhukovandHong(2003),FlorensandSimoni(2012),Kato(2013)}
on the BvM theorem for quasi-posterior distributions.

Importantly, our Berry--Esseen type bound improves on conditions on the critical dimension for the BvM theorem.
\cite{Bontemps(2011),Ghosal(1999),Spokoiny(2014)} study such critical dimensions for sieve priors.
First, \cite{Bontemps(2011)} does not cover the case with an unknown error variance, while
the results in \cite{Ghosal(1999),Spokoiny(2014)} cover the case with an unknown error variance.
Our result is consistent with the result of \cite{Bontemps(2011)} when the error variance is assumed to be known.
Meanwhile, our result substantially improves on the results of \cite{Ghosal(1999),Spokoiny(2014)} for the unknown error variance case. 
Namely, the results of \cite{Ghosal(1999),Spokoiny(2014)} show that the BvM theorem holds if $p^{3}=o(n)$ under typical situations when the error variance is unknown, where $p$ is the number of regressors and $n$ is the sample size;
on the other hand, our result shows that the BvM theorem holds if $p^{2}(\log n)^{3}=o(n)$, thereby improving on the condition of \cite{Ghosal(1999),Spokoiny(2014)}. 
See Remark \ref{rem: critical dimension} for more details.
Our BvM-type result allows us to cover wider smoothness classes of functions 
when applied to the analysis of Bayesian credible bands in nonparametric models.

\subsection{Organization and notation}
The rest of the paper is organized as follows. In Section \ref{section: frequentist validation of Bayesian credible rectangles}, we consider Bayesian credible rectangles for the slope vector in an approximately linear regression model and derive finite sample bounds on frequentist coverage errors of the credible rectangles. In Section \ref{sec: applications}, we discuss applications of the general result established in Section \ref{section: frequentist validation of Bayesian credible rectangles} to nonparametric models. 
Specifically, we cover Gaussian white noise models, linear inverse models, and nonparametric regression models with possibly non-Gaussian errors.
In Section \ref{section: proof of frequentist validation of Bayesian credible rectangles},
we give a proof of  the main theorem (Theorem \ref{theorem: frequentist validation of Bayesian credible rectangles}).
Proofs of the other results are given in \cite{supplement}.

Throughout the paper, we will obey the following notation. 
Let $\| \cdot \|$ denote the Euclidean norm, and let $\| \cdot \|_{\infty}$ denote the max or supremum norm for vectors or functions.
Let $\mathcal{N}(\mu,\Sigma)$ denote the Gaussian distribution with mean vector $\mu$ and covariance matrix $\Sigma$. 
For  $x \in \R$, let $x_{+} = \max \{ x, 0 \}$.
For two sequences $\{a_{n}\}$ and $\{b_{n}\}$ depending on $n$,
we use the notation $a_{n}\lesssim b_{n}$ if  $a_{n}\leq c b_{n}$ for some universal constant $c>0$, and 
 $a_{n}\sim b_{n}$ if $a_{n}\lesssim b_{n}$ and $b_{n}\lesssim a_{n}$.
For any symmetric positive semidefinite matrices $A$ and $B$,
the notation $A \preceq B$ means that $B-A$ is positive semidefinite.
Constants $c_{1},c_{2},\ldots$, $c$, and $\tilde{c}_{1},\tilde{c}_{2},\ldots$
do not depend on the sample size $n$ and the dimension $p$.
The values of $c, c_{1},c_{2},\ldots$ and $\tilde{c}_{1},\tilde{c}_{2},\ldots$
may be different at each appearance.

\section{Bayesian credible rectangles}\label{section: frequentist validation of Bayesian credible rectangles}

Consider an approximately linear regression model 
\begin{equation}
Y = X\beta_{0} + r + \varepsilon, \label{eq: linear model}
\end{equation}
where $Y=(Y_{1},\dots,Y_{n})^{\top} \in \R^{n}$ is a vector of outcome variables, 
$X$ is an $n \times p$ design matrix,
$\beta_{0} \in \R^{p}$ is an unknown coefficient vector,
$r = (r_{1},\dots,r_{n})^{\top} \in \R^{n} $ is a deterministic (i.e., non-random) bias term,
and $\varepsilon = (\varepsilon_{1},\dots,\varepsilon_{n})^{\top} \in \R^{n}$ is a vector of i.i.d.~error terms 
with mean zero and variance $0 < \sigma_{0}^{2} < \infty$. 
We are primarily interested in the situation where the number of regressors $p$ increases with the sample size $n$, 
i.e., $p = p_{n} \to \infty$ as $n \to \infty$, but we often suppress the dependence on $n$ for the sake of notational simplicity. 
In addition, we allow the error variance $\sigma_{0}^{2}$ to depend on $n$, i.e., $\sigma_{0}^{2} = \sigma_{0,n}^{2}$, which allows us to include Gaussian white noise models in the subsequent analysis as a special case. In the general setting, the error variance $\sigma_{0}^{2}$ is also unknown.
In the present paper, we work with the dense model with moderately high-dimensional regressors where $\beta_{0}$ need not be sparse and $p=p_{n}$ may increase with the sample size $n$ but $p \le n$. To be precise, we will maintain the assumption that the design matrix $X$ is of full column rank, i.e., $\rank X = p$. The approximately linear model (\ref{eq: linear model}) is flexible enough to cover various nonparametric models such as Gaussian white noise models, linear inverse problems, and nonparametric regression models, via series expansions of functions of interest in those nonparametric models; see Section \ref{sec: applications}. 

We consider Bayesian inference on the slope vector $\beta_{0}$. To this end, we work under the quasi-likelihood with a Gaussian distribution on the error $\varepsilon$.
Namely, we work with the \textit{quasi}-likelihood of the form 
\[
(\beta,\sigma^{2}) \mapsto (2\pi \sigma^{2})^{-n/2} \mathrm{e}^{-\|Y-X\beta\|^{2}/(2\sigma^{2})}.
\]
We assume independent priors on $\beta$ and $\sigma^{2}$, i.e., 
\begin{equation}
\beta \sim \Pi_{\beta}, \ \sigma^{2} \sim \Pi_{\sigma^{2}}, \ \beta \indep \sigma^{2},
\label{eq: prior}
\end{equation}
where we assume that $\Pi_{\beta}$ is absolutely continuous with density $\pi$, i.e., $\Pi_{\beta}(d\beta) = \pi (\beta) d\beta$, and $\Pi_{\sigma^{2}}$ is supported in $(0,\infty)$. Then the resulting quasi-posterior distribution for $(\beta,\sigma^{2})$ is 
\[
\Pi ( d(\beta,\sigma^{2}) \mid Y) \propto (2\pi \sigma^{2})^{-n/2} \mathrm{e}^{-\|Y-X\beta\|^{2}/(2\sigma^{2})} \pi (\beta)d\beta  \Pi_{\sigma^{2}} (d\sigma^{2}),
\]
and the marginal quasi-posterior distribution for $\beta$ is $\Pi_{\beta}( d\beta \mid Y) = \pi (\beta \mid Y) d\beta$, where 
\[
\pi (\beta \mid Y) = \pi (\beta) \int \frac{\mathrm{e}^{-\|Y-X\beta\|^{2}/(2\sigma^{2})}}{\int \mathrm{e}^{-\|Y-X\tilde{\beta} \|^{2}/(2\sigma^{2})} \pi (\tilde{\beta}) d\tilde{\beta}} \Pi_{\sigma^{2}}(d \sigma^{2} \mid Y).
\]
Here $\Pi_{\sigma^{2}}(d \sigma^{2} \mid Y)$ denotes the marginal quasi-posterior distribution for $\sigma^{2}$:
\[
\Pi_{\sigma^{2}}(d \sigma^{2} \mid Y)= \frac{ \int (2\pi \sigma^{2})^{-n/2} \mathrm{e}^{-\|Y-X\beta\|^{2}/(2\sigma^{2})} \pi (\beta) d\beta \Pi_{\sigma^{2}} (d\sigma^{2})}{ \int \int (2\pi \tilde{\sigma}^{2})^{-n/2} \mathrm{e}^{-\|Y-X\beta\|^{2}/(2\tilde{\sigma}^{2})} \pi (\beta) d\beta \Pi_{\sigma^{2}} (d\tilde{\sigma}^{2}) }.
\]
We will assume that $\Pi_{\sigma^{2}}$ may be data-dependent, e.g., $\Pi_{\sigma^{2}} = \delta_{\hat{\sigma}^{2}}$ for some estimator $\hat{\sigma}^{2}$ of $\sigma^{2}$ (in that case, $\Pi_{\sigma^{2}} (\cdot \mid Y) = \delta_{\hat{\sigma}^{2}}$),  but $\Pi_{\beta}$ is data-independent. 

We will derive finite sample bounds on frequentist coverage errors of Bayesian credible rectangles for the approximately linear model (\ref{eq: linear model}) under a prior of the form (\ref{eq: prior}). 
For a vector $c = (c_{1},\dots,c_{p})^{\top} \in{\R}^{p}$, a positive number $R>0$, and a positive sequence $\{w_{j}\}_{j=1}^{p}$,
let $I(c,R)$ denote the hyperrectangle of the form 
\[
I(c,R):= \left \{ \beta = (\beta_{1},\dots,\beta_{p})^{\top} \in \R^{p} 
: \frac{| \beta_{j} - c_{j} |}{w_{j}} \leq R, \ 1 \leq \forall j \leq p  \right \}.
\]
Let $\hat{\beta}$ denote the OLS estimator for $\beta_{0}$ with $r=0$, i.e., $\hat{\beta} = \hat{\beta}(Y) = (X^{\top}X)^{-1}X^{\top} Y$.
For given $\alpha\in(0,1)$, we consider a $(1-\alpha)$-credible rectangle of the form $I(\hat{\beta},\hat{R}_{\alpha})$,
where the radius $\hat{R}_{\alpha}$ is chosen in such a way 
that the posterior probability of the set $I(\hat{\beta},\hat{R}_{\alpha})$
is $1-\alpha$, i.e., $\Pi_{\beta}\{ I(\hat{\beta}, \widehat{R}_{\alpha} ) \mid Y \} =1-\alpha$.

We assume the following conditions on the priors $\Pi_{\beta}$ and $\Pi_{\sigma^{2}}$. 
For $R > 0$, let 
\begin{align}
B(R):=\{\beta \in \R^{p}: \|X(\beta-\beta_{0})\|\leq R\sigma_{0} \} 
\quad \text{and} \quad
\phi_{\Pi_{\beta}}(R) := 1 - \mathop{\inf}_{\beta,\tilde{\beta}\in B(R)} \left\{ \frac{\pi(\tilde{\beta})}{\pi(\beta)}\right\}, \label{def:phi1}
\end{align}
where $\phi_{\Pi_{\beta}}$ quantifies ``lack of flatness'' of the prior density $\pi (\beta)$ around the true value $\beta_{0}$.

\begin{condition}\label{Condition: prior mass condition}
There exists a positive constant $C_{1}$ such that
\[\pi (\beta_{0})  \geq \sigma_{0}^{-p} \sqrt{\det(X^{\top}X)} n^{ - C_{1}p }.\]
\end{condition}

\begin{condition}\label{Condition: marginal posterior contraction}
There exist nonnegative constants $\delta_{1}, \delta_{2},\delta_{3} \in [0,1)$ such that with probability at least $1-\delta_{3}$, 
$\Pi_{\sigma^{2}} \left( \left\{ \sigma^{2} : \left | \sigma^{2} / \sigma_{0}^{2} - 1 \right | > \delta_{1} \right\} \mid Y \right) 
\leq \delta_{2}.$
\end{condition}

\begin{condition}\label{Condition: for simplicity}
The inequality $\phi_{\Pi_{\beta}}(1/\sqrt{n}) \leq 1/2$ holds.
\end{condition}
Condition \ref{Condition: prior mass condition} assumes that the prior $\Pi_{\beta}$ on $\beta$ has a sufficient mass around its true value $\beta_{0}$.
Condition \ref{Condition: marginal posterior contraction} is an assumption on the marginal posterior contraction for the error variance $\sigma^{2}$. Condition \ref{Condition: marginal posterior contraction} includes the known error variance case as a special case; if the error variance is known, then we may take $\Pi_{\sigma^{2}} = \delta_{\sigma_{0}^{2}}$ (Dirac delta at $\sigma_{0}^{2}$) and $\delta_1 = \delta_2 = \delta_3=0$. 
Condition \ref{Condition: for simplicity} is a preliminary flatness condition on $\Pi_{\beta}$.
More detailed discussions on these conditions are provided after the main theorem (Theorem \ref{theorem: frequentist validation of Bayesian credible rectangles}).

We also assume the following conditions on the model.

\begin{condition}\label{Condition: residual condition}
There exists a positive constant $C_{2}$ such that
$\|X(X^{\top}X)^{-1}X^{\top}r\| \leq C_{2} \sigma_{0}\sqrt{p\log n}.$
\end{condition}

\begin{condition}\label{Condition: moment condition}
There exists a positive constant $C_{3}$ such that one of the following conditions holds: 
\begin{enumerate}
\item[(a)] $\Ep[|\varepsilon_{1}/(\sigma_{0} C_{3}) |^{q}] \leq 1$ for some integer $ 4 \leq q < \infty$; 
\item[(b)] $\Ep[\exp\{\varepsilon_{1}^{2}/(\sigma_{0} C_{3})^{2}\} ] \leq 2$.
\end{enumerate}
\end{condition}

Condition \ref{Condition: residual condition} controls the norm of the bias term.
Condition \ref{Condition: moment condition} is a moment condition on the error distribution.
These conditions are sufficiently weak and in particular covers all the applications we will cover.

The following theorem, which is the main result of this section, provides bounds on frequentist coverage errors of the Bayesian credible rectangle $I(\hat{\beta},\hat{R}_{\alpha})$ together with bounds on the ``radius'' $\hat{R}_{\alpha}$ of $I(\hat{\beta},\hat{R}_{\alpha})$. 
In what follows, let $\uplambda$ and $\downlambda$ denote  the maximum and minimum eigenvalues of the matrix $(X^{\top}X)^{-1}$, respectively, and let 
$\upw:=\max\{w_{1},\ldots,w_{p}\}$ and $\downw:=\min\{w_{1},\ldots,w_{p}\}$
denote the maximal and minimal weights, respectively. 

\begin{theorem}[Coverage errors of credible rectangles]
\label{theorem: frequentist validation of Bayesian credible rectangles}
Suppose that Conditions \ref{Condition: prior mass condition}--\ref{Condition: residual condition}
and either of Condition \ref{Condition: moment condition} (a) or (b)  hold. 
Then there exist positive constants $c_{1}$ and $c_{2}$ depending only on $C_{1},C_{2},C_{3}$ and $q$ 
such that the following hold. For every $n \geq 2$, we have 
\begin{equation}
\begin{split}
	\Big{|} &  \Pr ( \beta_{0} \in I(\hat{\beta},\hat{R}_{\alpha}) ) -(1-\alpha) \Big{|} 
\\
&\leq
\phi_{\Pi_{\beta}}\left (c_{1}\sqrt{p\log n} \right) + 
c_{1}\left(\delta_{1} p\log n +\delta_{2}+\delta_{3}
+\frac{\tau}{\sigma_{0}\downlambda^{1/2} } \sqrt{\log p}
+\zeta_{n}\right)
\end{split}
\label{eq: coverage error}
\end{equation}
where $\tau:=\| (X^{\top}X)^{-1}X^{\top}r \|_{\infty} $ and
\[
\zeta_{n}
= 
\begin{cases}
p^{1-q/2}(\log n)^{-q/2} 
+  \left ( \frac{\uplambda}{\downlambda} \frac{p\log^{7}(pn)}{n}  \right )^{1/6}  +\left ( \frac{\uplambda}{\downlambda} \frac{p\log^{3}(pn)}{n^{1-2/q}} \right  )^{1/3}
& \text{under Condition \ref{Condition: moment condition} (a)} \\
n^{-c_{2} p}
	+ \left ( \frac{\uplambda}{\downlambda} \frac{p\log^{7}(pn)}{n}  \right )^{1/6}
& \text{under Condition \ref{Condition: moment condition} (b)} \\
n^{-c_{2}p } & \text{if $\varepsilon_{i}$'s are Gaussian}
\end{cases}
.
\]
In addition, 
there exist positive constants $c_{3}$ and $c_{4}$ depending only on $\alpha$ and $\downw$ such that
the following two bounds (\ref{eq: upper bound on R in theorem}) and (\ref{eq: lower bound on R in theorem}) 
hold with probability at least
\[
\begin{cases}
1-c_{1}p^{1-q/2}(\log n)^{-q/2}-\delta_{3} & \text{under Condition \ref{Condition: moment condition} (a)} \\
1- c_{1} n^{- c_{2} p} -\delta_{3} & \text{under Condition \ref{Condition: moment condition} (b)}
\end{cases}
.
\]
Provided that the right hand side on (\ref{eq: coverage error}) is smaller than $\min\{\alpha / 2, (1-\alpha)/2\}$,
the diameter $\hat{R}_{\alpha}$ is bounded from above as
\begin{align}
\hat{R}_{\alpha} \leq c_{3} \sigma_{0} \uplambda^{1/2} \Ep\Big{[}\max_{1 \le i \le p}| N_{i} / w_{i} |\Big{]}
\label{eq: upper bound on R in theorem}
\end{align}
for $N_{1},\dots,N_{p} \sim \mathcal{N}(0,1)$ i.i.d., and
for sufficiently large $p$ depending only on $\alpha$,
the diameter $\hat{R}_{\alpha}$ is bounded from below as 
\begin{align}
c_{4} \sigma_{0} \downlambda^{1/2} \upw^{-1}\sqrt{\log p}\leq\hat{R}_{\alpha}.
\label{eq: lower bound on R in theorem}
\end{align}
\end{theorem}

Theorem \ref{theorem: frequentist validation of Bayesian credible rectangles} shows that
that the frequentist coverage error of the Bayesian credible rectangle depends on the prior $\Pi_{\beta}$ on $\beta$ only through 
the lack-of-flatness function $\phi_{\Pi_{\beta}}$.
The discussions below provide a typical bound on $\phi_{\Pi_{\beta}}$.
We note that the requirement that the right hand side on (\ref{eq: coverage error}) is smaller than $\alpha/2$ is used to derive the upper bound on $\hat{R}_{\alpha}$, while the requirement that the same quantity is smaller than $(1-\alpha)/2$ is used to derive the lower bound on $\hat{R}_{\alpha}$.

\subsection{Discussions on conditions}

We first verify that a \textit{locally log-Lipschitz prior} satisfies 
Conditions \ref{Condition: prior mass condition} and \ref{Condition: for simplicity},
providing an upper bound of $\phi_{\Pi_{\beta}}$.

\begin{definition}\label{definition: locally log-Lipschitz prior}
A \textit{locally log-Lipschitz prior} is defined as a prior distribution on $\beta$
such  there exists $L=L_{n}>0$ with
\[
| \log \pi (\beta) - \log \pi (\beta_{0}) | \leq L \| \beta - \beta_{0} \|
\text{ for all $\beta$ with }
\|\beta-\beta_{0}\| \leq \sigma_{0} \uplambda^{1/2} \sqrt{p\log n}.\]
\end{definition}

\begin{proposition}\label{proposition: rate of convergence locally log-Lipschitz priors}
For a locally log-Lipschitz prior $\Pi_{\beta}$ with  log-Lipschitz constant $L$,
we have 
$ \phi_{\Pi_{\beta}} (c \sqrt{p\log n} ) \leq c L \sigma_{0} \uplambda^{1/2} \sqrt{p\log n}$
for any $c>0$.
Hence the prior $\Pi_{\beta}$ satisfies Condition \ref{Condition: for simplicity} if 
$\sigma_{0}L\uplambda^{1/2}/\sqrt{n} \leq 1/2$.
\end{proposition}

To provide examples of prior distributions on $\beta$ that satisfy Condition \ref{Condition: prior mass condition},
we focus on the following two subclasses of locally log-Lipschitz priors.
Let $B:=\|\beta_{0}\|$ denote the Euclidean norm of $\beta_{0}$. 

\begin{enumerate}
\setlength{\leftskip}{1.5cm}
\item[(Isotropic prior)]
An \textit{isotropic prior} is of the form $\pi(\beta) = \rho (\|\beta\|) / \int \rho ( \|\beta\|) d\beta$
where $\rho$ is a probability density function on $\R_{+}$ such that
$\rho$ is strictly positive and continuously differentiable on $[0,B+\sigma_{0}\uplambda^{1/2}\sqrt{p\log n}]$,
and
such that
$\int_{0}^{\infty}x^{k}\rho(x)dx \le \exp(m k\log k)$ for all $k\in\mathbb{N}$ for some positive constant $m$.
\item[(Product prior)]
A \textit{product prior} of log-Lipschitz priors is of the form $\pi(\beta) = \prod_{i=1}^{p}\pi_{i}(\beta_{i})$
where each $\log \pi_{i}$ is strictly positive on $[0,B+\sigma_{0}\uplambda^{1/2}\sqrt{p\log n}]$ and  $\tilde{L}$-Lipschitz for some $\tilde{L}>0$.
\end{enumerate}

For the sake of exposition, 
we make the following additional condition 
to verify that isotropic or product priors satisfy  Condition \ref{Condition: prior mass condition}.

\begin{condition}
\label{Condition: for prior log-Lipschitz}
There exists a positive constant $c$ such that
$\log \{ \sqrt{\mathrm{det} (X^{\top}X)} / \sigma^{p}_{0} \} \leq c p \log n.$
\end{condition}

This condition is satisfied in all the applications we will cover in Section \ref{sec: applications}. 
The following proposition shows that isotropic or product priors
are locally log-Lipschitz priors satisfying Condition \ref{Condition: prior mass condition}.

\begin{proposition}\label{proposition: examples of locally log Lipschitz priors}
Under Condition \ref{Condition: for prior log-Lipschitz},
an isotropic prior and a product prior of log-Lipschitz priors satisfy
Condition \ref{Condition: prior mass condition}.
An isotropic prior is a locally log-Lipschitz prior with locally log-Lipschitz constant $L$ such that
\[L \leq  c_{1}B\max_{x: 0\leq x \leq B+\sigma_{0}\uplambda^{1/2}\sqrt{p\log n}} | (\log \rho )' ( x ) |\]
for some positive constant $c_{1}$ depending only on $m$ and $c$ that appear  in the definition of $\rho$ 
and Condition \ref{Condition: for prior log-Lipschitz}.
In particular, if $\pi(\beta)$ is the standard Gaussian density, then  $L \leq c_{1}B^{2}$.
A product prior of log-Lipschitz priors with log-Lipschitz constant $\tilde{L}$
is locally log-Lipschitz with $L=\tilde{L}p^{1/2}$.
\end{proposition}

Next, we will discuss  Condition \ref{Condition: marginal posterior contraction}.
We consider following two cases:

\begin{enumerate}
\setlength{\leftskip}{1.5cm}
\item[(Plug-in)] $\Pi_{\sigma^{2}}=\Pi_{\hat{\sigma}^{2}_{\mathrm{u}}}$ with
$\hat{\sigma}^{2}_{\mathrm{u}}(Y) := \|Y - X(X^{\top}X)^{-1}X^{\top} Y\|^{2} / (n-p)$;
\item[(Full-Bayes)]$\Pi_{\beta}$ is the standard Gaussian distribution and 
$\Pi_{\sigma^{2}}$ is the inverse Gamma distribution $\mathrm{IG}(\mu_{1},\mu_{2})$ with  shape parameter $\mu_{1}>1/2$ and  scale parameter $\mu_{2}>1/2$.
\end{enumerate}

The following two propositions yield possible choices of $\delta_{1},\delta_{2},$ and $\delta_{3}$.

\begin{proposition}[Plug-in] \label{proposition: Condition in the case of empirical Bayes}
Suppose that Condition \ref{Condition: moment condition} holds and also that $n \geq c p$ for some $c > 1$. 
In addition, suppose that $\delta_1 >0$ satisfies that 
$\tilde{\delta}_{1} := [\delta_{1} - 2 \| r \|^{2} / \{ \sigma_{0}^{2} (n-p) \} - 1 / (n-p)] > 0.$
Then
there exist positive constants $c_{1}$ and $c_{2}$ depending only on $c$, $C_{3}$ and $q$ such that
\begin{align*}
	\Pr\left(|\hat{\sigma}^{2}_{\mathrm{u}}/\sigma^{2}_{0}-1 | \geq \delta_{1} \right)
	\leq
	\begin{cases}
		c_{1}\max\{n^{-4/q}\delta_{1}^{-q/2}, n^{1-q/2} \tilde{\delta}^{-q}_{1} \}  
		&\text{under Condition \ref{Condition: moment condition} (a)}, \\
		c_{1} \exp( - c_{2} n \max\{\delta_{1}^{2} , \tilde{\delta}_{1}^{2}\} )  
		&\text{under Condition \ref{Condition: moment condition} (b)}.
	\end{cases}
\end{align*}
\end{proposition}

\begin{proposition}[Full-Bayes] \label{proposition: Condition in the case of full Bayes}
Suppose that Condition \ref{Condition: moment condition} holds and also $n \geq c p$ for some $c > 1$. 
In addition, suppose that $\delta_1 >0$ satisfies that 
$\tilde{\delta}_{1} := [\delta_{1} - 2 \| r \|^{2} / \{ \sigma_{0}^{2} (n-p) \} - 1 / (n-p)] > 0.$
Then
there exist positive constants $c_{1}$ and $c_{2}$ depending only on $c$, $\mu_{1}$, $\mu_{2}$, $C_{3}$ and $q$ 
such that
\begin{align*}
	\Pi_{\sigma^{2}}(\sigma^{2}: |\sigma^{2}/\sigma^{2}_{0}-1| > \delta_{1} \mid Y)
	\leq c_{1}(n\tilde{\delta}_{1})^{-1}
\end{align*}
with probability at least 
\[
\begin{cases}
	1-c_{1}\max\{n^{-4/q}\delta^{-q/2}_{1}, n^{1-q/2}\tilde{\delta}_{1}^{-q}\} 
	& \text{ under Condition \ref{Condition: moment condition} (a)}, \\
	1-c_{1}\exp(-c_{2}n\max\{\delta_{1}^2,\tilde{\delta}_{1}^{2} \})
	& \text{ under Condition \ref{Condition: moment condition} (b)}.
\end{cases}
\]
\end{proposition}

To better understand implications of these propositions, 
Table \ref{table:orders} summarizes possible rates of $\delta_{1},\delta_{2},\delta_{3}$
when $n \geq c p$ for some $c > 0$, $\|r\|^{2}/n=o(n^{-1/2})$, and $\sigma_{0}^{2}$ is independent of $n$.
\begin{table}[htb]
\caption{Possible rates of $\delta_{1},\delta_{2},\delta_{3}$ with respect to $n$: $\kappa$ is arbitrary.}
	\label{table:orders}
\begin{center}
	\begin{tabular}{|l||c|c|c|}\hline
		Condition \ref{Condition: moment condition} and prior & $\delta_{1}$            & $\delta_{2}$ & $\delta_{3}$ \\ \hline
		(a) and plug-in                                          & $n^{-1/2+\kappa/q}$     & $0$          & $\max\{n^{-\kappa/2},n^{1-\kappa}\}$ \\ \hline
		(a) and full Bayes                                           & $n^{-1/2+\kappa/q}$     & $n^{-1/2-\kappa/q}$          & $\max\{n^{-\kappa/2},n^{1-\kappa}\}$ \\ \hline
		(b) and plug-in                                          & $n^{-1/2}\sqrt{\log n}$ & $0$          & $n^{-1}$ \\ \hline   
		(b) and full Bayes                                          & $n^{-1/2}\sqrt{\log n}$ & $n^{-1/2}(\log n)^{-1/2}$    & $n^{-1}$ \\ \hline
	\end{tabular}
\end{center}
\end{table}

\begin{remark}[Comparison with \cite{YooandGhosal(2016)}]
Proposition 4.1 in \cite{YooandGhosal(2016)} studies possible rates for $\delta_1$ 
when a prior for $\beta$ is Gaussian and the error distribution is sub-Gaussian. 
Our results in Propositions \ref{proposition: Condition in the case of empirical Bayes} and \ref{proposition: Condition in the case of full Bayes}
are compatible with their result up to logarithmic factors under their setup. 
\end{remark}

\subsection{Berry--Esseen type bounds on posterior distributions}\label{section: Berry Esseen type bounds}

Before presenting applications of the main theorem,
we derive an important ingredient of the proof of Theorem \ref{theorem: frequentist validation of Bayesian credible rectangles}, 
namely, the Berry--Esseen type bound on posterior distributions.
For $R>0$,
let $H(R)$ be the intersection of the sets  $\{Y \in \mathbb{R}^{n} : \| X (\hat{\beta}(Y)-\beta_{0}) \| \leq R \sqrt{p\log n} \sigma_{0} / 4 \}$
and $\{Y \in \mathbb{R}^{n} : \Pi_{\sigma^{2}}(|\sigma^{2}/\sigma_{0}^{2} -1 | \geq \delta_{1} \mid Y) \leq \delta_{2}  \}$.
For two probability measures $P$ and $Q$,
$\|P-Q\|_{\mathrm{TV}}$ denotes the total variation between $P$ and $Q$.

\begin{proposition}[Berry--Esseen type bounds on posterior distributions]
\label{proposition: Berry Esseen type bound on posterior distribution}
Under Conditions \ref{Condition: prior mass condition}--\ref{Condition: for simplicity},
there exist positive constants $c_{1}$ and $c_{2}$
depending only on $C_{1},C_{2},C_{3}$
such that
for every $n\geq 2$,
\begin{align*}
\left \| \Pi_{\beta}(\cdot \mid Y)-\mathcal{N}(\hat{\beta},\sigma^{2}_{0}(X^{\top}X)^{-1}) \right \|_{\mathrm{TV}}
\leq    \phi_{\Pi_{\beta}} (c_{1}\sqrt{p\log n})+
c_{1} (\delta_{1}p\log n +\delta_{2} + n^{-c_{2}p})
\end{align*}
whenever  $Y\in H(c_{1})$.
\end{proposition}

\begin{proposition}
\label{proposition: test set}
Under Conditions \ref{Condition: residual condition} and \ref{Condition: moment condition},
there exist positive constants $c_{1}$ and $c_{2}$ depending only on $C_{2}$, $C_{3}$,  and $q$ 
such that
\begin{align*}
        \Pr (Y \notin H(c_{1})) \leq
	\begin{cases}
		c_{1}  p^{1-q/2}(\log n)^{-q/2}+\delta_{3}
		& \text{ under Condition \ref{Condition: moment condition} (a)}, \\
		c_{1} n^{-c_{2} p} + \delta_{3}
		& \text{ under Condition \ref{Condition: moment condition} (b)}.
	\end{cases}
\end{align*}
\end{proposition}

\begin{remark}[Critical dimension for the Bernstein--von Mises theorem]
\label{rem: critical dimension}

The previous propositions immediately lead to the critical dimension for the BvM theorem.
We will compare our result with the results on the critical dimension by \cite{Bontemps(2011),Ghosal(2000),Spokoiny(2014)}.
In this comparison, we assume a locally log-Lipschitz prior with locally log-Lipschitz constant $L$;  that $\|\beta_{0}\|$ and $L$ are independent of $n$; and that $\sigma_{0}\uplambda^{1/2}\sim n^{-1/2}$.
The following are a summary of the existing results: 
\begin{itemize}
\item \cite{Ghosal(2000)} shows that 
when the error distribution has a smooth density
with known scale parameter,
the BvM theorem holds if $p^{4}\log p = o(n)$ and some additional assumptions are verified;
\item \cite{Spokoiny(2014)} shows that
when the high-dimensional local asymptotic normality holds,
the BvM theorem holds if $p^{3} = o(n)$; see also \cite{PanovandSpokoiny(2015)};
\item \cite{Bontemps(2011)} shows that
when the error distribution is Gaussian with known variance,
the BvM theorem holds if $p \log n = o(n)$.
\end{itemize}
Our result (Propositions \ref{proposition: rate of convergence locally log-Lipschitz priors}, \ref{proposition: Condition in the case of empirical Bayes},
\ref{proposition: Berry Esseen type bound on posterior distribution}, and \ref{proposition: test set}) 
improves on \cite{Ghosal(2000),Spokoiny(2014)} in that
\begin{itemize}
\item when the error variance is assumed to be known (i.e., $\delta_1=\delta_2=\delta_3=0$), our result implies that
      the BvM theorem (for the quasi-posterior distribution) holds if $p\log n =o(n)$ and if the error distribution has finite fourth moment. 
      Compared to \cite{Ghosal(2000)},
      our result substantially improves on the critical dimension by employing the Gaussian likelihood 
      even when the Gaussian specification  is incorrect;
\item when the error variance is unknown,
      our result shows that
      the BvM theorem holds for $\beta$ if $p^{2}(\log n)^{3} =o(n)$ for sub-Gaussian error distributions, 
      thereby improving on the condition of \cite{Spokoiny(2014)}. 
\end{itemize}

Importantly, our result covers the unknown error variance case, 
which makes our analysis different from  \cite{Bontemps(2011)}. 
In nonparametric regression, it is usually the case that the error variance is unknown, 
and hence it is important to consider unknown variance cases in such an application.
If the error distribution is Gaussian with a known error variance, 
our result is consistent with \cite{Bontemps(2011)}.
\end{remark}

\section{Applications}
\label{sec: applications}

In this section, we consider applications of the general results developed in the previous sections to quantifying coverage errors of Bayesian credible sets in Gaussian white noise models, linear inverse problems, and (possibly non-Gaussian) nonparametric regression models. 

\subsection{Gaussian white noise model}
\label{section: Castillo--Nickl credible bands in Gaussian white noise model}

We first consider a Gaussian white noise model and analyze coverage errors of Castillo-Nickl credible bands. Consider a Gaussian white noise model 
\[
dY(t)=f_{0}(t)dt+\frac{1}{\sqrt{n}} dW(t),\ t\in[0,1],
\]
where $dW$ is a canonical white noise and $f_{0}$ is an unknown function. 
We assume that $f_{0}$ is in the H\"{o}lder--Zygmund space $B_{\infty,\infty}^{s}$ with smoothness level $s>0$.
It will be convenient to define the H\"{o}lder--Zygmund space $B_{\infty,\infty}^{s}$ by using a wavelet basis. 
Let $S > s$ be an integer and fix sufficiently large $J_{0}=J_{0}(S)$.
Let $\{\phi_{J_{0},k}: 0 \leq k \leq 2^{J_{0}}-1\} \cup \{ \psi_{l,k} : J_{0} \leq l , 0 \leq k \leq 2^{l}-1 \}$
be an $S$-regular Cohen--Daubechies--Vial (CDV) wavelet basis of $L^{2}[0,1]$. 
Then the H\"{o}lder--Zygmund space $B_{\infty,\infty}^{s}$ is defined by 
$B_{\infty,\infty}^{s}=\{f:\|f\|_{B^{s}_{\infty,\infty}} <\infty\}$
with
\[
\|f\|_{B^{s}_{\infty,\infty}}:=
\max_{0 \leq k \leq 2^{J_{0}}-1 } |\langle \phi_{J_{0},k} , f \rangle|
+
\sup_{ J_{0} \leq l < \infty, 0 \leq k \leq 2^{l} -1} 2^{l(s+1/2)} |\langle \psi_{l,k} , f \rangle|,
\]
where $\langle \cdot, \cdot \rangle$ denotes the $L^{2}[0,1]$ inner product, i.e., $\langle f,g \rangle := \int_{[0,1]} f(t)g(t)dt$.
In what follows, for the notational convention,
let $\psi_{J_{0}-1,k}:= \phi_{J_{0},k}$ for $0\leq  k \leq 2^{J_{0}}-1$.

Consider a sieve prior for $f$,
that is, a prior deduced from a prior $\Pi_{\beta}$ on $\R^{2^{J}}$ with $J\geq J_{0}$ via the map
$ (\beta_{J_{0}-1,0},\beta_{J_{0}-1,1},\ldots,\beta_{J-1,2^{J-1}-1})  \mapsto \sum_{(l,k)\in \mathcal{I}(J)} \psi_{l,k}(\cdot)\beta_{l,k}, $
where
$\mathcal{I}(J):=\{(l,k): J_{0} \leq l \leq J-1 , 0 \leq k \leq 2^{l}-1 \} \cup \{(l,k) : l=J_{0}-1 , 0\leq k \leq 2^{J_{0}}-1\}$.

For given $\alpha\in(0,1)$,
the $(1-\alpha)$-Castillo--Nickl credible band
based on an efficient estimator $\hat{f}$, an admissible sequence $w=(w_{1},w_{2},\ldots)$, 
and a sieve prior $\Pi_{\beta}$
is defined as
\begin{align*}
	\mathcal{C}_{w}(\hat{f},\hat{R}_{\alpha}) 
        := \left \{ f: \sup_{(l,k) \in \mathcal{I}_{\infty} } \frac{| \langle f - \hat{f} ,\psi_{l,k} \rangle |}{ w_{l} } \leq \hat{R}_{\alpha} \right \}
\end{align*}
where 
$\mathcal{I}_{\infty}
:=\{(l,k): J_{0} \leq l < \infty  , 0 \leq k \leq 2^{l}-1 \} \cup \{(l,k) : l=J_{0}-1 , 0\leq k \leq 2^{J_{0}}-1\}$,
and
an admissible sequence $w$ is defined as a positive sequence such that $w_{l}/\sqrt{l} \uparrow \infty$ as $l\to\infty$.
The radius $\hat{R}_{\alpha}$ of the band is taken
in such a way that $\Pi_{\beta} \{ \mathcal{C}_{w}(\sum_{(l,k)\in\mathcal{I}(J)}\langle \hat{f} , \psi_{l,k} \rangle \psi_{l,k} ,\hat{R}_{\alpha}) \mid Y \}= 1-\alpha.$
Truncating a centering estimator ensures that such radius indeed exists for a sieve prior.

The following proposition derives bounds on the coverage error and the $L^{\infty}$-diameter of the Castillo--Nickl credible band
based on a sieve prior.
In the following proposition,
we use $\hat{f}_{\infty} := \sum_{(l,k)\in \mathcal{I}_{\infty} } \psi_{l,k}  \int \psi_{l,k} dY$
(which converges almost surely in $\mathcal{M}_{0}(w)$) as a centering estimator.
See p.~1946 of \cite{CastilloandNickl(2014)} for the definition of $\mathcal{M}_{0}(w)$ and well-definedness of $\hat{f}_{\infty}$.
Let 
\[
u_{J}:=\inf_{J\leq l<\infty} w_{l} / \sqrt{l},\
v_{J}:=\max_{J_{0}-1 \leq l \leq J-1} w_{l} / \sqrt{l} ,\
\text{ and }
\overline{w}_{J}:=\max_{J_{0}-1 \leq l \leq J-1 } w_{l}.
\]
In addition, let $\tilde{H}:=\{Y: \sup_{J\leq l <\infty , 0\leq k\leq 2^{l}-1}|\langle f_{0}-\hat{f}_{\infty},\psi_{l,k} \rangle|/w_{l}\leq \hat{R}_{\alpha}\}$.
For simplicity,
we assume that $\sqrt{l}\leq w_{l}$ for $J_{0}-1\leq l < \infty$ and $1\leq (J/\overline{w}^{2}_{J})u_{J}^{2} \uparrow \infty$ as $J\to\infty$.

\begin{proposition}\label{theorem: frequentist evaluation of multiscale Castillo--Nickl credible bands}
Under Conditions \ref{Condition: prior mass condition} and \ref{Condition: for simplicity} for $\Pi_{\beta}$ with $p=2^{J}$, $X=I_{p}$, and $\sigma_{0}=1/\sqrt{n}$,
there exist positive constants $c_{1},c_{2}$ depending only on $C_{1}$ appearing in Condition \ref{Condition: prior mass condition}
such that
the following hold.
For $n\geq 2$,
we have
\begin{align*}
	| \Pr ( f_{0} \in \mathcal{C}_{w} (\hat{f}_{\infty} ,\hat{R}_{\alpha}) ) - (1-\alpha) |
	\leq
	\phi_{\Pi_{\beta}} \left(c_{1} \sqrt{ 2^{J} \log n } \right) 
	+ c_{1} n^{ -c_{2} 2^{J} } + \Pr(Y\not\in \tilde{H}).
\end{align*}
In addition, there exist positive constants $c_{3},c_{4}$ depending only on $\alpha$
such that the following hold.
Assume that the right hand side above except $\Pr(Y\not\in \tilde{H})$ is smaller than $\min\{\alpha/2,(1-\alpha)/2\}$. Then
\[
\Pr(Y\not\in\tilde{H})\leq  c_{3} \left ( \mathrm{e}^{-c_{4}J(J/\overline{w}^{2}_{J})u^{2}_{J}} + n^{-c_{2}2^{J}} \right )
\]
 for sufficiently large $J$ depending only on $\alpha$ and $\{w_{l}\}$; 
 and the $L^{\infty}$-diameter of the intersection $\mathcal{C}_{w}^{B}(\hat{f}_{\infty},\hat{R}_{\alpha}) := \mathcal{C}_{w}(\hat{f}_{\infty},\hat{R}_{\alpha}) \cap \{f:\|f\|_{B^{s}_{\infty,\infty}} \leq B\}$ for any $B > 0$
is bounded from above as
\begin{align*}
	\sup_{f,g\in \mathcal{C}_{w}^{B}(\hat{f}_{\infty},\hat{R}_{\alpha})} \| f-g \|_{\infty} 
        \leq c_{3} \left ( v_{J}  \sqrt{\frac{2^{J}J}{n}} + 2^{-Js}B \right )
\end{align*}
with probability at least $1-c_{1} n^{-c_{2}2^{J}}$.

\end{proposition}

\begin{proof}[Proof sketch of Proposition \ref{theorem: frequentist evaluation of multiscale Castillo--Nickl credible bands}]
First, we transform the Gaussian white noise model into a Gaussian infinite sequence model 
$Y_{l,k} = \beta_{0,l,k} + \varepsilon_{l,k}, \ (l,k) \in\mathcal{I}_{\infty} ,$
where $\beta_{0,l,k} := \langle  f_{0}, \psi_{l,k} \rangle$ for $(l,k)\in\mathcal{I}_{\infty}$,
and $\varepsilon_{l,k}$ are i.i.d. $\mathcal{N}(0,1/n)$ variables. 
Second, we apply Theorem \ref{theorem: frequentist validation of Bayesian credible rectangles}. 
Let $Y_{\infty}=\{Y_{l,k}:(l,k)\in\mathcal{I}_{\infty}\}$ and observe that $\Pr(Y\not\in\tilde{H})=\Pr(Y_{\infty}\not\in \tilde{H}')$ with $\tilde{H}'=\{Y_{\infty}: \sup_{J<l,0\leq k \leq 2^{l}-1} |Y_{l,k}-\beta_{0,l,k}| / w_{l} \leq \hat{R}_{\alpha}\}$.
Since 
\[
\Pr(f_{0} \in\mathcal{C}_{w}(\hat{f}_{\infty},\hat{R}_{\alpha})) 
= \Pr\left ( \max_{(l,k)\in\mathcal{I}(J)}| \varepsilon_{l,k}/w_{l} |  \bigvee \sup_{J\leq l<\infty,0\leq k\leq 2^{l}-1}|\varepsilon_{l,k}/w_{l}| \leq \hat{R}_{\alpha} \right ),
\]
we have
\begin{align*}
\left | \Pr(f_{0}\in\mathcal{C}_{w}(\hat{f}_{\infty},\hat{R}_{\alpha}))  - \Pr \left ( \max_{(l,k)\in\mathcal{I}(J)} | \varepsilon_{l,k}/w_{l} |  \leq \hat{R}_{\alpha} \right ) 
\right | 
\leq \Pr(Y_{\infty} \not \in \tilde{H}').
\end{align*}
Then we apply Theorem \ref{theorem: frequentist validation of Bayesian credible rectangles}
with $p=2^{J}$, $Y=\{Y_{l,k}:(l,k)\in\mathcal{I}(J)\}$, $X=I_{p}$, $\sigma_{0}=1/\sqrt{n}$, and $r=0$
to obtain bounds on $\Pr(\max_{(l,k)\in\mathcal{I}(J)} |\varepsilon_{l,k}/w_{l}| \leq \hat{R}_{\alpha})$ and $\hat{R}_{\alpha}$.
It remains to bound $\Pr(Y_{\infty} \not \in \tilde{H}')$. 
To this end, we use the concentration inequality for the Gaussian maximum together with a high-probability lower bound on $\hat{R}_{\alpha}$.
The detail can be found in Appendix C.1 of \cite{supplement}. 
\end{proof}

\begin{remark}[Coverage error rates]
\label{remark: coverage error rates in Gaussian white noise models}
The finite sample bound in Proposition \ref{theorem: frequentist evaluation of multiscale Castillo--Nickl credible bands} 
leads to the following asymptotic results as $n \to \infty$.
In this discussion, 
we assume a locally log-Lipschitz prior with locally log-Lipschitz constant $L=L_{n}$ 
and a true function $f_{0}$ with $\|f_{0}\|_{B^{s}_{\infty,\infty}}\leq B$ for some $B=B_{n}$.
Set $2^{J}=(n/\log n)^{1/(2s+1)}$ and set $w_{l}=\sqrt{l}$ for $l\leq J-1$ and $w_{l}=u_{l} \sqrt{l}$ for $l\geq J$
with $u_{l} \uparrow \infty$ as $l \to \infty$.
Then we have
\begin{align}
&	| \Pr ( f_{0} \in \mathcal{C}^{B}_{w}(\hat{f}, \hat{R}_{\alpha}) ) -  (1-\alpha) | 
	\leq O( L_{n} (n/\log n)^{-s/(2s+1)}) \quad \text{and} 
	\label{eq: rate of convergence of coverage error in white noise model} \\
& \sup_{f,g\in \mathcal{C}_{w}^{B}(\hat{f}_{\infty},\hat{R}_{\alpha})}\| f-g \|_{\infty} 
        \leq O( B_{n}(n / \log n )^{-s/(2s+1)} ),
\end{align}
where the latter holds with  probability at least $1-c_{1}n^{-c_{2}2^{J}}$  (the sequence $\{w_{l}\}$ here depends on $n$, but we can apply Proposition 3.1; see Remark C.1 in \cite{supplement} for the detail).
In particular,
for the standard Gaussian prior, the coverage error is $O ( B^{2}_{n}(n / \log n )^{-s/(2s+1)} )$. 
We note that the above asymptotic results are derived from the non-asymptotic result in Proposition \ref{theorem: frequentist evaluation of multiscale Castillo--Nickl credible bands} where the constants do no depend on $f_{0}$; hence the above asymptotic results hold uniformly in $f_{0}$ as long as $\| f_{0} \|_{B_{\infty,\infty}^{s}} \le B$. The same comments apply to the subsequent results. 
\end{remark}

\begin{remark}[Comparison of coverage errors]
\label{remark: comparison of coverage errors}
The previous remark shows that Bayesian credible bands have coverage errors (for the true function) decaying polynomially fast in the sample size $n$.
This rate is much faster than that of confidence bands based on Gumbel approximations (see Proposition 6.4.3 in \cite{GineandNickl_book});
confidence bands based on Gumbel approximations have coverage errors decaying only at the $1/\log n$ rate.
In the kernel density estimation case,
\cite{Hall(1991)} shows that confidence bands based on Gumbel approximations have coverage errors decaying only at the $1/\log n$ rate, 
while bootstrap confidence bands have coverage errors decaying polynomially fast in $n$ for the surrogate function.
\end{remark}

\begin{remark}[Undersmothing]
In most cases, a priori bound on $\| f_{0} \|_{B_{\infty,\infty}^{s}}$ is unknown, and so $B = B_{n}$ should be chosen as a slowly divergent sequence, which can be thought of as a ``undersmoothing'' penalty (cf. \cite{CastilloandNickl(2014)} Remark 5). Interestingly, however, our result shows that this undersmoothing penalty only affects the $L^{\infty}$-diameter and not affect the coverage error of the band, which is a sharp contrast with standard $L^{\infty}$-confidence bands for densities or regression functions. 
\end{remark}

Consider another centering estimator:
$\hat{f}_{J}:= \sum_{(l,k)\in \mathcal{I}(J) } \psi_{l,k}  \int \psi_{l,k} dY.$
The following proposition derives bounds on the coverage error and the $L^{\infty}$-diameter of the Castillo--Nickl credible band
based on a sieve prior and the centering estimator $\hat{f}_{J}$.
We use the same notation $u_{J},v_{J},\overline{w}_{J}$ as in the previous proposition.
Let 
\[
\tilde{H}_{2} := \left\{ Y : \sup_{J\leq l <\infty , 0\leq k \leq 2^{l}-1} | \langle f_{0}, \psi_{l,k} \rangle| / w_{l} \leq \hat{R}_{\alpha} \right\}.
\]
For simplicity,
we assume $\sqrt{l}\leq w_{l}$ for $J_{0}-1\leq l < \infty$.
\begin{proposition}\label{theorem: frequentist evaluation of multiscale Castillo--Nickl credible bands with cut-off}
Under Conditions \ref{Condition: prior mass condition} and \ref{Condition: for simplicity} for $\Pi_{\beta}$ with $p=2^{J}$, $X=I$, and $\sigma_{0}=1/\sqrt{n}$,
there exist positive constants $c_{1},c_{2},c_{3}$ depending only on $C_{1}$ appearing in Condition \ref{Condition: prior mass condition} 
and $\alpha$
such that
the following hold. 
For $n\geq 2$ and for $B>0$ satisfying $\|f_{0}\|_{B^{s}_{\infty,\infty}} \leq  B$,
we have
\begin{align*}
	| \Pr ( f_{0} \in \mathcal{C}_{w} (\hat{f}_{J} ,\hat{R}_{\alpha}) ) - (1-\alpha) |
	\leq
	\phi_{\Pi_{\beta}} \Big{(}c_{1} \sqrt{ 2^{J} \log n } \Big{)} 
	+ c_{1} n^{-c_{2} 2^{J}} + \Pr(Y\not\in \tilde{H}_{2}).
\end{align*}
In addition,
assume that 
the right hand side above except $\Pr(Y\not\in \tilde{H}_{2})$ is smaller than $\min\{\alpha/2,(1-\alpha)/2\}$. Then 
the $L^{\infty}$-diameter of the intersection 
$\mathcal{C}_{w}^{B}(\hat{f}_{J},\hat{R}_{\alpha}) :=\mathcal{C}_{w}(\hat{f}_{J},\hat{R}_{\alpha}) \cap \{f:\|f\|_{B^{s}_{\infty,\infty}} \leq B\}$ 
is bounded from above as
\begin{align*}
	\mathop{\sup}_{f,g\in \mathcal{C}_{w}^{B}(\hat{f}_{J},\hat{R}_{\alpha})}  \| f-g \|_{\infty} 
        \leq c_{3} \left ( v_{J}\sqrt{\frac{2^{J}J}{n}} + 2^{-Js}B \right )
\end{align*}
with probability at least $1-c_{1} n^{-c_{2}2^{J}}$.
If in addition $(\sqrt{n}\overline{w}_{J}B)/(u_{J}J2^{J(s+1/2)}) \downarrow 0 \text{ as }J\to\infty$, then 
$ \Pr(Y\not\in \tilde{H}_{2}) \leq c_{1} n^{ -c_{2}2^{J} } $
for sufficiently large $J$ depending only on $\alpha$, $\{w_{l}\}$, and $B$.
\end{proposition}

A proof of the proposition is given in Appendix C.2 of \cite{supplement}.

\begin{remark}[Choice of the sequence $w$]
Consider the same setting as in Remark \ref{remark: coverage error rates in Gaussian white noise models}.
Then we have $(\sqrt{n}\overline{w}_{J}B)/(u_{J}J2^{J(s+1/2)}) = O(B/u_{J})$ and so 
the  sequence $u_{l}$ must satisfy $u_{J}/B_{n} \to \infty$ as $n\to\infty$
to ensure that $(\sqrt{n}\overline{w}_{J}B)/(u_{J}J2^{J(s+1/2)}) \downarrow 0$ as $J\to\infty$.
Without this exception, 
the same asymptotic results hold as in Remark \ref{remark: coverage error rates in Gaussian white noise models}.
\end{remark}

\subsection{Linear inverse problem}
\label{section: credible bands in linear inverse problems}

In this section we extend the previous analysis to a linear inverse problem
\begin{align*}
	dY(t) = K(f_{0})(t)dt + \frac{1}{\sqrt{n}} d W(t), \ t\in[0,1],
\end{align*}
where $K$ is a known linear operator
and $f_{0}$ is included in the H\"{o}lder--Zygmund space $B_{\infty,\infty}^{s}$ for some $s > 0$ as described in the previous section.
To describe the degree of ill-posedness, we use the wavelet-vaguelette decomposition 
$\{\psi_{l,k} , v^{(1)}_{l,k} , v^{(2)}_{l,k} , \kappa_{l,k} : (l,k) \in \mathcal{I}_{\infty} \}$ of $K$,
where 
$\{\psi_{l,k}\}$ is a wavelet basis (with the same notational convention used in the previous subsection),
$\{v^{(1)}_{l,k}\}$ and $\{v^{(2)}_{l,k}\}$ are near-orthogonal functions,
and
$\{\kappa_{l,k}\}$ are quasi-singular values such that
$ K ( \psi_{l,k} )  = \kappa_{l,k} v^{(2)}_{l,k}$ for $(l,k) \in \mathcal{I}_{\infty}$.
For details, see \cite{AbramovichandSilverman(1998),Donoho(1995),Kato(2013),Johnstone(2017)} and references therein.
Our results cover both mildly ill-posed and severely ill-posed cases for $\{\kappa_{l,k}\}$. 
Say that the problem of recovering $f_{0}$ is mildly ill-posed if 
$\kappa_{l,k} \sim 2^{-r l}$ for some $r > 0$, and severely ill-posed if
$\kappa_{l,k} \sim \mathrm{e}^{- r 2^{l}}$ for some $r > 0$.

We consider a sieve prior induced from a prior $\Pi_{\beta}$ on $\R^{2^{J}}$ with $J\geq J_{0}$ 
via expanding the function $f$ using the wavelet basis $\{ \psi_{l,k}\}$.
For given $\alpha\in(0,1)$,
consider
the $(1-\alpha)$-Castillo--Nickl credible band for $f$ based on a sieve prior $\Pi_{\beta}$ and a sequence $w=(w_{1},w_{2},\ldots)$
such that $\min_{0\le k \le 2^{l}-1}\kappa_{l,k}w_{l}/\sqrt{l} \uparrow \infty$ as $l\to\infty$:
\[
\mathcal{C}_{w}(\hat{f}_{\infty},\hat{R}_{\alpha}):=
\left \{ f: \max_{(l,k) \in \mathcal{I}_{\infty}} \frac{ | \langle f - \hat{f}_{\infty} ,\psi_{l,k} \rangle |}{w_{l}} \leq \hat{R}_{\alpha} \right \},
\]
where
the centering estimator is
$\hat{f}_{\infty} := \sum_{(l,k)\in \mathcal{I}_{\infty} } \psi_{l,k}  \kappa_{l,k}^{-1}\int v^{(1)}_{l,k} d Y$,
which converges almost surely in $\mathcal{M}_{0}(w)$.
See the supplement for  well-definedness of $\hat{f}_{\infty}$.
In linear inverse problems, the radius $\hat{R}_{\alpha}$ is chosen  in such a way as
$\Pi_{\beta}   (   \mathcal{C}_{w}   (    \sum_{(l,k)     \in     \mathcal{I}(J)}     \langle     \hat{f}_{\infty}    ,    \psi_{l,k}    \rangle    \psi_{l,k}    ,\allowbreak    \hat{R}_{\alpha}    )      \mid    Y    )   =     1-\alpha,$
where $\Pi_{\beta} (\cdot\mid Y)$ is the quasi-posterior under the likelihood of the truncated indirect Gaussian sequence model:
$\int v^{(1)}_{l,k} d Y   =  \kappa_{l,k}  \beta_{l,k}  +  \frac{1}{\sqrt{n}}  \int   v^{(1)}_{l,k}  d  W $ for $ (l,k)   \in   \mathcal{I}(J).$
This slight modification using the quasi-posterior as well as truncating the centering estimator is required to apply the main theorem; see the proof sketch below.

The following theorem derives bounds on the  coverage error of the Castillo--Nickl credible band in  the linear inverse problem. 
We use the same notation $\overline{w}_{J}$ as in the previous section.
Let
$u_{J}:= \inf_{J \leq l, 0 \leq k \leq 2^{l}-1} \kappa_{l,k}w_{l}/\sqrt{l}$
and
$v_{J}:= \sup_{J_{0} \leq l \leq J-1 , 0 \leq k \leq 2^{l}-1} \kappa_{l,k}w_{l} / \sqrt{l}$.
In addition, let $\overline{\kappa}_{J}:=\max_{(l,k)\in\mathcal{I}(J)}\kappa_{l,k}$
and let $\underline{\kappa}_{J}:=\min_{(l,k)\in\mathcal{I}(J)}\kappa_{l,k}$.
Let $\Sigma$ be denote the $2^{J}\times 2^{J}$ covariance matrix of $\{\int v_{l,k}^{(1)} dY: (l,k)\in \mathcal{I}(J)\}$.
Finally, let $\tilde{H}_{3}=\{Y: \sup_{J\leq l,0\leq k \leq 2^{l}-1} |\langle f-\hat{f}_{\infty}, \psi_{l,k} \rangle | / w_{l} 
\leq \hat{R}_{\alpha}\}$.
For simplicity, we assume that $1\leq \{J^{1/2}/(\overline{\kappa}_{J}\overline{w}_{J}) \} u_{J} \uparrow \infty$ as $J\to\infty$.

\begin{proposition}\label{theorem: frequentist validation of undersmoothed credible bands in ill-posed problems}
Under Conditions \ref{Condition: prior mass condition} and \ref{Condition: for simplicity} for $\Pi_{\beta}$
with $p=2^{J}$, $X  =  \Sigma^{-1/2}  \mathrm{diag}  \{  \allowbreak \kappa_{l,k}  :  (l,k)  \in  \mathcal{I}(J)  \}$, and $\sigma_{0}=1$,
there exist positive constants $c_{1},c_{2}$
depending only on $C_{1}$ appearing in Condition \ref{Condition: prior mass condition},
$K$, and $\{\psi_{l,k}: (l,k) \in \mathcal{I}_{\infty}\}$
such that the following hold.
For $n\geq 2$,
we have
\[
\begin{split}
\left | \Pr  ( f_{0}\in \mathcal{C}_{w}(\hat{f}_{\infty} ,\hat{R}_{\alpha}) ) - (1-\alpha) \right | 
\leq \phi_{\Pi_{\beta}} \Big{(}c_{1} \sqrt{2^{J}\log n}\Big{)}
+  c_{1}n^{ - c_{2} 2^{J} }
+  \Pr(Y\not\in\tilde{H}_{3}).
\end{split}
\]
In addition, 
there exist positive constants $c_{3},c_{4}>0$ depending only on $\alpha$, $K$, and $\{\psi_{l,k}: (l,k) \in \mathcal{I}_{\infty}\}$
such that the following hold.
Assume that the right hand side above except $\Pr(Y\not\in\tilde{H}_{3})$ is smaller 
than $\min\{\alpha/2,(1-\alpha)/2\}$.
Then,
\[
\Pr(Y\not\in \tilde{H}_{3}) 
\leq c_{3}\left (  \mathrm{e}^{-c_{4}J\{J/(\overline{\kappa}_{J}\overline{w}_{J})^{2}\}u^{2}_{J}} + n^{-c_{2}2^{J}} \right )
\]
for sufficiently large $J$ depending only on $\alpha$, $\{w_{l}\}$, $K$, and $\{\psi_{l,k}: (l,k) \in \mathcal{I}_{\infty}\}$; 
and
the $L^{\infty}$-diameter of 
$\mathcal{C}_{w}^{B} (\hat{f}_{\infty},\hat{R}_{\alpha}) :=\mathcal{C}_{w}(\hat{f}_{\infty},\hat{R}_{\alpha})\cap \{f: \|f\|_{B^{s}_{\infty,\infty}} \leq B\}$ for any $B>0$ is bounded from above as
\[
	\sup_{ f,g \in \mathcal{C}_{w}^{B} (\hat{f}_{\infty},\hat{R}_{\alpha})} \| f-g \|_{\infty} 
	\leq c_{3} \left (    v_{J} \sqrt{\frac{2^{J}J}{\underline{\kappa}^{2}_{J}n}}+ 2^{-Js}B \right ) 
\]
with probability at least $1-c_{1} n^{ - c_{2} 2^{J}}$.
\end{proposition}

\begin{proof}[Proof sketch of Proposition \ref{theorem: frequentist validation of undersmoothed credible bands in ill-posed problems}]
The proof  is  almost the same as that of Proposition 
\ref{theorem: frequentist evaluation of multiscale Castillo--Nickl credible bands}, but it requires an additional analysis due to  the non-orthogonality of $\{v^{(1)}_{l,k}: (l,k)\in\mathcal{I}_{\infty}\}$.
First, we transform the indirect Gaussian white noise model into an indirect Gaussian sequence model via $\{v_{l,k}^{(1)}:(l,k)\in\mathcal{I}_{\infty}\}$:
$\tilde{Y}_{l,k} = \kappa_{l,k} \beta_{0,l,k} + \tilde{\varepsilon}_{l,k}, \ (l,k) \in \mathcal{I}_{\infty},$
where $\beta_{0,l,k}:=\langle f_{0}, \psi_{l,k} \rangle$ for $(l,k)\in\mathcal{I}_{\infty}$
and $\tilde{\varepsilon}_{l,k}$ are (dependent) jointly Gaussian variables. 
Then
\[
\Pr(f_{0} \in \mathcal{C}_{w}(\hat{f}_{\infty},\hat{R}_{\alpha}))
= \Pr \left( \sup_{(l,k)\in\mathcal{I}_{\infty}} | \kappa^{-1}_{l,k}\tilde{Y}_{l,k} - \beta_{0,l,k} | / w_{l} \leq \hat{R}_{\alpha}\right).
\]
Second, we apply Theorem \ref{theorem: frequentist validation of Bayesian credible rectangles}.
Let $\tilde{Y}_{\infty}=\{\tilde{Y}_{l,k}:(l,k)\in\mathcal{I}_{\infty}\}$
and observe that 
$\Pr(Y\not\in \tilde{H}_{3})=\Pr(\tilde{Y}_{\infty}\not\in\tilde{H}'_{3})$ with
$\tilde{H}'_{3}=\{\tilde{Y}_{\infty}: \sup_{J\leq l,0\leq k \leq 2^{l}-1} | \kappa^{-1}_{l,k} \tilde{Y}_{l,k} - \beta_{0,l,k}| / w_{l} 
\leq \hat{R}_{\alpha}\}$.
Then
\[
\left | 
\Pr(f_{0}\in\mathcal{C}_{w}(\hat{f}_{\infty},\hat{R}_{\alpha}))  -
\Pr \left ( \max_{(l,k)\in\mathcal{I}(J)} | \kappa^{-1}_{l,k} \tilde{Y}_{l,k} - \beta_{0,l,k}| / w_{l} 
\leq \hat{R}_{\alpha}  \right ) 
\right | 
\leq
\Pr(\tilde{Y}_{\infty}\not\in\tilde{H}'_{3}).
\]
Consider the linear regression model with $p=2^{J}$, $Y=\Sigma^{-1/2}(\tilde{Y}_{J_{0}-1,0},\ldots,\tilde{Y}_{J-1,2^{J-1}-1})^{\top}$, 
$X=\Sigma^{-1/2}\mathrm{diag}\{\kappa_{l,k}:(l,k)\in\mathcal{I}(J)\}$,
$\beta_{0}=(\beta_{0,J_{0}-1,0},\ldots,\beta_{0,J-1,2^{J-1}-1})^{\top}$,
$r=0$, $\sigma_{0}=1$, and $\varepsilon=\Sigma^{-1/2}(\tilde{\varepsilon}_{J_{0}-1,0},\ldots,\tilde{\varepsilon}_{J-1,2^{J-1}-1})^{\top} \sim \mathcal{N}(0,I_{p})$. 
For this model, the OLS estimator is  $\hat{\beta} = (X^{\top}X)^{-1} X^{\top} Y = (\kappa_{l,k}^{-1}\tilde{Y}_{l,k})_{(l,k) \in \mathcal{I}(J)}$, and so
\[
\Pr\left(\max_{(l,k)\in\mathcal{I}(J)} |\kappa_{l,k}^{-1}\tilde{Y}_{l,k} - \beta_{0,l,k} |/w_{l} \leq \hat{R}_{\alpha} \right)
=
\Pr(\beta_{0} \in I(\hat{\beta},\hat{R}_{\alpha}))
\]
with weights $w_{l,k} = w_{l}$ for $(l,k) \in \mathcal{I}(J)$. 
Thus we can apply Theorem \ref{theorem: frequentist validation of Bayesian credible rectangles} to obtain bounds on 
$\Pr ( \max_{(l,k)\in\mathcal{I}(J)} | \kappa^{-1}_{l,k} \tilde{Y}_{l,k} - \beta_{0,l,k} | / w_{l} \leq \hat{R}_{\alpha} )$
and $\hat{R}_{\alpha}$.
It remains to bound $\Pr(\tilde{Y}_{\infty}\not\in\tilde{H}'_{3})$, which is similar to the final step of the proof of Proposition 
\ref{theorem: frequentist validation of undersmoothed credible bands in ill-posed problems}. 
The detail can be found in Appendix C.3 of \cite{supplement}.
\end{proof}

\begin{remark}[Coverage error rates in linear inverse problems]
Consider a locally log-Lipschitz prior with  locally log-Lipschitz constant $L=L_{n}$.
We assume a true function $f_{0}$ with $\|f_{0}\|_{B^{s}_{\infty,\infty}}\leq B$ for some $B=B_{n}$.
Set $J$ as follows: for a (positive) constant $c$ with  $c < 1/(2r)$,
\begin{align*}
	2^{J} = 
	\begin{cases}
	(n /\log n)^{1/(2s+2r+1)} &\text{in mildly ill-posed cases (Case M)}; \\
         c\log n &\text{in severely ill-posed cases (Case S)}.
	\end{cases}
\end{align*}
Set $w_{l}=(\max_{0\leq k \leq 2^{l}-1}\kappa_{l,k})^{-1}\sqrt{l}$ for $l\leq J-1$ and 
$w_{l} = u_{l} (\min_{0 \leq k \leq 2^{l}-1}\kappa_{l,k})^{-1}\sqrt{l}$ for $l\geq J$ with $u_{l}\uparrow\infty$ as $l\to\infty$.
Then we have
\begin{align*}
&| \Pr  ( f_{0}\in \mathcal{C}_{w}(\hat{f} ,\hat{R}_{\alpha}) ) - (1-\alpha) |
\leq
	\begin{cases}
		O(L_{n}(n/\log n)^{-s/(2s+2r+1)}) &\text{in Case M} \\
		O(L_{n}(\log n)^{-s}) &\text{in Case S}
	\end{cases}
\quad \text{and} \\
&\sup_{f,g \in \mathcal{C}_{w}^{B_{n}} (\hat{f}_{\infty},\hat{R}_{\alpha})} \| f-g \|_{\infty}  \leq
	\begin{cases}
		O(B_{n} (n / \log n )^{-s/(2s+2r+1)}) &\text{in Case M} \\
		O(B_{n} (\log n)^{-s}) &\text{in Case S}
	\end{cases}
,
\end{align*}
where the latter holds 	
with probability at least $1-c_{1} n^{ - c_{2}  2^{J}}$ (again the sequence $\{ w_{l} \}$ here depends on $n$ but we can apply Proposition 3.3; see Remark C.2 in \cite{supplement} for the detail).
\end{remark}

\subsection{Nonparametric regression model}
\label{section: credible bands in nonparametric regression model}

Finally we consider a nonparametric regression model
\begin{align*}
	Y_{i}=f_{0}(T_{i})+\varepsilon_{i},\ i=1,\ldots, n,
\end{align*}
where
$\varepsilon= (\varepsilon_{1} , \ldots, \varepsilon_{n})^{\top}$ is the vector of i.i.d.~error terms with mean zero and variance $\sigma^{2}_{0}$
and $T_{1},\dots,T_{n}$ are an i.i.d.~sample with values in $[0,1]$.
For simplicity, we assume that 
$\varepsilon$ and $\{T_{i}:i=1,\ldots,n\}$ are independent,
and
$\sigma_{0}$ does not depend on $n$.

We consider a sieve prior for $f_{0}$. To this end, we use $p$ basis functions $\{\psi^{p}_{j}(\cdot) : 1 \leq j \leq p \}$,
and constrict a credible band for $f$ of the form
\begin{align*}
	\mathcal{C}(\hat{f},\hat{R}_{\alpha}) = \left\{ f : \left\| \frac{f(\cdot)-\hat{f}(\cdot)}{\|\psi^{p}(\cdot)\|} \right\|_{\infty} 
				      \leq \hat{R}_{\alpha} \right\},
\end{align*}
where
$\hat{f}(\cdot):= \sum_{j=1}^{p}\psi^{p}_{j}(\cdot)\hat{\beta}_{j}$ 
with $\hat{\beta} := \mathrm{argmin}_{\beta} \sum_{i=1}^{n} ( Y_{i}- \sum_{j=1}^{p} \psi^{p}_{j} (T_{i}) \beta_{j} )^{2}$,
$\hat{R}_{\alpha}$ is chosen in such a way that
$\Pi_{f}\{ \mathcal{C}(\hat{f} , \hat{R}_{\alpha}) \mid Y \} = 1-\alpha $,
and
$\psi^{p}(\cdot):=(\psi^{p}_{1}(\cdot),\ldots, \psi^{p}_{p}(\cdot))^{\top}$.
We consider a prior $\Pi_{f}$ of $f$ induced from a sieve prior $\Pi_{\beta}$ on $\R^{p}$
via the map
$ (\beta_{1}, \ldots, \beta_{p}) \mapsto \sum_{j=1}^{p}\beta_{j}\psi^{p}_{j}(\cdot). $

The setting of the nonparametric regression is different from that of Section \ref{section: frequentist validation of Bayesian credible rectangles} in that the regressors $T_{1},\dots,T_{n}$ are stochastic. 
Due to this additional randomness, we need an additional analysis
to develop bounds on the coverage error and the $L^{\infty}$-diameter of the band.
To this end,
we modify Conditions \ref{Condition: prior mass condition} and \ref{Condition: for simplicity},
and add conditions on the basis functions
Let
$\tilde{\psi}^{p}(\cdot):= \psi^{p}(\cdot) / \| \psi^{p}(\cdot) \| $,
$\xi_{p}:=\sup_{t\in [0,1]}\|\psi^{p}(t)\|$, and 
 $\beta_{0} := \mathrm{argmin}_{\beta} \Ep [( f_{0}(T_{1}) - \psi^{p}(T_{1})^{\top}\beta )^{2}]$.
For $R>0$,
let 
\begin{align*}
	\tilde{B}(R) := \{\beta: \| \beta - \beta_{0} \| \leq n^{-1/2} R \} \ \text{ and } \
	\tilde{\phi}_{\Pi_{\beta}}(R) := 1- \inf_{\beta , \tilde{\beta} \in \tilde{B}(R) } \frac{\pi(\beta)}{\pi(\tilde{\beta})}.
\end{align*}

\begin{condition}\label{Condition: prior mass condition 2}
	There exists a positive constant $C_{1}$ such that $ \pi(\beta_{0}) \geq n^{ - C_{1} p }$.
\end{condition}

\begin{condition}\label{Condition: for simplicity 2}
	The inequality $\tilde{\phi}_{\Pi_{\beta}} (1/\sqrt{n}) \leq 1/2$ holds.
\end{condition}

\begin{condition}\label{Condition: basis eigenvalues}
	There exist strictly positive constants $\underline{b}$ and $\overline{b}$ such that
	the eigenvalues of the $p\times p$ matrix $( \Ep [\psi^{p}_{i}(T_{1}) \psi^{p}_{j}(T_{1})] )_{1 \le i,j \le p}$ are 
	included in $[\underline{b}^{2} , \overline{b}^{2}]$.
\end{condition}

\begin{condition}\label{Condition: basis conditions}
	There exist positive constants $C_{4}$ and $C_{5}$ such that  
\begin{align*}
	\log \xi_{p} \leq C_{4} \log p \ \text{ and } \
	\log \sup_{t\neq t' \in [0,1]}\frac{\| \tilde{\psi}^{p}(t) -\tilde{\psi}^{p}(t')\|}{|t-t'|} \leq C_{5} \log p.
\end{align*}
\end{condition}

Conditions \ref{Condition: prior mass condition 2} and \ref{Condition: for simplicity 2}
are versions of Conditions \ref{Condition: prior mass condition} and \ref{Condition: for simplicity} under stochastic regressors. 
Condition \ref{Condition: basis eigenvalues} is standard.
Condition \ref{Condition: basis conditions} is not restrictive; 
for example,
this condition holds for Fourier series, Spline series, CDV wavelets, and local polynomial partition series;
see \cite{BelloniChernozhukovChetverikovKato(2015)} for details.

The following proposition derives bounds on the coverage error and the $L^{\infty}$-diameter of $\mathcal{C}(\hat{f}, \hat{R}_{\alpha})$.
Let
$\tau_{2} := \sqrt{ \Ep[ (f_{0}(T_{1}) - \psi^{p}(T_{1})^{\top}\beta_{0} )^2] }$, 
$\tau_{\infty} := \| f_{0}(\cdot) - \psi^{p}(\cdot)^{\top} \beta_{0} \|_{\infty}$, and 
$\tau := \left\| | f_{0}(\cdot) -\psi^{p}(\cdot)^{\top} \beta_{0}|  / \|\psi^{p}(\cdot)\| \right\|_{\infty}$.
These parameters quantify the approximation errors by the basis functions.

\begin{proposition}\label{proposition: credible bands in nonparametric regression model with general basis functions}
Under Conditions \ref{Condition: prior mass condition 2}-\ref{Condition: basis conditions} together with Conditions \ref{Condition: marginal posterior contraction} and \ref{Condition: moment condition},
there exist positive constants $c_{1},c_{2}$ depending only on 
$C_{1},\ldots,C_{5}$, $\underline{b}, \overline{b}$, and $q$ appearing in these conditions 
such that the following hold. 
For $n \geq 2$ and any sufficiently small $\delta>0$,
we have
\begin{equation}
\begin{split}
&| \Pr  (f_{0} \in \mathcal{C}(\hat{f},\hat{R}_{\alpha}) ) - (1-\alpha) | \\
&\quad \leq \tilde{\phi}_{\Pi_{\beta}}\left( c_{1} \sqrt{p\log n}\right) + \delta_{2} + \delta_{3} 
+ c_{1} (n^{-2\delta} + \delta_{1}p\log n + \zeta_{n} + \gamma_{n}),
\end{split}
\label{eq: coverage error in nonparametric regression}
\end{equation}
where
\begin{align*}
\gamma_{n}
	&:= \frac{n}{\log n}\frac{\tau^{2}_{2}}{p}
	+ \max\left\{1, \left( p \xi_{p}^{2} / n \right)^{1/2}\right\} \tau_{\infty} n^{\delta}\log p
	+ \sqrt{n}\tau \sqrt{\log p} \quad \text{and} \\
	\zeta_{n}
	&:=
	\begin{cases}
	n^{\delta}(\log n)^{7/6}
	\max\left\{\left( \frac{\xi_{p}^{2}}{n} \right)^{1/2}n^{1/q}(\log n)^{1/3}
		, \left( \frac{\xi_{p}^{2}}{n} \right)^{1/6}
	\right\}
	& \text{ under Condition \ref{Condition: moment condition} (a)}\\
	n^{\delta}(\log n)^{7/6}\left( \xi_{p}^{2} / n\right)^{1/6}
	& \text{ under Condition \ref{Condition: moment condition} (b)}
	\end{cases}
	.
\end{align*}
In addition,
there exists a positive constant $c_{3}$ depending only on $\alpha$ and $\underline{b}$ such that the following holds:
provided that the right hand side on (\ref{eq: coverage error in nonparametric regression}) is smaller than $\alpha/2$,
we have
\[
\sup_{f,g \in \mathcal{C}(\hat{f},\hat{R}_{\alpha})} \| f - g \|_{\infty} \leq c_{3}  \sqrt{ \xi_{p}^{2}(\log p) /n}
\]
with probability at least $1-\delta_{3}-c_{1}\{\sqrt{n}\tau \sqrt{\log p} + n^{-c_{2}p}\}$.
\end{proposition}

We note that the proof of Proposition \ref{proposition: credible bands in nonparametric regression model with general basis functions} does not use a lower bound on $\hat{R}_{\alpha}$ in Theorem \ref{theorem: frequentist validation of Bayesian credible rectangles} (more precisely, its version for random designs). Hence we do not have to assume
that the right hand side on (\ref{eq: coverage error in nonparametric regression}) is smaller than $(1-\alpha)/2$;
see the discussion after Theorem \ref{theorem: frequentist validation of Bayesian credible rectangles}.

\begin{remark}[Magnitudes of $\xi_{p}$, $\tau_{2}$, $\tau_{\infty}$, and $\tau$]
For typical basis functions including Fourier series, spline series, and CDV wavelets,
we have
$\xi_{p}\lesssim \sqrt{p}$; see Section 3 in \cite{BelloniChernozhukovChetverikovKato(2015)}.
If $f_{0}$ is in the H\"{o}lder--Zygmund space with  smoothness level $s > 0$, then
$\tau_{2}\sim\tau_{\infty}\sim p^{-s}$ for an $S$-regular CDV wavelet basis with $S > s$. 
For other bases and other function classes,
bounds on $\tau_{2}$ and $\tau_{\infty}$ can be found in approximation theory; 
see e.g. \cite{DevoreandLorentz(1993)} and Section 3 in \cite{BelloniChernozhukovChetverikovKato(2015)}.
Finally, for the Haar wavelet basis,
we have $\tau\sim \tau_{\infty}/\sqrt{p}$,
since $\tau \leq \tau_{\infty} / \inf_{t\in [0,1]} \|\psi^{p}(t)\|$;
for periodic $S$-regular wavelets,
we also have $\tau \sim \tau_{\infty}/\sqrt{p}$ as shown in Appendix C.4.3 of  \cite{supplement}.
\end{remark}

\begin{remark}[Coverage error rates for the true function]
Consider the unknown variance case. 
Assume that there exists a constant $s > 1/2$ such that
$\tau_{2} \sim \tau_{\infty} \sim p^{-s}$,
$\tau \sim p^{-s-1/2}$,
and
$\xi_{p} \lesssim \sqrt{p}$.
Assume also that the error distribution is Gaussian (for the non-Gaussian case, add $\zeta_{n}$ to the bound on the coverage error). 
We use a locally log-Lipschitz prior with locally log-Lipschitz constant $L=L_{n}$ on $\beta$
and use the estimator $\hat{\sigma}^2=\hat{\sigma}^{2}_{\mathrm{u}}$ as in Proposition 2.3.
Take $p \sim (n /\log n)^{1/(2s+1)} b_{n}$ with a positive nondecreasing sequence $b_{n}=O(\log n)$.
In this case, we have
\begin{align*}
&| P(f_{0} \in \mathcal{C}(\hat{f} , \hat{R}_{\alpha}) ) - (1-\alpha) |  \\
&\leq C
	\left [ L_{n}\left(\frac{n}{\log n}\right)^{-s/(2s+1)}b_{n}^{1/2}
	+
	\left(\frac{n}{\log n} \right)^{-(s-1/2)/(2s+1)} b_{n}\log n
	+\frac{\log n}{b_{n}^{s+1/2}}\right ] \quad \text{and}  \\
&\sup_{f,g \in \mathcal{C}(\hat{f},\hat{R}_{\alpha})} \| f - g \|_{\infty}
	\leq C \left(\frac{n}{\log n}\right)^{-s/(2s+1)}b_{n}^{1/2},
\end{align*}
where the latter holds with  probability at least $1-c_{1}(\log n)/b_{n}^{s+1/2}$, and the constant $C$ is independent of $n$. 
\end{remark}

\begin{remark}[Coverage error rates for the surrogate function]
\label{remark: coverage errors for the surrogate function}
Consider coverage errors for the surrogate function $f_{0,p}:=\psi^{p}(\cdot)^{\top}\beta_{0}$ when the error distribution is Gaussian.
In this case, since $\tau_{\infty}=\tau_{2}=\tau=0$,  we have
\begin{align*}
	| \Pr ( f_{0,p} \in \mathcal{C}(\hat{f}, \hat{R}_{\alpha} )- (1-\alpha) | 
	&\leq O((n /\log n)^{-(s-1/2)/(2s+1)} b_{n} \log n) \quad \text{and} \\
	\sup_{f,g \in \mathcal{C}(\hat{f},\hat{R}_{\alpha})} \| f - g \|_{\infty}
	&\leq O((n / \log n )^{-s/(2s+1)}b_{n}^{1/2})
\end{align*}
where the latter holds with probability at least $1-c_{1}\exp\{ -c_{2} (n/\log n)^{1/(2s+1)}\}$.
This shows that Bayesian credible bands have coverage errors (for the surrogate function) 
decaying polynomially fast in the sample size $n$ in nonparametric regression models.
\end{remark}

\section{Proof of Theorem \ref{theorem: frequentist validation of Bayesian credible rectangles}}
\label{section: proof of frequentist validation of Bayesian credible rectangles}

\subsection{Supporting lemmas}

We begin with stating some supporting lemmas that will be used in the proof of Theorem \ref{theorem: frequentist validation of Bayesian credible rectangles}. 
They include
the high-dimensional CLT on hyperrectangles,
the anti-concentration inequality for the Gaussian distribution, 
Anderson's lemma,
and
the concentration inequality for the Gaussian maximum.

The high-dimensional CLT on hyperrectangles is stated as follows:
in the following lemma,
let $Z_{1},\ldots,Z_{n}$ be independent $p$-dimensional random vectors with mean zero.
Let $Z_{ij}$ ($i=1,\ldots,n$,$j=1,\ldots,p$) denote the $j$-th coordinate of $Z_{i}$.
Let $\tilde{Z_{1}},\ldots,\tilde{Z_{n}}$ be independent centered $p$-dimensional Gaussian vectors 
such that each $\tilde{Z_{i}}$ has the same covariance matrix as $Z_{i}$.
Let
$\mathcal{A}^{\mathrm{re}}$ be the class of all closed hyperrectangles in ${\R}^{p}$:
for any $A\in\mathcal{A}^{\mathrm{re}}$,
$A$ is of the form
$A=\{\beta\in\mathbb{R}^{p}: \underline{a}_{i}\leq \beta_{i}\leq \overline{a}_{i},\, 1\leq \forall i \leq p\}$
with $(\underline{a}_{1}, \ldots, \underline{a}_{p})^{\top} \in \R^{p}$ and $(\overline{a}_{1},\ldots, \overline{a}_{p})^{\top} \in \R^{p}$.
We assume the following three conditions:
\begin{enumerate}
	\item[H1.]\label{H1} There exists $ b >0$
		such that $n^{-1}\sum_{i=1}^{n}\Ep [Z_{ij}^{2}] \geq b$ for all $1\leq j \leq p$;
	\item[H2.]\label{H2} There exists a sequence $B_{n}\geq 1$ 
		such that $n^{-1}\sum_{i=1}^{n}\Ep[|Z_{ij}|^{2+k}]\leq B_{n}^{4}$ for all $1\leq  j \leq p$ and for $k=1,2$;
	\item[H3.]\label{H3} Either one of the following two conditions holds:
			\begin{enumerate}
			\item[(a)] There exists an integer $4\leq q < \infty$ 
				such that $\Ep[( \max_{1 \le j \le p}|Z_{ij}|/B_{n} )^{q}]\leq 1$ 
				for all $1\leq i \leq n$;
			\item[(b)] $\Ep[\exp ( |Z_{ij}|/B_{n})]\leq 2$ 
				for all $1\leq i \leq n$ and $1\leq j \leq p$.
			\end{enumerate}
\end{enumerate}

\begin{lemma}[High dimensional CLT on hyperrectangles; Proposition 2.1 in \cite{CCK(2016)}]\label{lemma: high dimensional CLT}
Let 
\begin{align*}
	\rho= \rho_{n} :=\sup_{A\in\mathcal{A}^{\mathrm{re}}}
	\left|\Pr \left(\sum_{i=1}^{n} Z_{i}/\sqrt{n}\in A\right)
	-\Pr\left(\sum_{i=1}^{n}\widetilde{Z_{i}}/\sqrt{n}\in A \right) \right|.
\end{align*}
Under Conditions H1-H3,
there exists a positive constant $\tilde{c}_{1}$ such that 
\begin{align*}
	\rho \leq
	\begin{cases}
		\tilde{c}_{1} \left( \frac{B_{n}^{2}\log^{7}(pn)}{ n} \right)^{1/6}
		+ \tilde{c}_{1} \left(  \frac{B_{n}^{2}\log^{3}(pn)} { n^{1-2/q}} \right)^{1/3}
	& \text{ under Condition H3 (a)}, \\
		\tilde{c}_{1}\left( \frac{B_{n}^{2}\log^{7}(pn)}{ n } \right)^{1/6}
	& \text{ under Condition H3 (b)}.
	\end{cases}
\end{align*}
The constant 
$\tilde{c}_{1}$
depends
only on $b$ appearing in Condition H1 and $q$ appearing in Condition H3.
\end{lemma}

Next we state 
the anti-concentration inequality for the Gaussian distribution, 
Anderson's lemma,
and the concentration inequality for the Gaussian maximum. 
\begin{lemma}[Anti-concentration inequality for the Gaussian distribution; \cite{Nazarov(2003)}]\label{lemma: anti-concentration}
Let $Z=(Z_{1},\ldots,Z_{p})^{\top}$ be a centered  Gaussian random vector in $\R^{p}$ with
$\sigma_{j}^{2}:=\Ep [Z_{j}]^{2}>0$ for all $1 \leq j \leq p$.
Let $\underline{\sigma}:=\min\{\sigma_{j}\}$.
There exists a universal positive constant $\tilde{c}_{2}$
such that 
for every $z = (z_{1},\ldots, z_{p})^{\top} \in \R^{p}$ and $R>0$,
\begin{align*}
	\gamma:=\gamma(R):=&\Pr ( Z_{j} \leq z_{j} + R \  1 \le \forall j \le p ) 
	- \Pr (Z_{j} \leq z_{j}  \  1 \le \forall j \le p  )
\le \tilde{c}_{2} \frac{R}{\underline{\sigma}} \sqrt{\log p}.
\end{align*}
\end{lemma}
\begin{lemma}[Anderson's lemma; Corollary 3 in \cite{Anderson(1955)}]\label{lemma: Andersons lemma}
Let $\Sigma$ and $\tilde{\Sigma}$ 
be symmetric positive semidefinite $p\times p$ matrices, and 
let $\mathcal{C}$ be a symmetric convex  set in $\R^{p}$.
If $\Sigma-\tilde{\Sigma}$ is positive semidefinite,
then $\Pr ( Z \in \mathcal{C} ) \leq \Pr ( \tilde{Z} \in \mathcal{C} )$ for $Z \sim \mathcal{N}(0,\Sigma)$ and $\tilde{Z} \sim \mathcal{N}(0,\tilde{\Sigma})$,
\end{lemma}
\begin{lemma}[Concentration inequality for the Gaussian maximum; Theorem 2.5.8. in \cite{GineandNickl_book}]\label{lemma: concentration inequality for Gaussian maxima}
	Let $N_{1},\dots,N_{p} \sim \mathcal{N}(0,1)$ i.i.d. and 
	let $\{w_{i} \}_{i = 1}^{p}$ be a positive sequence with $\downw=\min_{1 \le i \le p}w_{i}$.
	Then for every $R > 0$, 
	\begin{align*}
	\Pr \Big{(} \Big{|} \max_{1 \le i \le p} |N_{i}/w_{i}| - \Ep \Big{[}\max_{1 \le i \le p} |N_{i}/w_{i}| \Big{]} \Big{|}
	\geq R  \Big{)} \leq  2 \exp( - \downw^{2} R^{2} / 2 ).
	\end{align*}
\end{lemma}

\subsection{Proof of Theorem \ref{theorem: frequentist validation of Bayesian credible rectangles}}

We only prove the theorem under Condition \ref{Condition: moment condition} (a).
The proof under Condition \ref{Condition: moment condition} (b) is done by replacing
Lemma \ref{lemma: high dimensional CLT} (a)
by Lemma \ref{lemma: high dimensional CLT} (b).

The proof is divided into two parts.
We first derive an upper bound on the coverage error $|\Pr(\beta_{0}\in I(\hat{\beta}(Y),\hat{R}_{\alpha}))-(1-\alpha)|$ and then bound 
the radius $\hat{R}_{\alpha}$ of $I(\hat{\beta}(Y),\hat{R}_{\alpha})$.

\subsubsection*{Step 1: Upper bound on the coverage error}
We start with proving that  $\hat{R}_{\alpha}$ concentrates on the $(1-\alpha)$-quantile of some distribution 
with high probability. 
Let $\overline{\zeta}$ be the upper bound 
in Proposition \ref{proposition: Berry Esseen type bound on posterior distribution}.
From Proposition \ref{proposition: Berry Esseen type bound on posterior distribution},
we have
\begin{align*}
|\underbrace{\Pi_{\beta}(I(\hat{\beta}(Y),\hat{R}_{\alpha}) \mid Y)}_{=1-\alpha}-\mathcal{N}(I(\hat{\beta}(Y),\hat{R}_{\alpha}) \mid \hat{\beta}(Y),\sigma_{0}^{2}(X^{\top}X)^{-1})|
\leq \overline{\zeta}\
\text{ for $Y \in H$},
\end{align*}
where recall that $H = \{ Y : \| X (\hat{\beta}(Y) - \beta_{0} ) \| \leq c_{1}\sqrt{p\log n} \sigma_{0}/4 \} \cap \{Y : \Pi_{\sigma^{2}}(|\sigma^{2}/\sigma_{0}^{2} -1 | \geq \delta_{1} \mid Y) \leq \delta_{2} \}$. 
Let $\tilde{S} \sim \mathcal{N}(0,(X^{\top}X)^{-1})$ and 
let $G$ be the  distribution function of
$\sigma_{0} \max \{ |e_{(p),i}^{\top} \tilde{S} | / w_{i} \}$,
where 
$e_{(p),i}$ is the $p$-dimensional unit vector whose $i$-th component is 1.
Now since $\mathcal{N}(I(\hat{\beta}(Y),\hat{R}_{\alpha}) \mid \hat{\beta}(Y),\sigma_{0}^{2}(X^{\top}X)^{-1})
        = G(\hat{R}_{\alpha})$, 
we have $|(1-\alpha)-G(\hat{R}_{\alpha})|\leq \overline{\zeta}$ for $Y\in H$.
This implies
\begin{align}
	G^{-1}(1-\alpha-\overline{\zeta})\leq \hat{R}_{\alpha} \leq G^{-1}(1-\alpha+\overline{\zeta}) 
	\ \text{ for $Y\in H$},
	\label{eq: evaluate hat R}
\end{align}
where $G^{-1}$ denotes the quantile function of $G$.

Next, 
we will derive an upper bound on $\Pr(\beta_{0}\in I(\hat{\beta}(Y),\hat{R}_{\alpha}))-(1-\alpha)$ (the lower bound follows similarly). 
Let $\rho$ be the constant in Lemma \ref{lemma: high dimensional CLT} 
when
$Z_{j} = n(X^{\top}X)^{-1}X_{j\cdot}\varepsilon_{j}$ for $j=1,\ldots,n$,
where $X_{j\cdot} = (X_{j1},\ldots,X_{jp})^{\top}$ for $j=1,\ldots,n$.
For $R>0$,
let $\gamma(R)$ be the constant in Lemma \ref{lemma: anti-concentration} 
when 
$Z = \sigma_{0}\tilde{S}$.
Finally, let $\tilde{r}:=(X^{\top}X)^{-1}X^{\top}r$.
From inequality (\ref{eq: evaluate hat R}) and
by the definitions of $\rho$, $G$, and $\gamma$,
we have
\begin{align*}
&\Pr(\beta_{0}\in I(\hat{\beta}(Y),\hat{R}_{\alpha}))-(1-\alpha)
\\
&\leq
\Pr\Big{(} \max_{1 \le i \le p}\{|e_{(p),i}^{\top}(X^{\top}X)^{-1}X^{\top}(\varepsilon + r) |/w_{i}\}
\leq G^{-1}(1-\alpha+\overline{\zeta})\Big{)}
-(1-\alpha) + \Pr(Y \not\in H)
\\
&\leq
\Pr\Big{(} \max_{1 \le i \le p}\{|e_{(p),i}^{\top}(\sigma_{0}\tilde{S}+\tilde{r}) | /w_{i}\}
\leq G^{-1}(1-\alpha+\overline{\zeta})\Big{)}
-(1-\alpha) + \rho +  \Pr(Y \not\in H)
\\
&\leq
\gamma( \| \tilde{r} \|_{\infty} ) + \overline{\zeta} +  \rho + \Pr( Y \not\in H).
\end{align*}
Proposition \ref{proposition: test set} gives an upper bound on $\Pr(Y \not\in H)$.
From Lemmas \ref{lemma: high dimensional CLT} and \ref{lemma: anti-concentration},
we obtain the following bounds on $\rho$ and $\gamma$:
For some $\tilde{c}_{1}>0$ depending only on $q$,
\begin{align*}
	\rho \leq \tilde{c}_{1} \left\{
	\left( \frac{ p \log^{7}(pn)}{ n } \frac{ \uplambda } {\downlambda } \right)^{1/6}
+  \left( \frac{p\log^{3}(pn)}{ n^{1-2/q}} \frac{\uplambda }{\downlambda } \right)^{1/3}\right\}
\text{ and }
\gamma \leq \tilde{c}_{1} \frac{ \|\tilde{r}\|_{\infty}  }{ \sigma_{0} \downlambda^{1/2} }\sqrt{\log p},
\end{align*}
which completes Step 1. 

\subsubsection*{Step 2: Upper bound on the max-diameter}
We start with deriving a high-probability upper bound on $\hat{R}_{\alpha}$
using the quantile function $F^{-1}$ of $\max_{1 \le i \le p}|N_{i}/w_{i}|$
for  independent standard Gaussian random variables $\{N_{i}:i=1,\ldots,p\}$.
From Lemma \ref{lemma: Andersons lemma}, we have 
\begin{align*}
\Pr \Big{(} \max_{1 \le i \le p}|N_{i}/w_{i}|\leq R \Big{/} \Big{(}\sigma_{0}\overline{\lambda}^{1/2}\Big{)} \Big{)}
\leq \Pr \Big{(} \max_{1 \le i \le p} | \sigma_{0}\tilde{S}_{i} / w_{i} | \leq R \Big{)} \text{ for $R>0$.}
\end{align*}
Together with inequality  (\ref{eq: evaluate hat R}), we have
\begin{align}
\hat{R}_{\alpha}
\leq 
\sigma_{0}\overline{\lambda}^{1/2}F^{-1}(1-\alpha+\overline{\zeta})\
\text{ for $Y\in H$}.
\label{eq: evaluate hat R 2}
\end{align}

Next, 
we will bound
$F^{-1} (1-\alpha + \overline{\zeta} ) / \Ep[\max_{1 \le i \le p}|N_{i}/w_{i}|]$.
From Lemma \ref{lemma: concentration inequality for Gaussian maxima},
there exists $\tilde{c}_{2}>1$ depending only on $\alpha$ and $\downw$ such that 
\begin{align*}
\Pr&\Big{(}\max_{1 \le i \le p}|N_{i}/w_{i}| -\Ep\Big{[}\max_{1 \le i \le p}|N_{i}/w_{i}|\Big{]}
          \geq \tilde{c}_{2} \Ep\Big{[}\max_{1 \le i \le p}|N_{i}/w_{i}| \Big{]}\Big{)}
	   <\alpha - \alpha/2 < \alpha -\overline{\zeta}.
\end{align*}
Therefore, by the definition of $F^{-1}$, we have
\begin{align*}
F^{-1} (1-\alpha + \overline{\zeta}) 
&=\inf \Big{\{}R:\Pr \Big{(}\max_{1 \le i \le p}|N_{i}/w_{i} | \geq R \Big{)} \leq \alpha - \overline{\zeta} \Big{\}} \\
&\leq (1+\tilde{c}_{2}) \Ep\Big{[}\max_{1 \le i \le p}|N_{i}/w_{i}|\Big{]}.
\end{align*}
Together with (\ref{eq: evaluate hat R 2}), 
we obtain the desired upper bound on $\hat{R}_{\alpha}$.

\subsubsection*{Step 3: Lower bound on the max-diameter}
As in Step 2, 
we have
\begin{align}
\sigma_{0}\overline{\lambda}^{1/2}\upw^{-1} \tilde{F}^{-1}(1-\alpha-\overline{\zeta})
\leq \hat{R}_{\alpha}
\text{ for $Y\in H$}
\label{eq: evaluate hat R 3}
\end{align}
Next, we will show that 
$\tilde{F}^{-1} ( 1-\alpha - \overline{\zeta} ) \geq \tilde{c}_{3} \sqrt{\log p}$ for some constant $\tilde{c}_{3}$ depending only on $\alpha$.
From the Paley--Zygmund inequality,
we have
for $\theta\in [0,1]$,
\begin{align}
\Pr \Big{(}\max_{1 \le i \le p} |N_{i}| \geq \theta \Ep \Big{[}\max_{1 \le i \le p} |N_{i}|\Big{]} \Big{)} 
	\geq
	(1-\theta)^{2} 
	\frac{(\Ep[\max_{1 \le i \le p}|N_{i}|])^{2}}
	{\Ep[ (\max_{1 \le i \le p}|N_{i}|)^{2} ]}.
	\label{eq: Paley-Zygmund}
\end{align}
From Lemma \ref{lemma: concentration inequality for Gaussian maxima}
together with the inequality $\Ep[\max_{1 \le i \le p}|N_{i}|]\geq \sqrt{\log p}/12$,
there exists a universal positive constant $\tilde{c}_{4}$ such that
\begin{align}
\Ep\Big{[}\Big{(}\max_{1 \le i \le p}|N_{i}|\Big{)}^{2}\Big{]}
\leq
\Big{(}\Ep\Big{[}\max_{1 \le i \le p}|N_{i}|\Big{]}\Big{)}^{2}(1+\tilde{c}_{4}/\sqrt{\log p}),
\label{eq: expectation of squared maxima}
\end{align}
where we have used use Lemma \ref{lemma: concentration inequality for Gaussian maxima}  to deduce that 
\begin{align*}
	\Ep \Big{[} \Big{(}\max_{1 \le i \le p}|N_{i}| \Big{)}^{2}\Big{]}
		&\leq
		\Big{(} \Ep \Big{[}\max_{1 \le i \le p}|N_{i}|\Big{]} \Big{)}^{2}
		+ 4\int_{\Ep[\max_{1 \le i \le p}|N_{i}|]}^{\infty} t \mathrm{e}^{-\Big{(}t-\Ep\Big{[} \max_{1 \le i \le p}|N_{i}|\Big{]}\Big{)}^2 /2 } dt
		\nonumber\\
		&\leq
		\Big{(} \Ep \Big{[}\max_{1 \le i \le p}|N_{i}|\Big{]}  \Big{)}^{2}
		+ \tilde{c}_{5} \Big{(}\Ep\Big{[} \max_{1 \le i \le p}|N_{i}|\Big{]}+1\Big{)}
\end{align*}
for some universal positive constant $\tilde{c}_{5}$. 
Let $\eta:=(1+\alpha)/2$.
Take $p$ such that $1 / \{1+\tilde{c}_{4}/\sqrt{\log p}\} \geq (\eta+1)/2$,
and take $\theta^{*}_{\alpha} = 1- \sqrt{ (2\eta) / (\eta+1) }$.
Then,
from inequalities  (\ref{eq: Paley-Zygmund}) and (\ref{eq: expectation of squared maxima}),
we have
\[\Pr \Big{(}\max_{1 \le i \le p} | N_{i} | \geq
\theta^{*}_{\alpha} \Ep \Big{[}\max_{1 \le i \le p} | N_{i} | \Big{]} \Big{)} 
    \geq (1-\theta^{*}_{\alpha})^{2}\frac{ (\Ep [\max_{1 \le i \le p}|N_{i}|] )^{2}}{ \Ep [(\max_{1 \le i \le p}|N_{i}|)^{2}]} 
    \geq \eta \geq  \alpha + \overline{\zeta}.\]
Thus we have
\begin{align}
\tilde{F}^{-1}(1 - \alpha - \overline{\zeta} ) \geq \theta^{*}_{\alpha} \Ep\Big{[}\max_{1 \le i \le p}|N_{i}|\Big{]} \geq (\theta^{*}_{\alpha}/12) \sqrt{\log p}.
\label{eq: lower estimate of tilde F}
\end{align}
Together with (\ref{eq: evaluate hat R 3}),
we obtain the desired lower bound on $\hat{R}_{\alpha}$.
\qed

\section{Conclusion}

We have studied finite sample bounds on frequentist coverage errors of Bayesian credible rectangles to approximately linear regression models with moderately high dimensional regressors. 
As an application, we have shown that 
Bayesian credible bands have coverage errors (for the true function) decaying polynomially fast in the sample size 
in Gaussian white noise models and linear inverse problems;
the similar results hold for the surrogate function in nonparametric regression models.
This supports the use of Bayesian approaches to constructing nonparametric confidence bands.


\section{Acknowledgement}
We are thankful to the Editor, the Associate Editor, and anonymous referees for their helpful comments.
K.~Yano is supported by the Grant-in-Aid for Young Scientists from the JSPS (19K20222) and
by JST CREST (JPMJCR1763).

\begin{supplement}[id=suppA]
\stitle{Supplement to ``On frequentist coverage errors of Bayesian credible sets in high dimensions"}
\sdatatype{.pdf}
\sdescription{The supplementary material contains the proofs omitted in the main text.}
\end{supplement}

\bibliographystyle{imsart-number}
\bibliography{BvM}

\newpage

\section*{Supplement to ``On frequentist coverage errors  of Bayesian credible sets in moderately high dimensions"}

This supplemental material is organized as follows:
Appendix A contains proofs of Propositions 2.5-2.6 in \cite{YanoKato}.
Appendix B contains proofs of Propositions 2.1-2.4 in \cite{YanoKato}.
Appendix C contains proofs for Section 3 in \cite{YanoKato}.
Appendices C.1-C.2 provide proofs of Propositions 3.1-3.2 in \cite{YanoKato}.
Appendix C.3 provides a proof of of Proposition 3.3 in \cite{YanoKato}.
Appendix C.4 provides a proof of Proposition 3.4 and a bound on $\tau$ in Remark 3.6 of \cite{YanoKato}.
Hereafter, 
the numbering for theorems, conditions, and propositions follows that of \cite{YanoKato}.

\appendix

\section{Proofs for Subsection 2.2}

In this section, we provide proofs of Propositions 2.5-2.6.

\subsection{Technical Lemmas}

We present here some technical lemmas that will be used to prove Proposition 2.5.

\begin{lemma}[Scheff\'{e}'s lemma]
\label{lemma: Scheffe}
Let $Q_{1}$ and $Q_{2}$ be probability measures on a measurable space with a common dominating measure $\mu$. 
Let $q_{1} = dQ_{1}/d\mu$ and $q_{2}=dQ_{2}/d\mu$. Then
\[
	\|Q_{1}-Q_{2}\|_{\mathrm{TV}}=\frac{1}{2}\int|q_{1}(x)-q_{2}(x)| d \mu(x)=\int(q_{1}(x)-q_{2}(x))_{+} d \mu(x),
\]
\end{lemma}

\begin{proof}
See, e.g., p.84 in \cite{Tsybakov(2009)}.
\end{proof}

\begin{lemma}[Posterior contraction of a marginal prior distribution]\label{lemma: posterior tail mass}
Recall that $B(R)=\{ \beta \in \R^{p} : \| X (\beta - \beta_{0}) \| \leq \sigma_{0} R \}$ for $R > 0$.
Under Conditions 2.1 and 2.3,
there exist positive constants $\tilde{c}_{1}$ and $\tilde{c}_{2}$ depending 
only on $C_{1}$ in Condition 2.1
such that
for a sufficiently large $R>0$,
the inequality
\[
\Pi_{\beta}(\beta\not\in B(R)\mid Y,\sigma^{2})
\leq
4\exp\{ \tilde{c}_{1}  p\log n - \tilde{c}_{2} (\sigma_{0}^{2}/\sigma^{2})R^{2} \}
\]
holds for $Y \in H$,
where recall that 
\[H:=\{Y: \|X(\hat{\beta}(Y)-\beta_{0})\|\leq R\sigma_{0}/4 \}\cap \{Y : \Pi_{\sigma^{2}}(|\sigma^{2}/\sigma_{0}^{2}-1| \geq \delta_{1} \mid Y) \leq \delta_{2} \}.\]
\end{lemma}

\begin{proof}
We use the following lower bounds on the small ball probability of a prior distribution:
\begin{lemma}[Lower bounds on the small ball probability of a prior distribution]
\label{lemma: prior small ball probability}
Let $\Pi_{\beta}$ be a probability measure with a density $\pi$ 
with respect to the $p$-dimensional Lebesgue measure.
Recall that $\phi_{\Pi_{\beta}}(R)=1-\inf_{\beta,\tilde{\beta}\in B(R)}\{\pi(\beta)/\pi(\tilde{\beta})\}$ for $R>0$.
Then, we have, for every $R > 0$,
\[
\Pi_{\beta}(\beta \in B(R))\geq 
	\frac{\{1-\phi_{\Pi_{\beta}}(R)\}(\pi \mathrm{e} R )^{p/2}}{2 (p/2+1)^{p/2+1/2}}\frac{\pi(\beta_{0})\sigma_{0}^{p}}{\sqrt{\mathrm{det}(X^{\top}X)}}.
\]
\end{lemma}
\begin{proof}[Proof of Lemma \ref{lemma: prior small ball probability}]
Observe that 
\[
	\Pi_{\beta} ( \beta \in B(R)) =\int_{B(R)} \pi(\beta) d \beta
\geq\mathop{\inf}_{\beta\in B(R)}\left\{\frac{\pi(\beta)}{\pi(\beta_{0})}\right\}\pi(\beta_{0})\int_{B(R)}\mathrm{d}\beta.
\]
Changing variables, we have that 
\[
\int_{B(R)}d\beta
= 
	\frac{(\sigma_{0}^{2}R^{2})^{p/2}}{\sqrt{\det(X^{\top}X)}}\int_{\|\beta\|\leq 1}d\beta
	=\frac{(\sigma_{0}^{2}R^{2})^{p/2}\pi^{p/2}}{\sqrt{\det(X^{\top}X)}\Gamma(p/2+1)},
\]
where $\Gamma(\cdot)$ is the Gamma function.
Using the bound 
\[
\Gamma(p/2+1)\leq \sqrt{2\pi} (p/2+1)^{p/2+1/2} \exp(-p/2-17/18)
\]
(e.g., seesection 5.6.1. in \cite{NIST}), we have that
\[
	\int_{B(R)}\mathrm{d}\beta\geq \frac{(\sigma_{0}^{2} \pi\mathrm{e} R^{2})^{p/2}\mathrm{e}^{17/18}}{\sqrt{2\pi}\sqrt{\det(X^{\top}X)}(p/2+1)^{p/2+1/2}}.
\]
Since $\mathrm{e}^{17/18}/\sqrt{2\pi}\geq 1/2$, we obtain the desired inequality. 
\end{proof}

Return to the proof of Lemma \ref{lemma: posterior tail mass}.
Letting $P:=X(X^{\top}X)^{-1}X^{\top}$, we have
\begin{align}
	\Pi_{\beta}( \beta \in B \mid Y,\sigma^{2})
	&=
	\frac{\int_{B^{\mathrm{c}}}\mathrm{e}^{-\langle P(\varepsilon+r),X(\beta-\beta_{0})\rangle/\sigma^{2}-\|X(\beta-\beta_{0})\|^{2}/(2\sigma^{2})}\pi(\beta)d\beta }
	{  \int \mathrm{e}^{-\langle P(\varepsilon+r),X(\beta-\beta_{0})\rangle/\sigma^{2}-\|X(\beta-\beta_{0})\|^{2}/(2\sigma^{2})}\pi(\beta)d\beta  }.
\label{eq: posterior tail mass mid}
\end{align}
Since $cx^{2}+c^{-1}y^{2}\geq 2xy$ for $x,y,c>0$,
we have, for any $c>1$,
\begin{align}
	\int_{B^{\mathrm{c}}}&\exp\{-\langle P(\varepsilon+r),X(\beta-\beta_{0})\rangle/\sigma^{2}-\|X(\beta-\beta_{0})\|^{2}/(2\sigma^{2})\}\pi(\beta)d\beta 
\nonumber\\
	&\leq \int_{B^{\mathrm{c}}}\exp\{\|P(\varepsilon+r)\|\|X(\beta-\beta_{0})\|/\sigma^{2}-\|X(\beta-\beta_{0})\|^{2}/(2\sigma^{2}) \}\pi(\beta)d\beta
\nonumber\\
	&\leq \int_{B^{\mathrm{c}}}\exp[ \{c\|P(\varepsilon+r)\|^{2}+c^{-1}\|X(\beta-\beta_{0})\|^{2}\}/(2\sigma^{2}) -\|X(\beta-\beta_{0})\|^{2}/(2\sigma^{2}) ]
\pi(\beta)d\beta
\nonumber\\
	&\leq \exp\{c\|P(\varepsilon+r)\|^{2}/(2\sigma^{2})-(1-c^{-1})(\sigma_{0}^{2}/\sigma^{2})R^{2}/2\}.
\label{eq: upper bound of numerator}
\end{align}
Letting $\tilde{R}=1/\sqrt{\pi\mathrm{e}n}$, we have
\begin{align}
	&\int \exp\{-\langle P(\varepsilon+r),X(\beta-\beta_{0})\rangle/\sigma^{2}-\|X(\beta-\beta_{0})\|^{2}/(2\sigma^{2})\}\pi(\beta)d\beta 
\nonumber\\
	&\geq \int_{B(\tilde{R})}\exp\{-\langle P(\varepsilon+r),X(\beta-\beta_{0})\rangle/\sigma^{2}-\|X(\beta-\beta_{0})\|^{2}/(2\sigma^{2})\}\pi(\beta)d\beta 
\nonumber\\
	&\geq \int_{B(\tilde{R})}\exp[-\{c\|P(\varepsilon+r)\|^{2}+c^{-1}\|X(\beta-\beta_{0})\|^{2}\}/(2\sigma^{2}) -\|X(\beta-\beta_{0})\|^{2}/(2\sigma^{2})]\pi(\beta)d\beta
\nonumber\\
	&\geq \exp \{ -c\|P(\varepsilon+r)\|^{2}/(2\sigma^{2})-(1+c^{-1})(\sigma_{0}^{2}/\sigma^{2})\tilde{R}^{2}/2  \} \Pi_{\beta}(B(\tilde{R})).
\label{eq: lower bound of denominator mid}
\end{align}

From (\ref{eq: lower bound of denominator mid}),
from Lemma \ref{lemma: prior small ball probability},
and from Condition 2.3,
we have
\begin{align}
	&\int \exp\{-\langle P(\varepsilon+r),X(\beta-\beta_{0})\rangle/\sigma^{2}-\|X(\beta-\beta_{0})\|^{2}/(2\sigma^{2})\}\pi(\beta)d\beta 
\nonumber\\
&\geq 
\frac{1-\phi_{\Pi_{\beta}}(\tilde{R})}{2} \exp\left\{p\log n /2 -p\log p - C_{1} p\log n
-c\frac{\|P(\varepsilon+r)\|^{2}}{2\sigma^{2}} -(1+c^{-1})\frac{\tilde{R}^{2}}{2}\frac{\sigma_{0}^{2}}{\sigma^{2}}\right\}
\nonumber\\
&\geq 4^{-1}\exp\left\{ p\log n /2 -p\log p -C_{1}p\log n
 - c\frac{\|P(\varepsilon+r)\|^{2}}{2\sigma^{2}} -(1+c^{-1})\frac{\sigma_{0}^{2}}{\sigma^{2}}\frac{\tilde{R}^{2}}{2} \right\},
\label{eq: lower bound of denominator}
\end{align}
where 
the first inequality follows from (\ref{eq: lower bound of denominator mid}) and from Lemma \ref{lemma: prior small ball probability}
and the second inequality follows from Condition 2.3.

Combining (\ref{eq: upper bound of numerator}) and (\ref{eq: lower bound of denominator}) with (\ref{eq: posterior tail mass mid}),
we have, for $Y \in H $,
\begin{align*}
&\int_{B^{\mathrm{c}}}\mathrm{e}^{-\|Y-X\beta\|^{2}/(2\sigma^{2})}\pi(\beta)d\beta \big{/} \int \mathrm{e}^{-\|Y-X\beta\|^{2}/(2\sigma^{2}) }\pi(\beta)d\beta 
\nonumber\\
&\leq
4\exp[(C_{1}+1/2)p\log n+\{(1+c^{-1})/(2n)\}(\sigma_{0}^{2}/\sigma^{2}) - \{(1-c^{-1})/2-c/16 \}(\sigma_{0}^{2}/\sigma^{2})R^{2} ].
\end{align*}
Taking $c=3$ completes the proof.
\end{proof}

\begin{lemma}\label{lemma: exponential inequality for quadratic forms}
Let $A$ be an $n \times n$ symmetric positive semidefinite matrix 
such that $\| A \|_{\op} \leq 1$ and $\rank (A) < n $.
Let $\varepsilon=(\varepsilon_{1},\ldots,\varepsilon_{n})^{\top}$ be a vector of i.i.d.~random variables with mean zero and unit variance.
\begin{enumerate}
\item[(a)] 
If in addition Condition 2.5 (a) holds for an integer $q \geq 2$ and $C_{3}>0$,
then
there exists a positive constant $\tilde{c}_{1}$ depending only on $q$ and $C_{3}$ such that,
for every  $R >\sqrt{\rank(A)}$,
\[
	\Pr \left (\varepsilon^{\top}A\varepsilon \geq R^{2} \right )\leq \tilde{c}_{1} \rank(A)/(R-\sqrt{\rank(A)})^{q}.
\]
\item[(b)] 
If instead Condition 2.5 (b) holds for $C_{3}>0$,
then
there exists a positive constant $\tilde{c}_{1}$ depending only on $C_{3}$ such that,
for every $R>0$,
\[
	\Pr\left (|\varepsilon^{\top}A\varepsilon-\Ep[\varepsilon^{\top}A\varepsilon]| > R^{2} \right )
		\leq 2 \exp\{-\tilde{c}_{1} \min \left( R^{4} / \|A\|^{2}_{\mathrm{HS}} , R^{2} \right)\},
\]
where $\|\cdot\|_{\mathrm{HS}}$ denotes the Hilbert--Schmidt norm.
\end{enumerate}
\end{lemma}
\begin{proof}
For Case (a), see Corollary 5.1 in \cite{Baraud(2000)}. 
The inequality in Case (b) is called the Hanson-Wright inequality; for a proof, we refer to \cite{HansonandWright(1971)} and \cite{RudelsonandVershynin(2013)}.
\end{proof}

\subsection{Proof of Proposition 2.5}

\subsubsection*{Notations}
We define additional notation before the proof.
Let $\widetilde{\mathcal{N}}:=\mathcal{N}(\hat{\beta}(Y),\sigma_{0}^{2}(X^{\top}X)^{-1})$.
Let $B:= B( c_{1} \sqrt{p\log n})$ and $H:= H(c_{1})$ for a sufficiently large $c_{1}>0$ depending on $C_{1}$ and $C_{2}$.
Let $\Pi_{\beta}^{B}(d\beta\mid Y)$ be the probability measure defined by
\begin{align*}
	\Pi_{\beta}^{B}(d\beta\mid Y) := 1_{\beta\in B}\Pi_{\beta}(d\beta\mid Y) \bigg{/} \int_{B}\Pi_{\beta}(d\tilde{\beta}\mid Y)
\end{align*}
and
let $\widetilde{\mathcal{N}}^{B}$ be the probability measure defined by
\begin{align*}
\widetilde{\mathcal{N}}^{B}(d\beta)
:=1_{\beta\in B}\widetilde{\mathcal{N}}(d\beta) \bigg{/} \int_{B}\widetilde{\mathcal{N}}(d\tilde{\beta}).
\end{align*}
Let $\Pi_{\beta}(\cdot \mid Y,\sigma^{2})$ be the distribution defined by
\begin{align*}
\Pi_{\beta}(d\beta\mid Y,\sigma^{2})
:=\mathrm{e}^{ -\|Y-X\beta\|^{2}/(2\sigma^{2}) }\pi(\beta)d\beta \bigg{/} \int \mathrm{e}^{ -\|Y-X\tilde{\beta}\|^{2}/(2\sigma^{2}) }\pi(\tilde{\beta})d\tilde{\beta}
\end{align*}
and
let $\Pi_{\beta}^{B}(\cdot\mid Y,\sigma^{2})$ be the distribution defined by
\begin{align*}
\Pi_{\beta}^{B}(d\beta\mid Y,\sigma^{2})
:=1_{\beta\in B}\mathrm{e}^{ -\|Y-X\beta\|^{2}/(2\sigma^{2})  }\pi(\beta)d\beta \bigg{/}
\int_{B} \mathrm{e}^{ -\|Y-X\tilde{\beta}\|^{2}/(2\sigma^{2}) }\pi(\tilde{\beta})d\tilde{\beta}.
\end{align*}
In the proof,
$\tilde{c}_{1} , \tilde{c}_{2}, \ldots $ are positive constants depending only on $C_{1}$, $C_{2}$, and $c_{1}$.

\subsubsection*{Proof sketch} 
We present a proof sketch ahead.
The triangle inequality gives
\begin{align}
\|\Pi_{\beta}(d\beta\mid Y)-\tilde{\mathcal{N}}\|_{\mathrm{TV}}
\leq
\|\Pi_{\beta}(d\beta\mid Y)-\Pi_{\beta}(d\beta\mid Y,\sigma_{0}^{2})\|_{\mathrm{TV}}
+
\|\Pi_{\beta}(d\beta\mid Y,\sigma_{0}^{2})-\tilde{\mathcal{N}}\|_{\mathrm{TV}}.
\label{eq: decomposition of target total variation}
\end{align}
Consider the first term on the right hand side of (\ref{eq: decomposition of target total variation}).
Let 
$S=S(\delta_{1}):=\left\{\sigma^{2}:|\sigma^{2}/\sigma^{2}_{0}-1\right|\leq \delta_{1}\}.$
From the application of Jensen's inequality to the function $x\to|x|$
and from Condition 2.2,
we have
\begin{align*}
	\|&\Pi_{\beta}(d\beta\mid Y)-\Pi_{\beta}(d\beta\mid Y,\sigma^{2}_{0})\|_{\mathrm{TV}}
	\\	
	&\leq \int_{S}\|\Pi_{\beta}(d\beta\mid Y,\sigma^{2})-\Pi_{\beta}(d\beta\mid Y,\sigma^{2}_{0})\|_{\mathrm{TV}} \Pi_{\sigma^{2}}(d\sigma^{2}\mid Y)
	+\delta_{2}
\end{align*}
with probability at least $1-\delta_{3}$.
Consider the first term on the rightmost hand in the above inequality.
The triangle inequality gives
\begin{align}
\int_{S}\|\Pi_{\beta}(d\beta\mid Y,\sigma^{2})-\Pi_{\beta}(d\beta\mid Y,\sigma^{2}_{0})\|_{\mathrm{TV}} \Pi_{\sigma^{2}}(d\sigma^{2}\mid Y)
\leq A_{1} + A_{2} + A_{3},
\label{eq: decomposition of marginal posterior}
\end{align}
where 
\begin{align*}
&A_{1}:=\int_{S}\|\Pi_{\beta}(d\beta\mid Y,\sigma^{2})-\Pi_{\beta}^{B}(d\beta\mid Y,\sigma^{2})\|_{\mathrm{TV}} \Pi_{\sigma^{2}}(d\sigma^{2}\mid Y),
\nonumber\\
&A_{2}:=
\int_{S}\|\Pi_{\beta}^{B}(d\beta\mid Y,\sigma^{2})-\Pi_{\beta}^{B}(d\beta\mid Y,\sigma^{2}_{0})\|_{\mathrm{TV}} \Pi_{\sigma^{2}}(d\sigma^{2}\mid Y),
\nonumber\\
&A_{3}:=
\int_{S}\|\Pi_{\beta}^{B}(d\beta\mid Y,\sigma^{2}_{0})-\Pi_{\beta}(d\beta\mid Y,\sigma^{2}_{0})\|_{\mathrm{TV}} \Pi_{\sigma^{2}}(d\sigma^{2}\mid Y).
\end{align*}
Upper bounds on $A_{1},A_{2},A_{3}$ will be presented in
(\ref{eq: bound 1 of the first term}),
(\ref{eq: bound 2 of the first term}), 
and (\ref{eq: bound 3 of the first term}),
respectively.
Consider the second term on the right hand side of (\ref{eq: decomposition of target total variation}).
The triangle inequality gives
\begin{align}
	\|\Pi_{\beta}(d\beta\mid Y,\sigma_{0}^{2})-\widetilde{\mathcal{N}}\|_{\mathrm{TV}}
	\leq
	A_{4} + A_{5} + A_{6}, 
	\label{eq: three steps}
\end{align}
where
$A_{4}:= \|\widetilde{\mathcal{N}}-\widetilde{\mathcal{N}}^{B}\|_{\mathrm{TV}},$
$A_{5}:= \|\widetilde{\mathcal{N}}^{B}-\Pi_{\beta}^{B}(d\beta\mid Y,\sigma_{0}^{2})\|_{\mathrm{TV}},$
and
$A_{6}:=  \|\Pi_{\beta}^{B}(d\beta\mid Y,\sigma_{0}^{2})-\Pi_{\beta}(d\beta\mid Y,\sigma_{0}^{2})\|_{\mathrm{TV}}.$
Upper bounds on $A_{4}, A_{5} ,A_{6}$ will be presented in 
(\ref{eq: bound 1 of the second term}), (\ref{eq: bound 2 of the second term}), and (\ref{eq: bound 3 of the second term}),
respectively.

\subsubsection*{Step 1: Upper bound on (\ref{eq: decomposition of marginal posterior})}
We start with bounding $A_{1}$ in (\ref{eq: decomposition of marginal posterior}).
From Lemmas \ref{lemma: Scheffe} and \ref{lemma: posterior tail mass},
taking a sufficiently large $c_{1}$ depending only on $C_{1}$ yields
\begin{align}
	A_{1}=\int_{S}\Pi_{\beta}( \beta \not \in B \mid Y,\sigma^{2}) \Pi_{\sigma^{2}}(d\sigma^{2}\mid Y)
	\leq 4 n^{-\tilde{c}_{1}p}.
\label{eq: bound 1 of the first term}
\end{align}

We next bound $A_{2}$ in (\ref{eq: decomposition of marginal posterior}).
Lemma \ref{lemma: Scheffe} gives
\begin{align*}
	A_{2}&=\int_{S}\int\left(1-\phi_{\Pi_{\beta},2}(\beta,\sigma^{2})\right)_{+}
	\Pi_{\beta}^{B}( d\beta\mid Y,\sigma_{0}^{2})
	\Pi_{\sigma^{2}}( d\sigma^{2}\mid Y),
\end{align*}
where
\begin{align*}
\phi_{\Pi_{\beta},2}(\beta,\sigma^{2}):=
\frac{\pi(\beta)\mathrm{e}^{-\|Y-X\beta\|^{2}/(2\sigma^{2}) } }
	{\int_{B}\mathrm{e}^{-\|Y-X\tilde{\beta}\|^{2}/(2\sigma^{2})} \pi(\tilde{\beta})d\tilde{\beta}}
	\frac{\int_{B}\mathrm{e}^{-\|Y-X\tilde{\beta}\|^{2}/(2\sigma_{0}^{2})}\pi(\tilde{\beta})d\tilde{\beta}}
	{\pi(\beta)\mathrm{e}^{-\|Y-X\beta\|^{2}/(2\sigma_{0}^{2})}}.
\end{align*}
From Cauchy--Schwarz's inequality and from Condition 2.4,
we have
\begin{align*}
	&\exp\{-\langle P(\varepsilon+r), X\beta_{0}-X\beta \rangle /\sigma^{2} - \|X\beta_{0}-X\beta\|^{2}/(2\sigma^{2}) \}
\nonumber\\
&\geq
	\mathrm{e}^{-\langle P(\varepsilon+r), X\beta_{0}-X\beta \rangle /\sigma^{2}_{0} - \|X\beta_{0}-X\beta\|^{2}/(2\sigma^{2}_{0})}
	\mathrm{e}^{- C_{2}c_{1}\delta_{1} p\log n /\{4(1-\delta_{1})\} -c^{2}_{1} \delta_{1} p\log n /(1-\delta_{1}) }.
\end{align*}
Likewise,
we have
\begin{align*}
	&\int_{B} \mathrm{e}^{-\langle P(\varepsilon+r), X\beta_{0}-X\tilde{\beta} \rangle /\sigma^{2} - \|X\beta_{0}-X\tilde{\beta}\|^{2}/(2\sigma^{2}) } \pi(\tilde{\beta})d\tilde{\beta}
\nonumber\\
&\leq
\mathrm{e}^{(C_{2}/4+c_{1}) c_{1}\delta_{1}(p\log n) /(1-\delta_{1})}
\int_{B} \mathrm{e}^{
-\langle P(\varepsilon+r), X\beta_{0}-X\beta \rangle/\sigma^{2}_{0}
-\|X\beta_{0}-X\beta\|^{2}/(2\sigma^{2}_{0})}\pi(\tilde{\beta})d\tilde{\beta}.
\end{align*}
Therefore,
we have
$\phi_{\Pi_{\beta},2}(\beta,\sigma^{2})
\geq
n^{-\tilde{c}_{2} \delta_{1}p}$
for $\beta\in B$, $Y\in H$, and $\sigma^{2}\in S$,
and thus
since $(1-\mathrm{e}^{-x})_{+}\leq x$ for $x>0$,
we obtain
\begin{align}
	A_{2} \leq \tilde{c}_{2}  \delta_{1}p\log n.
\label{eq: bound 2 of the first term}
\end{align}

We bound $A_{3}$ in (\ref{eq: decomposition of marginal posterior}).
From Lemmas \ref{lemma: Scheffe} and \ref{lemma: posterior tail mass},
taking a sufficiently large $c_{1}$ depending only on $C_{1}$,
we have
\begin{align}
	A_{3} &\leq \Pi_{\beta}( \beta \not\in B  \mid Y,\sigma_{0}^{2})
	\leq 4 n^{-\tilde{c}_{3}p} \text{ for }Y\in H.
\label{eq: bound 3 of the first term}
\end{align}
Therefore,
from inequalities (\ref{eq: bound 1 of the first term}), (\ref{eq: bound 2 of the first term}), and (\ref{eq: bound 3 of the first term}),
we obtain
\begin{align}
\int_{S}&\|\Pi_{\beta}(d\beta \mid Y,\sigma^{2})-\Pi_{\beta}(d\beta\mid Y,\sigma^{2}_{0})\|_{\mathrm{TV}}\Pi_{\sigma^{2}}(d\sigma^{2}\mid Y)
\leq
\tilde{c}_{4}\mathrm{e}^{-\tilde{c}_{5} p\log n}+\tilde{c}_{4} \delta_{1}p\log n,
\label{eq: upper bound  of the first term}
\end{align}
which completes Step 1.

\subsubsection*{Step 2: Upper bound on (\ref{eq: three steps})}
We start with bounding $A_{4}$ in (\ref{eq: three steps}).
From Lemmas \ref{lemma: Scheffe} and \ref{lemma: exponential inequality for quadratic forms},
we have
\begin{align}
A_{4} = \tilde{\mathcal{N}}(B^{\mathrm{c}})
\leq \exp\{- (3 c_{1} \sqrt{p\log n} / 4- \sqrt{p})^{2}/2\}.
\label{eq: bound 1 of the second term}
\end{align}

We next bound $A_{5}$ in (\ref{eq: three steps}).
Lemma \ref{lemma: Scheffe} gives
\begin{align*}
	A_{5} =\int (1- d\widetilde{\mathcal{N}}^{B}(\beta) / d\Pi_{\beta}^{B}(\cdot\mid Y,\sigma_{0}^{2}) )_{+}
	\Pi_{\beta}^{B}(d\beta \mid Y,\sigma_{0}^{2}).
\end{align*}
We denote by $\tilde{\phi}$ the density of $\widetilde{\mathcal{N}}$ with respect to the Lebesgue measure.
Observe that
\begin{align*}
	\frac{d\widetilde{\mathcal{N}}^{B}}{d\beta}(\beta)=\frac{\tilde{\phi}(\beta)}{\int_{B}\tilde{\phi}(\tilde{\beta})d\tilde{\beta}}
	\ \text{ and }\
	\frac{d\Pi_{\beta}^{B}(\cdot \mid Y,\sigma_{0}^{2})}{d\beta}(\beta)
	=\frac{\pi(\beta)\tilde{\phi}(\beta)}{\int_{B} \pi(\tilde{\beta})\tilde{\phi}(\tilde{\beta})d\tilde{\beta}}
\end{align*}
for $\beta\in B$.
Together with Jensen's inequality, this gives
\begin{align*}
	\int \Big{(}1-\frac{d\tilde{\mathcal{N}}^{B} }{d\Pi_{\beta}^{B}}(\beta \mid Y,\sigma_{0}^{2})\Big{)}_{+}\Pi_{\beta}^{B}(d\beta \mid Y)
	=& \int \left(1-\int_{B} \frac{\pi(\tilde{\beta})}{\pi(\beta)} \frac{\tilde{\phi}(\tilde{\beta})}{\int_{B} \tilde{\phi}(\beta')d\beta'}
	d\tilde{\beta} \right)_{+}\Pi_{\beta}^{B}(d\beta \mid Y)
	\nonumber\\
	\leq& \int \int_{B}\left(1-\frac{\pi(\tilde{\beta})}{\pi(\beta)}\right)_{+} \frac{\tilde{\phi}(\tilde{\beta})}{\int_{B}\tilde{\phi}(\beta')d\beta'}
	d\tilde{\beta}\Pi_{\beta}^{B}(d\beta \mid Y)
\end{align*}
and thus we obtain
\begin{align}
	A_{5} \leq \phi_{\Pi_{\beta}}(c_{1}\sqrt{p\log n}).
	\label{eq: bound 2 of the second term}
\end{align}

We bound $A_{6}$ in (\ref{eq: three steps}).
From Lemmas \ref{lemma: Scheffe} and \ref{lemma: posterior tail mass},
taking a sufficiently large $c_{1}>0$, we have
\begin{align}
	A_{6} =\Pi_{\beta}( \beta \not\in B  \mid Y,\sigma_{0}^{2})
	\leq 4 n^{ -\tilde{c}_{6} p}.
	\label{eq: bound 3 of the second term}
\end{align}
Therefore, 
from inequalities (\ref{eq: bound 1 of the second term}), (\ref{eq: bound 2 of the second term}), and (\ref{eq: bound 3 of the second term}),
we obtain
\begin{align}
\|\Pi_{\beta}&(d\beta\mid Y,\sigma_{0}^{2})-\widetilde{\mathcal{N}}\|_{\mathrm{TV}}
\leq
\phi_{\Pi_{\beta}}(c_{1}\sqrt{p\log n}) + \tilde{c}_{7} n^{-\tilde{c}_{8}p},
\label{eq: upper bound  of the second term}
\end{align}
which completes Step 2.

Combining (\ref{eq: upper bound of the first term}) and (\ref{eq: upper bound of the second term}) with (\ref{eq: decomposition of marginal posterior}) 
provides the upper bound of the target total variation and thus completes the proof.
\qed

\subsection{Proof of Proposition 2.6}

Let $c$ be any positive number.
Under Condition 2.5 (a),
Lemma \ref{lemma: exponential inequality for quadratic forms} (a) with $R=c\sqrt{p\log n}$
gives
\[\Pr ( Y \not\in H(c) ) \leq \tilde{c}_{1} p^{1-q/2} (\log n)^{-q/2}+\delta_{3}\]
for some $\tilde{c}_{1}>0$ depending only on $c$, $C_{3}$, and $q$.
Under Condition 2.5 (b),
Lemma \ref{lemma: exponential inequality for quadratic forms} (b) with $R=(c^{2}+1)p\log n$
gives
\[\Pr( Y \not\in H(c) ) \leq 2\exp[-\tilde{c}_{2} \min\{ p (\log n)^{2} , p\log n \} ]+\delta_{3}\]
for some $\tilde{c}_{2}>0$ depending only on $c$, $C_{3}$, and $q$.
Thus, we complete the proof.\qed

\section{Proofs of Propositions 2.1--2.4}

\subsection{Proof of Proposition 2.1}
Let $\tilde{B}(R) : = \{\beta:\|\beta-\beta_{0}\|\leq \sigma_{0}\uplambda^{1/2}R \}$ for $R>0$.
Observe that we have
\begin{align}
	\phi_{\Pi_{\beta}}(c\sqrt{p\log n})
	\leq \sup_{\beta,\tilde{\beta} \in \tilde{B}} (1- \exp[-\log\{\pi(\beta)/\pi(\tilde{\beta})\}])
	\leq c \sigma_{0}L \uplambda^{1/2} \sqrt{p\log n}
	\label{eq: flatness evaluation}
\end{align}
for any $c>0$,
where
the first inequality follows because $\|X(\beta-\beta_{0})\|\geq \uplambda^{-1/2}\|\beta-\beta_{0}\|$
and
the second inequality follows because $1-\mathrm{e}^{-x}\leq x$.
Substituting $c=1/(\sqrt{pn\log n})$ into (\ref{eq: flatness evaluation}),
we obtain the desired inequality
$\phi_{\Pi_{\beta}}(1/\sqrt{n})\leq L\uplambda^{1/2}\sigma_{0}/\sqrt{n}$.\qed

\subsection{Proof of Proposition 2.2}

We start with an isotropic prior.
Observe that
\begin{align*}
	\log \pi( \beta_{0} ) & = \log \rho (\|\beta_{0}\|) - \log \int \rho(\|\beta\|) d\beta
	\\
	&= \log \rho (\|\beta_{0}\|) - \log \left[
	\{p\pi^{p/2} / \Gamma(p/2+1)\} \int_{0}^{\infty} x^{p-1} \rho (x) dx \right]
	\\
	&\geq \log \left\{\inf_{x\in [0,B]} \rho(x)\right\} - \tilde{c}_{1} p\log p
        \\
	&\geq \log \left\{\inf_{x\in [0,B]} \rho(x)\right\} - \tilde{c}_{1} p \log p
	- \tilde{c}_{1} p \log n + \log \left\{\sqrt{\mathrm{det}(X^{\top}X)}/\sigma^{p}_{0}\right\}
\end{align*}
for some positive constant $\tilde{c}_{1}$ depending only on 
$m$ and $c$ appealing in the definition of an isotropic prior and Condition 2.6.
This shows that an isotropic prior satisfies Condition 2.1.
Taylor's expansion gives
\begin{align*}
 \left| \log \frac{\pi(\beta_{0}+s_{1})}{\pi(\beta_{0}+s_{2})}\right| 
	\leq 
	\sup_{x: 0 \leq x \leq B+ \sqrt{\sigma_{0}^2 \uplambda p\log n} } | (\log\rho)'(x) | ( \| \beta_{0} + s_{1} \| - \| \beta_{0} + s_{2} \|)
\end{align*}
for $s_{1},s_{2}\in\R^{p}$,
which shows that an isotropic prior satisfies the locally log-Lipschitz continuity,
Thus, we complete the proof for an isotropic prior.

We next prove the case with a product prior $\pi(\beta)=\prod_{i=1}^{p}\pi_{i}(\beta_{i})$.
Observe that
\begin{align*}
	\log \pi(\beta_{0})
	&\geq p \log \left\{\min_{i}\pi_{i}(0)\right\} -\tilde{L}p^{1/2}\|\beta_{0}\|
	\nonumber\\
	&\geq p\log \left\{\min_{i}\pi_{i}(0)\right\}-\tilde{L}Bp\log n
	\nonumber\\
	&\geq -\tilde{L} B p (1+o(1)) \log n - \tilde{c}_{2} p \log n + \log \{\sqrt{\mathrm{det}(X^{\top}X)}/\sigma^{p}_{0}\}
\end{align*}
for some positive constant $\tilde{c}_{2}$ depending only on
$c$ appearing in Condition 2.6.
This shows that a product prior satisfies Condition 2.1.
The Lipschitz continuity of $\log\pi(\beta)$ gives 
\begin{align*}
	|\log\pi(\beta)-\log\pi(\beta_{0})|\leq \sum_{i=1}^{p}|\log\pi_{i}(\beta_{i})-\log\pi_{i}(\beta_{0,i})|
	\leq \tilde{L}p^{1/2}\|\beta-\beta_{0}\|,
\end{align*}
which shows that a product prior satisfies the locally log-Lipschitz continuity and thus completes the proof.\qed

\subsection{Proof of Proposition 2.3}
We only prove the theorem under Condition 2.5 (a).
The proof under Condition 2.5 (b) is done
by replacing Lemma \ref{lemma: exponential inequality for quadratic forms} (a) with Lemma \ref{lemma: exponential inequality for quadratic forms} (b).

Observe that
\begin{align*}
\hat{\sigma}^{2}_{\mathrm{u}} 
&= \|Y-X(X^{\top}X)^{-1}X^{\top}Y\|^{2} \big{/} \big{\{} \sigma_{0}^{2}(n-p)\sigma_{0}^{2} \big{\}} \\
&\leq \big{\{} \|\varepsilon - X(X^{\top}X)^{-1}X^{\top}\varepsilon\|^{2}+ 2 \|r- X(X^{\top}X)^{-1}X^{\top}r\|^{2} + |\varepsilon^{\top}u |^{2} \big{\}}
\big{/} \{ \sigma_{0}^{2} (n-p) \} \\
&= \big{\{} \varepsilon^{\top}\tilde{A}\varepsilon + 2 \|r- X(X^{\top}X)^{-1}X^{\top}r\|^{2} \big{\}} \big{/} \{\sigma_{0}^{2}(n-p)\},
\end{align*}
where 
\begin{align*}
u:= 
\begin{cases} 
\frac{\{I-X(X^{\top}X)^{-1}X^{\top}\} r}{ \| \{I-X(X^{\top}X)^{-1}X^{\top}\}r\|}
& \text{ if $\{I-X(X^{\top}X)^{-1}X^{\top}\}r \neq 0 $}, \\
\text{ arbitrary }
& \text{ if otherwise},
\end{cases}
\end{align*}
and $\tilde{A}:= I-X(X^{\top}X)^{-1}X^{\top} + uu^{\top}$.
Then Lemma \ref{lemma: exponential inequality for quadratic forms} (a) gives
\[\Pr ( \hat{\sigma}^2_{\mathrm{u}}/\sigma_{0}^{2} - 1 \geq \delta_{1} ) \leq \tilde{c}_{1}/(n-p)^{q/2-1}\tilde{\delta}_{1}^{q}\]
for some positive constant $\tilde{c}_{1}$ depending only on $q$.

Next, we will show that
\begin{align}
\Pr\left( \widehat{\sigma}^{2}_{\mathrm{u}}(Y) / \sigma_{0}^{2} - 1 \leq -\delta_{1} \right)
\leq 
\tilde{c}_{2} \frac{\max\{n^{q/4},n\}}{ \delta_{1}^{q/2}(n-p)^{q/2}} + \tilde{c}_{2} \frac{ p^{q/2+1}}{ (n-p)^{q}\delta_{1}^{q}}
\label{eq:variance_bound_lower}
\end{align}
for some positive constant $\tilde{c}_{2}$ depending only on $q$.
Letting $\tilde{P}$ be the projection onto the linear space spanned by columns of $X$ and $(I-X(X^{\top}X)^{-1}X^{\top})r$,
we have
\begin{align}
&\Pr\left( \widehat{\sigma}^{2}_{\mathrm{u}}(Y) / \sigma_{0}^{2} -1 \leq -\delta_{1} \right)
\nonumber\\
&\leq \Pr\left( \{\|\varepsilon\|^{2}-\|\tilde{P}\varepsilon\|^{2}\} / \{\sigma_{0}^{2}(n-p)\} \leq 1-\delta_{1} \right)
\nonumber\\
&\leq
\Pr\left( \|\varepsilon\|^{2} / \sigma_{0}^{2}(n-p) - n / (n-p) \leq - \delta_{1} / 2 \right)
+
\Pr\left( \| \tilde{P} \varepsilon\|^{2} / \sigma_{0}^{2}(n-p) \geq p/(n-p) + \delta_{1} / 2 \right).
\label{eq:decomp_variance_bound_lower}
\end{align}
For bounding the first term on the rightmost side in (\ref{eq:decomp_variance_bound_lower}), 
we use Rosenthal's inequality:
\begin{lemma}[Rosenthal's inequality; see \cite{Petrov(1995)} and \cite{vonBahrandEsseen(1965)}.]
For some positive constant $\tilde{c}_{3}$ depending only on $q$, we have
$\Ep\left|\|\varepsilon/\sigma_{0}\|^{2}-n\right|^{q/2}\leq \tilde{c}_{3} \max\{n^{q/4},n\}.$
\end{lemma}
From Markov's inequality and from Rosenthal's inequality, we have
\begin{align}
\Pr &\left ( \|\varepsilon\|^{2} / \{\sigma_{0}^{2}(n-p)\} - n / (n-p) \leq  - \delta_{1} / 2 \right)
\leq \tilde{c}_{4} \max\{n^{q/4},n\} / \{\delta_{1}^{q/2}(n-p)^{q/2}\}
\label{eq:decomp_variance_bound_lower_1}
\end{align}
for some $\tilde{c}_{4}>0$ depending only on $q$.
For bounding the second term on the rightmost hand side in (\ref{eq:decomp_variance_bound_lower}),
Lemma \ref{lemma: exponential inequality for quadratic forms} (a) with $R = \sqrt{p+(n-p)\delta_{1}/2}$ gives
\begin{align}
\Pr &\left( \| \tilde{P}\varepsilon\|^{2} / \{\sigma_{0}^{2}(n-p)\} \geq p/(n-p)+ \delta_{1}/2 \right) \leq \tilde{c}_{4} n^{1-q/2}/\delta_{1}^{q/2}.
\label{eq:decomp_variance_bound_lower_2}
\end{align}
Combining (\ref{eq:decomp_variance_bound_lower_1}) and (\ref{eq:decomp_variance_bound_lower_2}) with (\ref{eq:decomp_variance_bound_lower}),
we have (\ref{eq:variance_bound_lower}),
which completes the proof under Condition 2.5 (a).\qed

\subsection{Proof of Proposition 2.4}
The marginal posterior distribution of $\sigma^{2}$ is given by
the inverse Gamma distribution $\mathrm{IG}(a^{*},b^{*})$,
where $a^{*} = \mu_{1} + n/2 - p/2$ and $b^{*} = \mu_{2} +\|Y-PY\|^{2}/2$.
The mean of this marginal posterior is 
$\{ 2\mu_{2} + \| (I-X (X^{\top}X)^{-1} X^{\top} ) Y \|^{2} \} / \{ 2\mu_{1} + n-p -2 \};$
while the variance is 
$2 \{ 2\mu_{2} + \| (I-X (X^{\top}X)^{-1} X^{\top} ) Y  \|^{2} \}^{2}  /  \{ (2\mu_{1} + n-p - 2 )^{2} ( 2\mu_{1} + n - p - 4) \} .$
From Chebyshev's inequality, we have
\[
\Pi_{\sigma^{2}} ( \sigma^{2} : |\sigma^{2} / \sigma^{2}_{0} - 1 | \geq \delta_{1} \mid Y )
\leq \tilde{c}_{1} \frac{\| (I-X (X^{\top}X)^{-1} X^{\top} ) Y \|^{2}}{n^{2} (\delta_{1} - |\Ep[\sigma^{2}/\sigma^{2}_{0} \mid Y] - 1| )^2}
\]
for some positive constant $\tilde{c}_{1}$ depending only on $\mu_{1}$ and $\mu_{2}$.
From the proof of Proposition 2.3,
we obtain the desired upper bound of $\Pr( \| (I-X(X^{\top}X)^{-1}X^{\top } ) Y  \|^{2} / (n-p) - 1 \geq \delta_{1}/2 )$
and thus complete the proof.
\qed

\section{Proofs for Section 3}

In this section, we provide proofs for Section 3.

\subsection{Proof of Proposition 3.1}

We use the same notation as in the proof sketch.
In addition, let $\{N_{l,k}:(l,k)\in\mathcal{I}(J)\}\sim \mathcal{N}(0,1)$ i.i.d.

\subsubsection*{Step 1: Upper bounds on $\Pr(\max_{(l,k)\in\mathcal{I}(J)}|\varepsilon_{l,k}/w_{l}|\leq \hat{R}_{\alpha})$ and $\hat{R}_{\alpha}$}

We start with bounding $\Pr(\max_{(l,k)\in\mathcal{I}(J)}|\varepsilon_{l,k}/w_{l}|\leq \hat{R}_{\alpha})$
and $\hat{R}_{\alpha}$.
From Theorem 2.1,
there exist $\tilde{c}_{1},\tilde{c}_{2}$ depending only on $C_{1}$ in Condition 2.1 for which we have
\begin{align}
\Big{|}\Pr\Big{(}\max_{(l,k)\in\mathcal{I}(J)}|\varepsilon_{l,k}/w_{l}|\leq \hat{R}_{\alpha}\Big{)} - (1-\alpha)\Big{|}
\leq 
\phi_{\Pi_{\beta}}\Big{(}\tilde{c}_{1}\sqrt{2^{J}\log n}\Big{)} + \tilde{c}_{1} n^{-\tilde{c}_{2}2^{J}}.
\label{eq: bound on coverage errors in low frequencies}
\end{align}
From the assumption that $w_{l}\geq \sqrt{l}$ and since
$\Ep [ \max_{(l,k)\in\mathcal{I}(J)}  |  N_{l,k} / \sqrt{l} |  ] < K$
with some universal constant $K$ (cf.~the proof of Proposition 2 in \cite{CastilloandNickl(2014)}),
we have
\begin{align}
\Ep\left[\max_{(l,k)\in\mathcal{I}(J)}\left|\frac{N_{l,k}}{w_{l}}\right|\right] 
\leq 
\Ep\left[\max_{(l,k)\in\mathcal{I}(J)}\left|\frac{N_{l,k}}{\sqrt{l}}\right|\right]
\leq K.
\label{eq: bound on E max}
\end{align}
Assume that the right hand side in (\ref{eq: bound on coverage errors in low frequencies}) is smaller than $\alpha/2$.
Then, 
from Theorem 2.1 and from (\ref{eq: bound on E max}),
there exists $\tilde{c}_{3}>0$ depending only on $\alpha$ for which we have
\begin{align}
\hat{R}_{\alpha} \leq \frac{\tilde{c}_{3}}{\sqrt{n}} \Ep\left[\max_{(l,k)\in\mathcal{I}(J)}\left|\frac{N_{l,k}}{w_{l}}\right|\right] 
\leq \frac{\tilde{c}_{3}K}{\sqrt{n}}
\label{eq: bound on hat R in low frequencies}
\end{align}
with probability at least $1-\tilde{c}_{1}n^{-\tilde{c}_{2}2^{J}}$.
This completes Step 1.
Note that it is unnecessary for upper bounding $\hat{R}_{\alpha}$ to assume that the right hand side is also smaller than $(1-\alpha)/2$;
see the remark below Theorem 2.1.

\subsubsection*{Step 2: Upper bound on $\Pr(Y_{\infty}\not\in\tilde{H}')$}
Next we bound $\Pr(Y_{\infty}\not\in\tilde{H}')$.
From Theorem 2.1,
we have the set $H$ satisfying the following:
\begin{enumerate}
\item[P1] 
Assume that the right hand side in (\ref{eq: bound on coverage errors in low frequencies}) is smaller than $(1-\alpha)/2$.
Then,
there exists $\tilde{c}_{4}>0$ depending only on $\alpha$ such that
we have
\begin{align*}
\tilde{c}_{4} \frac{1}{\sqrt{n}} \frac{J^{1/2}}{\overline{w}_{J}} \leq \hat{R}_{\alpha}
\end{align*}
for $Y_{\infty}\in H$ and for $J\ge J_{\alpha}$ with $J_{\alpha}$ depending only on $\alpha$;
\item[P2] We have $\Pr(Y_{\infty}\not\in H)\leq \tilde{c}_{1}n^{-\tilde{c}_{2}2^{J}}$.
\end{enumerate}
From the first property P1 of $H$, we have
\begin{align*}
\Pr(Y_{\infty}\not\in \tilde{H}') 
&= \Pr(Y_{\infty}\not\in\tilde{H}',Y_{\infty}\in H) + \Pr(Y_{\infty}\not\in\tilde{H}',Y_{\infty}\not\in H)\\
&\leq 
\Pr\left(\sup_{ J \leq l <\infty, 0\leq k \leq 2^{l}-1} \left|\frac{N_{l,k}}{w_{l}}\right| \geq \tilde{c}_{4}\frac{J^{1/2}}{\overline{w}_{J}} \right)
+ \Pr(Y_{\infty}\not\in H).
\end{align*}
The second property P2 of $H$ gives an upper bound on $\Pr(Y_{\infty}\not \in H)$.
From Lemma 4.4,
we have, for $J\leq l<\infty$,
\begin{align*}
&\Pr\left(\max_{0\leq k \leq 2^{l}-1} \left|\frac{N_{l,k}}{w_{l}}\right| \geq \tilde{c}_{4} \frac{J^{1/2}}{\overline{w}_{J}} \right)\\
&\leq \Pr\left(\max_{0\leq k \leq 2^{l}-1} |N_{l,k}| \geq \tilde{c}_{4} \frac{J^{1/2}}{\overline{w}_{J}} u_{J} \sqrt{l} \right) \\
&\leq \Pr\left(\max_{0\leq k \leq 2^{l}-1} |N_{l,k}| - \Ep\left[\max_{0 \leq k \leq 2^{l}-1} |N_{l,k}|\right] \geq \tilde{c}_{4} \frac{J^{1/2}}{\overline{w}_{J}} u_{J} \sqrt{l} - \Ep[\max_{0 \leq k \leq 2^{l}-1} |N_{l,k}|] \right)\\
&\leq \Pr\left(\max_{0\leq k \leq 2^{l}-1} |N_{l,k}| - \Ep\left[\max_{0 \leq k \leq 2^{l}-1} |N_{l,k}|\right] \geq \tilde{c}_{4} \frac{J^{1/2}}{\overline{w}_{J}} u_{J} \sqrt{l} - \sqrt{2} \sqrt{l} \right)\\
&\leq 2\exp\left\{ - \tilde{c}_{5} \left( (J^{1/2}/\overline{w}_{J})u_{J}-\sqrt{2}/\tilde{c}_{4}\right)^{2}l \right\}
\end{align*}
with a positive constant $\tilde{c}_{5}$ depending only on $\tilde{c}_{4}$.
Together with the assumption that $1\leq (J/\overline{w}_{J}^{2})u_{J}^{2}$,
this implies
that there exist $\tilde{c}_{6},\tilde{c}_{7}>0$ depending only on $\tilde{c}_{4}$ such that we have
\begin{align*}
\Pr\left(\sup_{ J \leq l <\infty, 0\leq k \leq 2^{l}-1} \left|\frac{N_{l,k}}{w_{l}}\right| 
\geq \tilde{c}_{4}\frac{J^{1/2}}{\overline{w}_{J}} \right)
&\leq \sum_{J \leq l < \infty} \Pr\Bigg{(}\max_{0\leq k \leq 2^{l}-1} \left|\frac{N_{l,k}}{w_{l}}\right| 
\geq \tilde{c}_{4} \frac{J^{1/2}}{\overline{w}_{J}} \Bigg{)} \\
&\leq \sum_{J \leq l < \infty} 2\exp \{ - \tilde{c}_{6} l (J/\overline{w}_{J}^{2})u_{J}^{2} \} \\
&\leq \tilde{c}_{7}\exp\{-\tilde{c}_{6}J(J/\overline{w}_{J}^{2})u_{J}^{2}\}
\end{align*}
for $J$ satisfying
$\{ (J^{1/2}/\overline{w}_{J}) u_{J} - \sqrt{2} / \tilde{c}_{4}\}^{2} \geq (1/2)(J^{1/2}/\overline{w}_{J})^{2}u_{J}^{2}$
(such $J$ exists since $(J/\overline{w}_{J}^{2})u_{J}^{2}\uparrow \infty$ as $J\to\infty$).
Thus we complete Step 2.

\subsubsection*{Step 3: Upper bound on the $L^{\infty}$-diameter}
Finally we provide a high-probability upper bound on the $L^{\infty}$-diameter.
Fix $f,g\in \mathcal{C}_{w}^{B}(\hat{f}_{\infty},\hat{R}_{\alpha})$ and let $h:=f-g$.
From the property of a wavelet basis (cf.~p.~325 of \cite{GineandNickl_book}),
there exists $\tilde{c}_{8}>0$ depending only on $\{\psi_{l,k}:(l,k)\in\mathcal{I}_{\infty}\}$
for which we have
\[
\|h\|_{\infty} \leq 
\tilde{c}_{8}\sum_{J_{0}-1\leq l<\infty}2^{l/2}\max_{0\leq k\leq 2^{l}-1}|\langle h ,\psi_{l,k} \rangle|
= \tilde{c}_{8} ( A_{1} + A_{2} ),
\]
where 
\[
A_{1} := \sum_{J_{0}-1 \leq l \leq J-1} 2^{l/2}\max_{0\leq k \leq 2^{l}-1} |\langle h, \psi_{l,k} \rangle|
\text{ and }
A_{2} := \sum_{J \leq l < \infty } 2^{l/2} \max_{ 0 \leq k \leq 2^{l}-1} |\langle h, \psi_{l,k} \rangle|.
\]
Inequality (\ref{eq: bound on hat R in low frequencies}) gives
\begin{align*}
A_{1} \leq \max_{J_{0}-1 \leq l \leq J-1}\left\{\frac{w_{l}}{\sqrt{l}}\right\} 
\sum_{J_{0}-1 \leq l \leq J-1}2^{l/2}\sqrt{l} 2\hat{R}_{\alpha}
      \leq \tilde{c}_{9} v_{J} \sqrt{\frac{2^{J}J}{n}}
\end{align*}
with some $\tilde{c}_{9}>0$ depending on $\tilde{c}_{3}$ appearing in (\ref{eq: bound on hat R in low frequencies}).
Since $\max\{\|f\|_{B^{s}_{\infty,\infty}}, \|g\|_{B^{s}_{\infty,\infty}}\} \leq B$,
we have
\begin{align*}
A_{2} \leq \sum_{J\leq l < \infty} 2^{-ls}\max_{0\leq k \leq 2^{l}-1}2^{l(s+1/2)}|\langle h, \psi_{l,k}\rangle|
      \leq 2^{-Js+2}B,
\end{align*}
which completes Step 3 and thus completes the proof.

\begin{remark}[The choice of $J$ in the second part of Proposition 3.1]
For ``sufficiently large $J$" appearing in the second part of Proposition 3.1,
we can take $J$ satisfying $J\ge J_{\alpha}$ and
\[
\{ (J^{1/2}/\overline{w}_{J})u_{J} - \sqrt{2}/\tilde{c}_{4} \}^{2} \ge (1/2) (J^{1/2} / \overline{w}_{J})^{2} u_{J}^{2},
\]
where $J_{\alpha}$ and $\tilde{c}_{4}=\tilde{c}_{4}(\alpha)$ are the constants in the property P1.
Thus even in the case that $\{w_{l}\}$ depends on $n$ as in Remark 3.1,
we can apply Proposition 3.1 to deduce the coverage error.
\end{remark}

\subsection{Proof of Proposition 3.2}

The proof follows essentially the same line and the same notation as those of Proposition 3.1.
The only difference is the way of bounding $\Pr(Y\not\in\tilde{H}_{2})$.
From the lower estimate of $\hat{R}_{\alpha}$ in Theorem 2.1,
we have
\[
\Pr(Y\not\in\tilde{H}_{2},Y\in H) \leq \Pr\left(\sup_{J\leq l < \infty, 0\leq k \leq 2^{l}-1} \frac{\langle f_{0}, \psi_{l,k} \rangle}{w_{l}}
\geq \tilde{c}_{1}\frac{(J^{1/2}/\overline{w}_{J})}{\sqrt{n}} \right)
\]
for sufficiently large $J$ depending only on $\alpha$.
From the assumption that $\|f_{0}\|_{B^{s}_{\infty,\infty}} \leq B$,
we have
\[
\Pr(Y\not\in\tilde{H}_{2},Y\in H) \leq \Pr\left( \frac{ \sqrt{n} \overline{w}_{J} B}{u_{J} J 2^{J(s+1/2)}} 
\geq \tilde{c}_{1} \right).
\]
This completes the proof.
\qed

\subsection{Proof of Proposition 3.3}

We use the same notation as in the proof sketch.
In addition, let $\{N_{l,k}:(l,k)\in\mathcal{I}(J)\}\sim \mathcal{N}(0,1)$ i.i.d..
From the near-orthogonality of $\{v^{(1)}_{l,k}\}$,
there exist positive constants $\underline{b}$ and $\overline{b}$
depending only on $K$ and $\{\psi_{l,k}:(l,k)\in\mathcal{I}_{\infty}\}$
such that
\begin{align}
\underline{b} I_{2^{J} } / n \preceq \Sigma \preceq \overline{b} I_{2^{J} } / n.
\label{eq: near-orthogonality}
\end{align}

\subsubsection*{Step 0: the well-definedness of $\hat{f}_{\infty}$}

Before proving the proposition, we will show that $\hat{f}_{\infty}$ converges almost surely in $\mathcal{M}_{0}(w)$ for any sequence $w$ such that $\min_{0\le k \le 2^{l}-1} \kappa_{l,k}w_{l} / \sqrt{l} \uparrow \infty$.
Here 
for a positive sequence $w=(w_{1},w_{2},\ldots)$,
$\mathcal{M}(w):=\{f: \|f\|_{\mathcal{M}(w)}  :=  \sum_{(l,k)\in\mathcal{I}_{\infty}} | \langle f , \psi_{l,k} \rangle  |  /  w_{l}  < \infty  \}$
and $\mathcal{M}_{0} (w)   :=   \{   f  \in  \mathcal{M}(w) : \lim_{l\to\infty}\max_{k=0,\ldots, 2^{l}-1} | \langle  f, \psi_{l,k} \rangle | / w_{l} = 0 \}$.

We begin with showing that $\|\hat{f}_{\infty}\|_{\mathcal{M}(w)}$ has a finite expectation, which implies it exists almost surely in $\mathcal{M}(w)$. Observe that for $M>0$,
\begin{align*}
\Pr\left( \|\hat{f}_{\infty} - f_{0} \|_{\mathcal{M}(w)} > \frac{M}{\sqrt{n}} \right)
&= \Pr  \left( \sup_{(l,k) \in \mathcal{I}_{\infty}} \frac{| \tilde{\varepsilon}_{l,k}  | }{\kappa_{l,k}w_{l}} > \frac{M}{\sqrt{n}}  \right)  \\
&\le \sum_{J_{0}-1 \le l < \infty} \Pr \left( \max_{0\le k \le 2^{l}-1} \frac{|\tilde{\varepsilon}_{l,k} |}{ \sqrt{l} }  > \frac{M}{\sqrt{n}} \min_{0 \le k \le 2^{l}-1} \frac{\kappa_{l,k}w_{l}}{\sqrt{l}} \right)\\
&\le \sum_{J_{0}-1 \le l < \infty} \Pr \left( \max_{0\le k \le 2^{l}-1} |N_{l,k}| > \frac{M}{\overline{b}} \min_{0\le k \le 2^{l}-1} \frac{\kappa_{l,k}w_{l}}{\sqrt{l}} \sqrt{l} \right),
\end{align*}
where the last inequality follows from Anderson's lemma (Lemma 4.3).
Together with the concentration inequality (Lemma 4.4) and the maximal inequality,
this implies that for sufficiently large $M>0$,
\begin{align*}
\Pr( \| \hat{f}_{\infty} - f_{0} \|_{\mathcal{M}(w)} > M/\sqrt{n} )
&\le 2\sum_{J_{0} -1 \le l < \infty} \exp\left[ - \left\{ \frac{M}{\overline{b}} \min_{0 \le k \le 2^{l}-1}\frac{\kappa_{l,k}w_{l}}{\sqrt{l}}  - \sqrt{2} \right\}^{2} l/2  \right]\\
&\le \tilde{c}_{0} \exp\{ - \tilde{c}M^{2} \}
\end{align*}
with some $\tilde{c}_{0},\tilde{c}>0$,
where, for example, take $M$ such that
\begin{align*}
\frac{M}{\overline{b}}\min_{(l,k)\in\mathcal{I}_{\infty}}\frac{\kappa_{l,k}w_{l}}{\sqrt{l}} - \sqrt{2} 
> 
\frac{1}{2}
\frac{M}{\overline{b}}\min_{(l,k)\in\mathcal{I}_{\infty}}\frac{\kappa_{l,k}w_{l}}{\sqrt{l}}.
\end{align*}
Note that 
\[\inf_{(l,k)\in\mathcal{I}_{\infty}} \frac{\kappa_{l,k} w_{l}}{\sqrt{l}}
= \min_{(l,k)\in\mathcal{I}_{\infty}} \frac{\kappa_{l,k} w_{l}}{\sqrt{l}} > 0
\]
by the assumption that $\min_{0\le k \le 2^{l}-1} \kappa_{l,k}w_{l} / \sqrt{l}\uparrow \infty$.
Using $\Ep[X] \le K + \int_{K}^{\infty} \Pr(X \ge t) dt$ for any real valued random variable $X$ and any $K\ge 0$
and
observing that $\|f_{0}\|_{\mathcal{M}(w)}<\infty$ for $f_{0}\in B^{s}_{\infty,\infty}$,
we obtain that $\|\hat{f}_{\infty}\|_{\mathcal{M}(w)}$ has a finite expectation.

Next, the assumption that $\min_{0\le k \le 2^{l}-1} \kappa_{l,k}w_{l} / \sqrt{l}\uparrow \infty$
gives
\begin{align*}
\Pr&\left(\lim_{l\to\infty}\max_{0\le k \le 2^{k}-1}\frac{ | \langle (\hat{f}_{\infty} - f_{0}) , \psi_{l,k} \rangle | }{w_{l}} 
\neq 0\right)
=
\Pr\left(\lim_{l\to\infty}\max_{0\le k \le 2^{k}-1}\frac{ | \tilde{\varepsilon}_{l,k} | }{\kappa_{l,k} w_{l}}
\neq 0\right) \\
&\le 
\sum_{M \in \mathbb{Q} , M>0} \lim_{L\to\infty}
\sum_{l \ge L} \Pr \left( \max_{k=0,\ldots,2^{l}-1} \frac{|\tilde{\varepsilon}_{l,k}|}{\sqrt{l}} 
\ge \min_{0\le k \le 2^{l}-1} \frac{ \kappa_{l,k}w_{l}}{\sqrt{l}} M\right) \to 0,
\end{align*}
where the last convergence follows from Lemmas 4.3-4.4.
This shows that $\hat{f}_{\infty}$ converges almost surely in $\mathcal{M}_{0}(w)$.

\subsubsection*{Step 1: Upper bounds on $\Pr(\max_{(l,k)\in\mathcal{I}(J)}|\kappa_{l,k}^{-1}\tilde{Y}_{l,k}-\beta_{0,l,k}|/w_{l}\leq \hat{R}_{\alpha})$
and $\hat{R}_{\alpha}$}

We start with bounding $\Pr(\max_{(l,k)\in\mathcal{I}(J)}|\kappa_{l,k}^{-1}\tilde{Y}_{l,k}-\beta_{0,l,k}|/w_{l}\leq \hat{R}_{\alpha})$
and $\hat{R}_{\alpha}$.
From Theorem 2.1,
and by the same way as in the previous subsection,
there exist $\tilde{c}_{1},\tilde{c}_{2}>0$ depending only on 
$C_{1}$ in Condition 2.1 for which we have
\begin{align}
\left|\Pr\left(\max_{(l,k)\in\mathcal{I}(J)}\left|\frac{\kappa_{l,k}^{-1}\tilde{Y}_{l,k}-\beta_{0,l,k}}{w_{l}}\right|
\leq \hat{R}_{\alpha}\right) - (1-\alpha)\right|
\leq 
\phi_{\Pi_{\beta}}\Big{(}\tilde{c}_{1}\sqrt{2^{J}\log n}\Big{)} + \tilde{c}_{1}\mathrm{e}^{-\tilde{c}_{2}2^{J}\log n}.
\label{eq: bound on coverage errors in low frequencies in inverse problems}
\end{align}
Assume that the right hand side above is smaller than $\alpha/2$.
Then, from Theorem 2.1 and from (\ref{eq: near-orthogonality}),
there exist $\tilde{c}_{3}$ depending only on $\alpha$ and $\overline{b}$ in (\ref{eq: near-orthogonality}) for which we have
\begin{align}
\hat{R}_{\alpha} \leq \frac{\tilde{c}_{3}}{\underline{\kappa}_{J}\sqrt{n}}
\label{eq: bound on hat R in low frequencies in inverse problems}
\end{align}
with probability at least $1-\tilde{c}_{1}n^{-\tilde{c}_{2}2^{J}}$.

\subsubsection*{Step 2: Upper bound on $\Pr(\tilde{Y}_{\infty}\not\in\tilde{H}'_{3})$}
Next we bound $\Pr(\tilde{Y}_{\infty}\not\in\tilde{H}'_{3})$.
Theorem 2.1 gives the set $H$ satisfying the following:
\begin{enumerate}
\item[P'1] Assume that the right hand side in (\ref{eq: bound on coverage errors in low frequencies in inverse problems}) is smaller than $(1-\alpha)/2$.
Then, there exists $\tilde{c}_{4}>0$ depending only on $\alpha$ and $\underline{b}$ in (\ref{eq: near-orthogonality}) 
such that we have
\begin{align}
\tilde{c}_{4} \frac{J^{1/2}}{\overline{w}_{J}\overline{\kappa}_{J}n^{1/2}} \leq \hat{R}_{\alpha}
\label{eq: lower bound on hat R in low frequencies in inverse problems}
\end{align}
for $\tilde{Y}_{\infty} \in H$ and for $J \ge J_{\alpha}$ with $J_{\alpha}$ depending only on $\alpha$;
\item[P'2] We have $\Pr(\tilde{Y}_{\infty}\not\in H) \leq \tilde{c}_{1} n^{-\tilde{c}_{2}2^{J}}$.
\end{enumerate}
From the first property P'1, we have
\begin{align*}
\Pr(\tilde{Y}_{\infty}\not\in \tilde{H}'_{3}) 
&= \Pr(\tilde{Y}_{\infty}\not\in\tilde{H}'_{3},\tilde{Y}_{\infty}\in H) + \Pr(\tilde{Y}_{\infty}\not\in\tilde{H}'_{3},\tilde{Y}_{\infty}\not\in H)\\
&\leq 
\Pr\left(\sup_{ J \leq l <\infty, 0\leq k \leq 2^{l}-1} \left|\frac{\tilde{\varepsilon}_{l,k}}{\kappa_{l,k}w_{l}}\right| \geq 
\tilde{c}_{4}\frac{J^{1/2}}{n^{1/2}\overline{\kappa}_{J}\overline{w}_{J}} \right)
+ \Pr(\tilde{Y}_{\infty}\not\in H).
\end{align*}
The second property P'2 bounds $\Pr(\tilde{Y}_{\infty}\not\in H)$.
From Lemmas 4.3-4.4 together with the assumption that
$1\leq \{J/\overline{\kappa}_{J}^{2}\overline{w}_{J}^{2}\}u_{J}^{2}$,
there exist positive constants $\tilde{c}_{5},\tilde{c}_{6},\tilde{c}_{7}$ depending only on $\tilde{c}_{4}$ and $\overline{b}$
such that we have
\begin{align*}
\Pr\left(\sup_{ J \leq l <\infty, 0\leq k \leq 2^{l}-1} \left|\frac{\tilde{\varepsilon}_{l,k}}{\kappa_{l,k}w_{l}}\right| 
\geq \tilde{c}_{4}\frac{J^{1/2}}{n^{1/2}\overline{\kappa}_{J}\overline{w}_{J} }\right)
&\leq \sum_{J \leq l < \infty} \Pr\left(\max_{0\leq k \leq 2^{l}-1} \left|\frac{\tilde{\varepsilon}_{l,k}}{\kappa_{l,k}w_{l}}\right| 
\geq \tilde{c}_{4}\frac{J^{1/2}}{n^{1/2}\overline{\kappa}_{J}\overline{w}_{J} } \right) \\
&\leq 
\sum_{J \leq l < \infty} \Pr\left( \underline{b} \max_{0\leq k \leq 2^{l}-1} \left|\frac{N_{l,k}}{\kappa_{l,k}w_{l}}\right| 
\geq \tilde{c}_{4}\frac{J^{1/2}}{\overline{\kappa}_{J}\overline{w}_{J} } \right) \\
&\leq \sum_{J \leq l < \infty} \Pr\left(\max_{0\leq k \leq 2^{l}-1} |N_{l,k}| \geq 
\tilde{c}_{5}\frac{J^{1/2}u_{J}}{\overline{\kappa}_{J}\overline{w}_{J}} \sqrt{l} \right) \\
&\leq \tilde{c}_{6}\exp[-\tilde{c}_{7}J\{J/(\overline{\kappa}_{J}^{2}\overline{w}_{J}^{2})\}u_{J}^{2}]
\end{align*}
for sufficiently large $J$ satisfying 
$\{J^{1/2}u_{J}/(\overline{\kappa}_{J}\overline{w}_{J}) - \sqrt{2} (\overline{b}/\tilde{c}_{4}) \}^{2} 
\geq (1/2) J u_{J}^{2}/(\overline{\kappa}_{J}\overline{w}_{J})^{2}$,
which completes Step 2.

\subsubsection*{Step 3: Upper bound on the $L^{\infty}$-diameter}
We finally provide a high-probability upper bound on the $L^{\infty}$-diameter.
Fix $f,g\in \mathcal{C}_{w}^{B}(\hat{f}_{\infty},\hat{R}_{\alpha})$ and let $h:=f-g$.
From the property of a wavelet basis,
there exists $\tilde{c}_{8}>0$ depending only on $\{\psi_{l,k}:(l,k)\in\mathcal{I}_{\infty}\}$
for which we have
\[
\|h\|_{\infty} \leq 
\tilde{c}_{8}\sum_{J_{0}-1\leq l<\infty}2^{l/2}\max_{0\leq k\leq 2^{l}-1}|\langle h ,\psi_{l,k} \rangle|
= \tilde{c}_{8} ( A_{1} + A_{2} ),
\]
where 
\[
A_{1} := \sum_{J_{0}-1 \leq l \leq J-1} 2^{l/2}\max_{0\leq k \leq 2^{l}-1} |\langle h, \psi_{l,k} \rangle|
\text{ and }
A_{2} := \sum_{J \leq l < \infty } 2^{l/2} \max_{ 0 \leq k \leq 2^{l}-1} |\langle h, \psi_{l,k} \rangle|.
\]
Inequality (\ref{eq: bound on hat R in low frequencies in inverse problems}) gives
\begin{align*}
A_{1} \leq \max_{J_{0}-1 \leq l \leq J-1}\left\{\frac{w_{l}}{\sqrt{l}}\right\} \sum_{J_{0}-1 \leq l \leq J-1}2^{l/2}\sqrt{l} 2\hat{R}_{\alpha}
\leq \tilde{c}_{9} v_{J} \sqrt{\frac{2^{J}J}{\underline{\kappa}_{J}^{2}n}}
\end{align*}
with some $\tilde{c}_{9}>0$ depending only on $\tilde{c}_{3}$ appearing in (\ref{eq: bound on hat R in low frequencies in inverse problems}).
Since $\max\{\|f\|_{B^{s}_{\infty,\infty}}, \|g\|_{B^{s}_{\infty,\infty}}\} \leq B$,
we have
\begin{align*}
A_{2} \leq \sum_{J\leq l < \infty} 2^{-ls}\max_{0\leq k \leq 2^{l}-1}2^{l(s+1/2)}|\langle h, \psi_{l,k}\rangle|
      \leq 2^{-Js+2}B,
\end{align*}
which completes Step 3 and thus completes the proof.
\qed

\begin{remark}[The choice of $J$ in the second part of Proposition 3.3]
For ``sufficiently large $J$" appearing in the second part of Proposition 3.3,
we can take $J$ satisfying $J\ge J_{\alpha}$ and
\[
\{J^{1/2}u_{J}/(\overline{\kappa}_{J}\overline{w}_{J}) - \sqrt{2} (\overline{b}/\tilde{c}_{4}) \}^{2} 
\geq (1/2) J u_{J}^{2}/(\overline{\kappa}_{J}\overline{w}_{J})^{2},
\]
where $J_{\alpha}$ and $\tilde{c}_{4}=\tilde{c}_{4}(\alpha)$ are the constants in the property P'1.
\end{remark}

\subsection{Proof for Section 3.3}

We first transform the nonparametric regression model
into the following approximately regression model via $p$ basis functions $\{ \psi^{p}_{j} : 1 \leq j \leq p \}$:
\[
Y = X \beta_{0} + r + \varepsilon,
\]
where 
$Y = (Y_{1} , \ldots , Y_{n})^{\top}$,
$X = (X_{1} , \ldots , X_{n})^{\top}$ 
with  $X_{i}$ whose $j (\in \{1, \ldots, p\})$-th component is $\psi^{p}_{j}(T_{i})$,
and
$r = (r_{1} , \ldots , r_{n})^{\top}$ 
with $r_{i} = f_{0}(T_{i}) - \psi^{p}(T_{i})^{\top}\beta_{0}$.
Recall that $\beta_{0} \in \mathrm{argmin}~\Ep[(f_{0}(T_{1}) - \sum_{j=1}^{p}\psi^{p}_{j}(T_{1})\beta_{j})^{2}]$.

\subsubsection{Supporting lemmas}

We begin with stating five supporting lemmas used in the proof.
Let 
$N_{(n)}$ be a random $n$-vector from $\mathcal{N}( 0 , \sigma_{0}^{2} I_{n} )$,
and
$N_{(p)}$ be a random $p$-vector from $\mathcal{N}( 0, \sigma_{0}^{2} I_{p} )$.
Let $B=(B_{ij}):=(\Ep \psi^{p}_{i}(T_{1})\psi^{p}_{j}(T_{1}))$
and
recall $\tilde{\psi}^{p}(\cdot) := \psi^{p}(\cdot) / \|\psi^{p}(\cdot)\|$ and $\xi_{p} := \|\| \psi^{p}(\cdot) \|\|_{\infty}$.

\begin{lemma}[Matrix Chernoff inequality; \cite{Tropp(2012)}]
\label{lemma: Matrix Chernoff}
Let $\{A_{i} : i=1,\ldots, n \}$ be an i.i.d.~sequence of positive semi-definite and self-adjoint $p\times p$ matrices of which the maximum eigenvalues are almost surely bounded by $R$.
Then, we have
\begin{align*}
	\Pr &\Big{\{} \lambda_{\mathrm{min}} \Big{(} \sum A_{i} / n \Big{)}  \leq (1-\delta) \lambda_{\mathrm{min}}(\Ep[A_{1}])  \Big{\}} 
	\leq  
	p \{\mathrm{e}^{-\delta} / (1-\delta)^{1-\delta} \}^{ n \lambda_{\mathrm{min}}(\Ep [A_{1}]) / R }
	\nonumber\\
	\Pr &\Big{\{} \lambda_{\mathrm{max}} \Big{(} \sum A_{i} / n \Big{)}  \leq (1-\delta) \lambda_{\mathrm{max}}(\Ep[A_{1}]) \Big{\}} 
	\leq  
	p \{\mathrm{e}^{\delta} / (1+\delta)^{1-\delta} \}^{ n \lambda_{\mathrm{min}}(\Ep [A_{1}]) / R }
\end{align*}
for any $\delta\in (0,1]$, where $\lambda_{\mathrm{{min}}}(\cdot)$ and $\lambda_{\mathrm{max}}(\cdot)$ are the maximum and the minimum eigenvalues.
\end{lemma}

\begin{lemma}[Lemma 4.2 in \cite{BelloniChernozhukovChetverikovKato(2015)}]
\label{lemma: Uniform linearization}
Under Conditions 3.3-3.4 and 2.5, we have
\begin{align*}
	\left\| \tilde{\psi}^{p}(\cdot)^{\top} \sqrt{n}(\hat{\beta} -\beta_{0})
	- \tilde{\psi}^{p}(\cdot)^{\top} B^{-1} X^{\top}\varepsilon / \sqrt{n} \right\|_{\infty} 
	\leq R_{1} +R_{2},
\end{align*}
where $R_{1}$ and $R_{2}$ are random variables such that
there exist positive constants $\tilde{c}_{1}$ and $\tilde{c}_{2}$
depending only on $q$ appearing in Condition 2.5 (a) for which
we have
\begin{align*}
	R_{1} &\leq
	\begin{cases} 
		\tilde{c}_{1}\eta^{2} \sqrt{\{\xi^{2}_{p} \log p \}/ n} (n^{1/q}\sqrt{\log p}+ \sqrt{p} \tau_{\infty})
		& \text{under Condition 2.5 (a)}, \\
		\tilde{c}_{1}\eta^{2} \sqrt{\{\xi^{2}_{p} \log p \}/ n} (\sqrt{\log n} \sqrt{\log p}+ \sqrt{p} \tau_{\infty})
		& \text{under Condition 2.5 (b)},
	\end{cases} \\
	R_{2} &\leq \tilde{c_{2}} \eta \sqrt{\log p} \tau_{\infty}
\end{align*}
with probability at least $1-\tilde{c}_{2}/\eta$ with any $\eta>1$.
\end{lemma}
\begin{remark}
Belloni et al.~\cite{BelloniChernozhukovChetverikovKato(2015)} provides the proof under Condition 2.5 (a).
Observing $\Ep [\max_{i=1,\ldots, n} |\varepsilon_{i}| ]\leq Kn^{1/q}$ with some universal constant $K$,
we can prove the case under Condition 2.5 (b).
\end{remark}

\begin{lemma}[Corollary 2.2 and Proposition 3.3 in \cite{CCK(2014a)}]
\label{lemma: Gaussian approximation of linear statistics}
Under Conditions 3.3-3.4,
for any $\eta>0$,
there exists a random variable $\tilde{Z}\overset{d}{=}\|\tilde{\psi}^{p}(\cdot)^{\top}B^{-1}N_{(p)}\|_{\infty}$ 
such that
the inequality
\begin{align*}
	&\left| \left\| \tilde{\psi}^{p}(\cdot)^{\top} B^{-1}X^{\top}\varepsilon / \sqrt{n} \right\|_{\infty}
	- \tilde{Z} \right|
	\\
	& \leq
	\begin{cases}
	\tilde{c}_{1} \frac{n^{1/q} \log n}{\eta^{1/2}}
	\frac{ \xi_{p} }{n^{1/2}}
	+
	\frac{(\log n)^{3/4}}{\eta^{1/2}}
	\frac{ \xi_{p}^{1/2} }{n^{1/4} }
	+
	\frac{(\log n)^{2/3}}{\eta^{1/3}}
	\frac{\xi_{p}^{1/3}}{n^{1/6}}
		& \text{under Condition 2.5 (a)}, \\
	\tilde{c}_{1} \frac{\log n}{\eta^{1/2}}
	\frac{\xi_{p}}{n^{1/2}}
	+
	\frac{(\log n)^{3/4}}{\eta^{1/2}}
	\frac{ \xi_{p}^{1/2} }{n^{1/4} }
	+
	\frac{(\log n)^{2/3}}{\eta^{1/3}}
	\frac{\xi_{p}^{1/3}}{n^{1/6}}
		& \text{under Condition 2.5 (b)} 
	\end{cases}
\end{align*}
holds with probability at least $1 - \tilde{c}_{2} \{ \eta + (\log n) / n \}$
for some $\tilde{c}_{1},\tilde{c}_{2}>0$ not depending on $n$ and $p$.
\end{lemma}

\begin{lemma}
	\label{lemma: Expectation of suprema}
	Under Condition 3.4,
	we have
	$\Ep \|\tilde{\psi}^{p}(\cdot)^{\top}B^{-1}N_{(p)}\|_{\infty} \leq \tilde{c}_{1}\sqrt{\log p}$
	for some positive constant $\tilde{c}_{1}$ depending only on $C_{5}$ appearing in Condition 3.4.
\end{lemma}

\begin{proof}
From Dudley's entropy integral (e.g., see Corollary 2.2.8 in \cite{vanderVaartandWellner}), we have
\begin{align*}
	\Ep [ &\|\tilde{\psi}^{p}(\cdot)^{\top}B^{-1/2}N_{(p)} \|_{\infty} ] 
	\\
	&\leq
	\Ep [ | \tilde{\psi}^{p}(0)^{\top}B^{-1/2}N_{(p)} |  ]
	+
	\Ep [ \sup_{t \neq t' \in [0,1]} | \tilde{\psi}^{p}(t)^{\top}B^{-1/2}N_{(p)} - \tilde{\psi}^{p}(t')^{\top}B^{-1/2}N_{(p)}  | ]
	\\
	&\leq \underline{b} +  \int_{0}^{\theta} \sqrt{\log N([0,1], d_{X}, \delta )} d \delta,
\end{align*}
where
$N([0,1], d_{X}, \delta)$ is a $\delta$-covering number of $[0,1]$ with respect to
\[d_{X}(t,t') := \{\Ep [ \tilde{\psi}^{p}(t)^{\top}B^{-1/2}N_{(p)} -  \tilde{\psi}^{p}(t')^{\top}B^{-1/2}N_{(p)} ]^{2}\}^{1/2}\]
and
$\theta:= \sup_{t\in [0,1]} d_{X}(t,0)$.
Since $\theta$ is bounded by $2\underline{b}$, we have
\[
	\int_{0}^{\theta} \sqrt{\log N([0,1], d_{X}, \delta )} d \delta
	\leq \int_{0}^{2\underline{b}} \sqrt{\log N([0,1] , d_{X}, \delta  )} d\delta.
\]
From the bound on covering numbers of functions Lipschitz in one parameter
(e.g., see Theorem 2.7.11 in \cite{vanderVaartandWellner}),
we have $N ( [0,1] , d_{X} , \delta ) \leq \left( \tilde{c}_{2} p^{C_{5}} / \delta \right)$
for some $\tilde{c}_{2}>0$.
This gives
\begin{align*}
\int_{0}^{2\underline{b}} \sqrt{\log N([0,1] , d_{X}, \delta  )} d\delta
\leq \sqrt{C_{5} \log p} + \int_{0}^{2\underline{b}}\sqrt{\log (\tilde{c}_{2}\underline{b} / \delta)}d\delta.
\end{align*}
Thus, we obtain the desired inequality.
\end{proof}

\begin{lemma}
\label{lemma: Anti-concentration of Gaussian process suprema}
Under Conditions 3.3-3.4,
there exists a positive constant $\tilde{c}_{1}$ not depending on $n$ and $p$ for which
we have
\begin{align*}
	\sup_{x\in \R}  \Pr\big{(} \big{|} \| \tilde{\psi}^{p}(\cdot)^{\top}B^{-1/2}N_{(p)} \|_{\infty}- x \big{|} \leq R \big{)} \leq \tilde{c}_{1} R \sqrt{\log p}, \ R>0.
\end{align*}
\end{lemma}
\begin{proof}
	From Theorem 2.1 in \cite{CCK(2014a)}, we have
	\begin{align*}
		\sup_{x\in \R}  \Pr\big{(} \big{|} \| \tilde{\psi}^{p}(\cdot)^{\top}B^{-1/2}N_{(p)} \|_{\infty}- x \big{|} \leq R \big{)} 
		\leq \tilde{c}_{1} R \Ep [ \|\tilde{\psi}^{p}(\cdot)^{\top} B^{-1/2}N_{(p)} \|_{\infty} ]
	\end{align*}
	and
	thus from Lemma \ref{lemma: Expectation of suprema}, we complete the proof.
\end{proof}

\subsubsection{Proof of Proposition 3.5 }

We only prove the theorem under Condition 2.5 (a).
Although the proof is not a direct consequence of Theorem 2.1,
we can follow the same line as the proof of Theorem 2.1.

\noindent\textit{Step 1: Modification of the test set}

\noindent We start with modifying the test set $H$ that covers the randomness of the design.
Take $c_{1}>0$ sufficiently large.
Modify the test set
\[H = \{Y: \|X(\hat{\beta}(Y) - \beta_{0}) \| \leq c_{1}\sqrt{p \log n} \} \cap \{Y:  \Pi_{\sigma^{2}}(|\sigma^{2}/\sigma_{0}^{2}-1 | \geq \delta_{1} \mid Y) \leq \delta_{2} \}\]
in Proposition 2.5
as 
\begin{align*}
	H:=\{(X,Y) :& \|X(\hat{\beta}(Y) - \beta_{0} )\| \leq c_{1}\sqrt{p\log n} , (\underline{b}/2)^{2}I_{p} \preceq  X^{\top}X / n \preceq ( 2 \overline{b})^{2} I_{p} \}
\\
&\cap \{(X,Y) : \Pi_{\sigma^{2}}(|\sigma^{2}/\sigma_{0}^{2}-1| \geq \delta_{1}\mid Y)\leq \delta_{2} \}.
\end{align*}
We bound $\Pr ((X,Y)\not\in H)$ as follows:
\begin{align*}
	\Pr ( (X,Y) \not\in H)
	\leq &
	A_{1} + A_{2} + A_{3} + \delta_{3},
\end{align*}
where 
\begin{align*}
	A_{1}  & := \Pr (\|X(X^{\top}X)^{-1}X^{\top}\varepsilon \| \geq {c}_{1} \sqrt{p \log n} /2 , \  (\underline{b}/2)^2 I_{p} \preceq X^{\top}X / n \preceq ( 2\overline{b})^{2} I_{p} ),\\
	A_{2} &  := \Pr (\| X(X^{\top}X)^{-1}X^{\top}r \| \geq c_{1} \sqrt{p \log n} /2 ), \\
	A_{3} &  := \Pr ( X \not\in \{X: (\underline{b}/2)^2 I_{p} \preceq (X^{\top}X) / n  \preceq (2\overline{b})^{2} I_{p} \} ).
\end{align*}
Lemma \ref{lemma: exponential inequality for quadratic forms} gives
$A_{1} \leq \tilde{c}_{1} n^{-\tilde{c}_{2} p}$ for some $\tilde{c}_{1},\tilde{c}_{2}>0$.
Markov's inequality gives
\begin{align*}
	A_{2} &\leq  \frac{\Ep[r^{\top}X(X^{\top}X)^{-1}X^{\top}r] }{p\log n}  \leq \frac{n}{\log n}\frac{\tau^{2}_{2}}{p}.
\end{align*}
Lemma \ref{lemma: Matrix Chernoff} gives
$A_{3} \leq \tilde{c}_{1} n^{-\tilde{c}_{2} p}$.

\noindent\textit{Step 2: Upper bound on the coverage error}

\noindent We start with proving that $\hat{R}_{\alpha}$ concentrates on the $(1-\alpha)$-quantile of some distribution with high probability.
Let $\overline{\zeta}:= \tilde{\phi}_{\Pi_{\beta}} (c_{1} \sqrt{p\log n}) + c_{1}\delta_{1}p\log n + \delta_{2} + \delta_{3}
+c_{1}n^{-c_{2} p}$
with the constant $c_{2}$ in Proposition 2.5.
From Proposition 2.5, we have
\begin{align*}
	\big{|}& \Pi_{\beta}\{ \|\tilde{\psi}^{p} (\cdot)^{\top} (\hat{\beta} - \beta_{0}) \|_{\infty} \leq \hat{R}_{\alpha} \mid Y , X\}
	- \Pr ( \| \tilde{\psi}^{p}(\cdot)^{\top} (X^{\top}X)^{-1}X^{\top}N_{(n)} \|_{\infty} \leq \hat{R}_{\alpha} \mid Y , X ) \big{|}
	\\
	&\leq 
	\overline{\zeta} \text{ for }(X,Y)\in H.
\end{align*}
Letting $G$ be the distribution function of 
$\| \tilde{\psi}^{p}(\cdot)^{\top} (X^{\top}X)^{-1}X^{\top} N_{(n)} \|_{\infty}$
and letting $G^{-1}$ be its quantile function,
we have
\[\hat{R}_{\alpha}\leq G^{-1}(1-\alpha + \overline{\zeta}) \text{ for $(X,Y) \in H$}.\]

Next we bound the Kolmogorov distances between 
$\| \tilde{\psi}^{p} (\cdot)^{\top} (\hat{\beta}-\beta_{0}) \|_{\infty} $
and $\sqrt{n} \| \tilde{\psi}^{p} (\cdot)^{\top} \allowbreak B^{-1/2}N_{(p)} \|_{\infty}$;
between
$\| \tilde{\psi}^{p} (\cdot)^{\top} (X^{\top}X)^{-1}X^{\top} N_{(n)} \|_{\infty}$
and
$\sqrt{n} \| \tilde{\psi}^{p} (\cdot)^{\top} B^{-1/2}N_{(p)} \|_{\infty}$:
\begin{align*}
	\rho_{1}:=& \sup_{R>0} \big{|} \Pr( \|\tilde{\psi}^{p}(\cdot)^{\top} \sqrt{n} (\hat{\beta}-\beta_{0}) \|_{\infty} \leq R) 
	- \Pr( \| \tilde{\psi}^{p}(\cdot)^{\top} B^{-1/2}N_{(p)} \|_{\infty} \leq R) \big{|}, \\
\rho_{2}:=&
	\sup_{R>0} \big{|} \Pr( \| \tilde{\psi}^{p}(\cdot)^{\top} \sqrt{n} (X^{\top}X)^{-1}X^{\top}N \|_{\infty} \leq R) 
	- \Pr( \|\tilde{\psi}^{p}(\cdot)^{\top} B^{-1/2}N_{(p)} \|_{\infty} \leq R) \big{|}.
\end{align*}
We also bound the L\'{e}vy concentration function of $\sqrt{n} \| \tilde{\psi}^{p} (\cdot)^{\top} B^{-1/2}N_{(p)} \|_{\infty}$:
\begin{align*}
\gamma(R):=&\sup_{x > 0} \Pr ( | \| \tilde{\psi}^{p}(\cdot)^{\top}B^{-1/2}N_{(p)} \|_{\infty} -x | \leq R ).
\end{align*}
Let $\eta=\eta_{n}$ be an arbitrary divergent sequence.
We present useful inequalities for bounding $\rho_{1}$, $\rho_{2}$, and $\gamma(R)$ ahead.
Let
\begin{align*}
D_{1}&:=\sqrt{n} \bigg{|} \left\|\tilde{\psi}^{p}(\cdot)^{\top} (\hat{\beta} -\beta_{0} ) \right\|_{\infty} - \left\| \tilde{\psi}^{p}(\cdot)^{\top} B^{-1}X^{\top}\varepsilon / n \right\|_{\infty} \bigg{|}, \\
D_{2}&:=\sqrt{n} \bigg{|} \|\tilde{\psi}^{p}(\cdot)^{\top} (X^{\top}X)^{-1}X^{\top}N_{(n)} \|_{\infty} - \|\tilde{\psi}^{p}(\cdot)^{\top} B^{-1}X^{\top} N_{(n)} / n \|_{\infty} \bigg{|},\\
D_{3}&:=\sqrt{n} \bigg{|}
	\left\| \tilde{\psi}^{p}(\cdot)^{\top}B^{-1}X^{\top}\varepsilon / n \right\|_{\infty} 
	- \tilde{Z}  \bigg{|},\\
D_{4}&:=\sqrt{n} \bigg{|} 
	\left\| \tilde{\psi}^{p}(\cdot)^{\top}B^{-1}X^{\top}N_{(n)} / n \right\|_{\infty} 
	- \tilde{Z}  \bigg{|}.
\end{align*}
Then we have, for some $\tilde{c}_{3},\tilde{c}_{4}>0$ independent of $n$ and $p$,
\begin{align}
	&\Pr \bigg{(} D_{1}
	\geq \tilde{c}_{3} \eta \left\{ \left( \xi_{p}^{2} / n \right)^{1/2}
	\sqrt{\log p } (n^{1/q}\sqrt{\log p} + \sqrt{p}\tau_{\infty} )
	+ \sqrt{\log p} \tau_{\infty} \right\}
	\bigg{)}
	\leq  \tilde{c}_{4} / \eta^2,
	\label{ineq 1 with general basis functions}\\
	&\Pr  \bigg{(} D_{2}
	\geq
	\tilde{c}_{3} \eta \left( \xi_{p}^{2} / n \right)^{1/2}
	 n^{1/q}\log p
	\bigg{)}
	\leq  \tilde{c}_{4} / \eta^2,
	\label{ineq 2 with general basis functions}\\
	&\Pr \bigg{(} D_{3}
	\geq
	\tilde{c}_{3} \eta
	\bigg{\{}\bigg{(} \frac{\xi_{p}^{2}}{n}\bigg{)}^{1/2} (n^{1/q} \log n)
	       +
	       \bigg{(} \frac{\xi_{p}^{2}}{n} \bigg{)}^{1/4} (\log n)^{3/4}
	\bigg{\}}
	 +
	\tilde{c}_{3} \eta^{2/3}
		\left( \frac{\xi_{p}^{2}}{n} \right)^{1/6}  (\log n)^{2/3}
	\bigg{)}
	\nonumber\\
	&\leq \tilde{c}_{4} \left(\frac{1}{\eta^{2}} + \frac{\log n}{n}\right),
	\label{ineq 3 with general basis functions}
\end{align}
and
\begin{align}
&\Pr \bigg{(} D_{4} 
	\geq 
	\tilde{c}_{3} \eta
	\bigg{\{}\bigg{(} \frac{\xi_{p}^{2}}{n} \bigg{)}^{1/2} (n^{1/q} \log n)
	       +
	       \bigg{(} \frac{\xi_{p}^{2}}{n} \bigg{)}^{1/4} (\log n)^{3/4}
	\bigg{\}}
        +
	\tilde{c}_{3} \eta^{2/3}     
		\left( \frac{\xi_{p}^{2}}{n} \right)^{1/6}  (\log n)^{2/3}
	\bigg{)}
	\nonumber\\
	&\leq \tilde{c}_{4} \left(\frac{1}{\eta^{2}} + \frac{\log n}{n}\right),
	\label{ineq 4 with general basis functions}
\end{align}
where the first two inequalities follows from Lemma \ref{lemma: Uniform linearization}
and the last two inequalities follows from Lemma \ref{lemma: Gaussian approximation of linear statistics}.
From inequalities (\ref{ineq 1 with general basis functions}) and (\ref{ineq 3 with general basis functions})
and
from Lemma \ref{lemma: Anti-concentration of Gaussian process suprema},
we have 
\begin{align}
	\rho_{1} \leq & 
	\tilde{c}_{5} (A_{4} + A_{5}),
	\label{Berry--Esseen type 1 with general basis functions}
\end{align}
for some $\tilde{c}_{5}>0$,
where 
\begin{align*}
	A_{4} &:= \frac{1}{\eta^2} + \frac{\log n}{n}
	 + \eta (\log p)^{1/2}
	 \max\left\{ \left( \xi_{p}^{2} / n \right)^{1/2}n^{1/q}\log n ,
	 \left( \xi_{p}^{2} / n \right)^{1/6} (\log n)^{2/3} \right\} \text{ and }
	\nonumber\\
	A_{5} &:= \eta (\log p) \tau_{\infty}
	\max\left\{1 , \left( p \xi_{p}^{2} / n \right)^{1/2} \right\}.
\end{align*}
Likewise, 
from inequalities (\ref{ineq 2 with general basis functions}) and (\ref{ineq 4 with general basis functions}),
and from Lemma \ref{lemma: Anti-concentration of Gaussian process suprema},
we have 
\begin{align*}\rho_{2} \leq \tilde{c}_{5} A_{4}\end{align*}
for some $\tilde{c}_{5}>0$.
From Lemma \ref{lemma: Anti-concentration of Gaussian process suprema},
we have, for some $\tilde{c}_{5}>0$,
\begin{align}
	\gamma(R) \leq \tilde{c}_{5} R \sqrt{\log p}.
	\label{Anti-concentration bound with general basis functions}
\end{align}

Finally,
we have
\begin{align*}
	&\Pr(f_{0} \in \mathcal{C}(\hat{f},\hat{R}_{\alpha}) ) - (1-\alpha)\\
	&\leq \Pr \{\| \tilde{\psi}^{p}(\cdot)^{\top}(\hat{\beta}-\beta_{0}) \|_{\infty} 
\leq  G^{-1}(1-\alpha+ \overline{\zeta} ) + \tau \} - (1-\alpha) + \Pr\{ (X,Y) \not\in H\}
	\nonumber\\
	&\leq \Pr \{ \| \tilde{\psi}^{p}(\cdot)^{\top} B^{-1/2}N_{(p)} /\sqrt{n} \|_{\infty} 
\leq G^{-1}(1-\alpha+\overline{\zeta}) + \tau \} - (1-\alpha) + \rho_{1} + \Pr\{ (X,Y) \not\in H\}
	\nonumber\\
	&\leq
	\Pr \{ \| \tilde{\psi}^{p}(\cdot)^{\top}B^{-1/2}N_{(p)} / \sqrt{n} \|_{\infty} 
	\leq G^{-1}(1-\alpha+\overline{\zeta} ) \} - (1-\alpha)
	\nonumber\\
	&\qquad + \gamma(\sqrt{n} \tau) + \rho_{1} + \Pr\{ (X,Y) \not\in H\}
	\nonumber\\
	&\leq
	\overline{\zeta} + \rho_{1}+\rho_{2} + \gamma(\sqrt{n} \tau)+ \Pr( (X,Y) \not\in H),
\end{align*}
and thus
from (\ref{Berry--Esseen type 1 with general basis functions})-(\ref{Anti-concentration bound with general basis functions}),
taking $\eta = n^{\delta}$,
we obtain the desired upper bound of $\Pr(f_{0} \in \mathcal{C}(\hat{f},\hat{R}_{\alpha}) ) - (1-\alpha)$.
Likewise, we obtain the desired lower bound of $\Pr(f_{0} \in \mathcal{C}(\hat{f},\hat{R}_{\alpha})) - (1-\alpha)$,
which completes Step 2.

\noindent\textit{Step 3: Upper bound on the $L^{\infty}$-diameter}

\noindent We will show that $G^{-1}( 1 - \alpha + \overline{\zeta}) \leq \tilde{c}_{6} \sqrt{ (\log p) / n}$ for some $\tilde{c}_{6}>0$.
From the concentration inequality for the suprema of the Gaussian process,
and
from Lemma \ref{lemma: Expectation of suprema},
we have, for sufficiently large $\tilde{c}_{7}>0$ depending only on $\alpha$ and $\underline{b}$,
\[
\Pr( \| \tilde{\psi}^{p}(\cdot)^{\top}B^{-1/2}N_{(p)} \|_{\infty} - \Ep[ \| \tilde{\psi}^{p}(\cdot)^{\top}B^{-1/2}N_{(p)} \|_{\infty}] \geq \tilde{c}_{7} \sqrt{\log p} )
\leq \alpha - \overline{\zeta} - \rho_{2}.
\]
Observing
\begin{align*}
	G^{-1}( 1 - \alpha + \overline{\zeta})
	&:=\inf\{ R : \Pr( \| \tilde{\psi}^{p}(\cdot)^{\top}(X^{\top}X)^{-1}X^{\top}N_{(n)} \|_{\infty} \geq R ) \leq \alpha - \overline{\zeta} \}
	\nonumber\\
	&\leq
	\inf\{ R : \Pr( \| \tilde{\psi}^{p}(\cdot)^{\top}B^{-1/2}N_{(p)} / \sqrt{n} \|_{\infty} \geq R ) \leq \alpha - \overline{\zeta} - \rho_{2}\}
	\nonumber\\
	&=
	\inf\bigg{\{} R : 
		\Pr\bigg{(} \| \tilde{\psi}^{p}(\cdot)^{\top}B^{-1/2}N_{(p)} \|_{\infty} / \sqrt{n} 
		- \Ep[ \| \tilde{\psi}^{p}(\cdot)^{\top}B^{-1/2}N_{(p)} /\sqrt{n} \|_{\infty}] 
	\nonumber\\
	& \qquad\qquad\qquad \geq R - \Ep[\| \tilde{\psi}^{p}(\cdot)^{\top}B^{-1/2}N_{(p)} / \sqrt{n} \|_{\infty} ]  \bigg{)}
	\leq \alpha - \overline{\zeta} -\rho_{2} \bigg{\}},
\end{align*}
we have $G^{-1} (1-\alpha +\overline{\zeta})  \lesssim \sqrt{(\log p) /n}$
and thus we complete the proof.
\qed

\subsubsection{Proof of the bound on $\tau$}

We will show that $\tau \lesssim \tau_{\infty}/\sqrt{p}$ for periodic $S\geq 2$-regular wavelets.
Consider a wavelet pair $(\phi,\psi)$ satisfying the following three assumptions:
\begin{itemize}
\item There exists an integer $N$ for which the support of $\phi$ is included in $[0,N]$ and the support of $\psi$ is included in $[-N+1,N]$;
\item $\phi$ and $\psi$ are $C^{S}[0,1]$;
\item The inequality $\inf_{x \in \R} \sum_{k \in \mathbb{Z}} \{\psi(x-k)\}^{2} > 0$ holds.
\end{itemize}
We periodize the pair $(\phi,\psi)$ as follows:
\begin{align*}
\phi^{(\mathrm{per})}_{l,k} (t) := \sum_{m\in\mathbb{Z}} 2^{l/2} \phi ( 2^{l} t + 2^{l} m - k) \ \text{ and }
\psi^{(\mathrm{per})}_{l,k} (t) := \sum_{m\in\mathbb{Z}} 2^{l/2} \psi ( 2^{l} t + 2^{l} m - k)
\end{align*}
for $k=0,\ldots, 2^{l}-1$ and $l=1,\ldots,J$.
With $J_{0}$ such that $2^{J_{0}}\geq 2N$,
$\{ \phi^{(\mathrm{per})}_{J_{0}, k} : k=0, \ldots, 2^{J_{0}}-1 \} \cup \{ \psi^{(\mathrm{per})}_{l,k} : k=0 , \ldots , 2^{l}-1 , l=J_{0}, \ldots, J\}$
forms $p=2^{J}$ basis functions based on periodic $S$-regular wavelets.

It suffices to show that $\inf_{t\in [0,1]} \| \psi^{p} (t) \| \gtrsim \sqrt{p}$.
Since $2^{J}> 2N$ and since the support of $\psi$ is included in $[-N+1 , N ]$,
we have
\begin{align*}
	\|\psi^{p}(t)\|^2 \geq 2^{J} \sum_{k=0}^{2^{J}-1} \left\{ \sum_{m\in \mathbb{Z}} \psi(2^{J}t + 2^{J}m - k) \right\}^{2}
	= 2^{J} \sum_{k=0}^{2^{J}-1} \sum_{m\in\mathbb{Z}} \{\psi(2^{J}t + 2^{J}m - k )\}^{2}
\end{align*}
and 
\begin{align*}
2^{J} \sum_{k=0}^{2^{J}-1} \sum_{m\in\mathbb{Z}} \{\psi(2^{J}t + 2^{J}m - k )\}^{2}
\geq
2^{J} \inf_{ x  \in  \R}  \sum_{k  \in  \mathbb{Z}}  \{  \psi  (  x - k )  \}^{2}.
\end{align*}
Thus we complete the proof.

\end{document}